\documentclass[12pt,reqno]{amsart}
\usepackage{graphicx,indentfirst}
\usepackage{amsmath,amssymb,mathrsfs}
\usepackage{amsthm,amscd}
\usepackage{verbatim}
\usepackage{appendix}
\usepackage{enumitem,titletoc}
\usepackage{mathtools}
\usepackage[utf8]{inputenc}
\usepackage{imakeidx}
\makeindex[columns=2, title=Alphabetical Index]
\usepackage{fancyhdr}
\usepackage{amsfonts,color}
\usepackage[all]{xy}
\usepackage{tikz-cd}
\usepackage{syntonly}
\usepackage{fancyhdr}
\usepackage{array}
\usepackage[left=2.3cm,right=2.3cm,bottom=3cm,top=3cm]{geometry}
\usepackage{microtype}

\usepackage{tikz}
\usepackage{extarrows}
\usepackage{hyperref}
\usepackage{setspace}
\allowdisplaybreaks[2]

\def\XXint#1#2#3{{\setbox0=\hbox{$#1{#2#3}{\int}$ }
		\vcenter{\hbox{$#2#3$ }}\kern-.6\wd0}}

\newtheorem{thm}{Theorem}[section]
\newtheorem{prop}[thm]{Proposition}

\newtheorem{defn}[thm]{Definition}
\newtheorem{lem}[thm]{Lemma}
\newtheorem{cor}[thm]{Corollary}
\newtheorem{conj}[thm]{Conjecture}
\newtheorem{rk}{Remark}

\newtheorem{assumption}{Assumption}[section]

\allowdisplaybreaks
\newcommand{\del}{\partial}

\newcommand{\eps}{\varepsilon}

\newcommand{\C}{\mathbb C}
\newcommand{\R}{\mathbb R}

\newcommand{\cO}{\mathcal{O}}

\newcommand{\bT}{\mathbb{T}}
\newcommand{\bR}{\mathbb{R}}
\newcommand{\vphi}{\varphi}
\newcommand{\inte}{\operatorname{int}}
\newcommand{\diag}{\operatorname{diag}}
\newcommand{\Hess}{\operatorname{Hess}}

\makeindex

\DeclareMathOperator{\tr}{tr}

\definecolor{citationblue}{HTML}{008000}
\hypersetup{
	colorlinks=true,
	citecolor=citationblue,
	filecolor=green,
	linkcolor=blue,
	urlcolor=black
}

\numberwithin{equation}{section}

\author{Tristan C. Collins}

\author{Genggeng Huang}
\email{\href{mailto:genggenghuang@fudan.edu.cn}{genggenghuang@fudan.edu.cn}}
\address{School of Mathematical Sciences, Fudan University, Shanghai, China}

\author{Freid Tong}

\author{Yulun Xu}
\email{\href{mailto:tristanc@math.toronto.edu}{tristanc@math.toronto.edu}}
\email{\href{mailto:freid.tong@utoronto.ca}{freid.tong@utoronto.ca}}
\email{\href{mailto:yulun.xu@utoronto.ca}{yulun.xu@utoronto.ca}}
\address{Department of Mathematics, University of Toronto, 40 St. George Street, Toronto, ON, Canada}
\date{\today}
\newcommand{\nc}{\newcommand}
\nc{\p}{\partial}
\nc{\pb}{\partial_b}
\nc{\pc}{\partial_c}
\nc{\pd}{\partial_d}
\nc{\pe}{\partial_e}
\nc{\pf}{\partial_f}
\nc{\pg}{\partial_g}
\nc{\ph}{\partial_h}
\nc{\pari}{\partial_i}
\nc{\pj}{\partial_j}
\nc{\pk}{\partial_k}
\nc{\pl}{\partial_l}
\nc{\pell}{\partial_\ell}
\nc{\parm}{\partial_m}
\nc{\pn}{\partial_n}
\nc{\po}{\partial_o}
\nc{\pp}{\partial_p}
\nc{\pq}{\partial_q}
\nc{\pr}{\partial_r}
\nc{\ps}{\partial_s}
\nc{\pt}{\partial_t}
\nc{\pu}{\partial_u}
\nc{\pv}{\partial_v}
\nc{\pw}{\partial_w}
\nc{\px}{\partial_x}
\nc{\py}{\partial_y}
\nc{\pz}{\partial_z}

\nc{\pabar}{\partial_{\ol{a}}}
\nc{\pbbar}{\partial_{\ol{b}}}
\nc{\pcbar}{\partial_{\ol{c}}}
\nc{\pdbar}{\partial_{\ol{d}}}
\nc{\pebar}{\partial_{\ol{e}}}
\nc{\pfbar}{\partial_{\ol{f}}}
\nc{\pgbar}{\partial_{\ol{g}}}
\nc{\phbar}{\partial_{\ol{h}}}
\nc{\pibar}{\partial_{\ol{i}}}
\nc{\pjbar}{\partial_{\ol{j}}}
\nc{\pkbar}{\partial_{\ol{k}}}
\nc{\plbar}{\partial_{\ol{l}}}
\nc{\pellbar}{\partial_{\ol{\ell}}}
\nc{\pmbar}{\partial_{\ol{m}}}
\nc{\pnbar}{\partial_{\ol{n}}}
\nc{\pobar}{\partial_{\ol{o}}}
\nc{\ppbar}{\partial_{\ol{p}}}
\nc{\pqbar}{\partial_{\ol{q}}}
\nc{\prbar}{\partial_{\ol{r}}}
\nc{\psbar}{\partial_{\ol{s}}}
\nc{\ptbar}{\partial_{\ol{t}}}
\nc{\pubar}{\partial_{\ol{u}}}
\nc{\pvbar}{\partial_{\ol{v}}}
\nc{\pwbar}{\partial_{\ol{w}}}
\nc{\pxbar}{\partial_{\ol{x}}}
\nc{\pybar}{\partial_{\ol{y}}}
\nc{\pzbar}{\partial_{\ol{z}}}

\nc{\nababar}{\nabla_{\ol{a}}}
\nc{\nabbbar}{\nabla_{\ol{b}}}
\nc{\nabcbar}{\nabla_{\ol{c}}}
\nc{\nabdbar}{\nabla_{\ol{d}}}
\nc{\nabebar}{\nabla_{\ol{e}}}
\nc{\nabfbar}{\nabla_{\ol{f}}}
\nc{\nabgbar}{\nabla_{\ol{g}}}
\nc{\nabhbar}{\nabla_{\ol{h}}}
\nc{\nabibar}{\nabla_{\ol{i}}}
\nc{\nabjbar}{\nabla_{\ol{j}}}
\nc{\nabkbar}{\nabla_{\ol{k}}}
\nc{\nablbar}{\nabla_{\ol{l}}}
\nc{\nabelllbar}{\nabla_{\ol{\ell}}}
\nc{\nabmbar}{\nabla_{\ol{m}}}
\nc{\nabnbar}{\nabla_{\ol{n}}}
\nc{\nabobar}{\nabla_{\ol{o}}}
\nc{\nabpbar}{\nabla_{\ol{p}}}
\nc{\nabqbar}{\nabla_{\ol{q}}}
\nc{\nabrbar}{\nabla_{\ol{r}}}
\nc{\nabsbar}{\nabla_{\ol{s}}}
\nc{\nabtbar}{\nabla_{\ol{t}}}
\nc{\nabubar}{\nabla_{\ol{u}}}
\nc{\nabvbar}{\nabla_{\ol{v}}}
\nc{\nabwbar}{\nabla_{\ol{w}}}
\nc{\nabxbar}{\nabla_{\ol{x}}}
\nc{\nabybar}{\nabla_{\ol{y}}}
\nc{\nabzbar}{\nabla_{\ol{z}}}

\nc{\naba}{\nabla_{a}}
\nc{\nabb}{\nabla_{b}}
\nc{\nabc}{\nabla_{c}}
\nc{\nabd}{\nabla_{d}}
\nc{\nabe}{\nabla_{e}}
\nc{\nabf}{\nabla_{f}}
\nc{\nabg}{\nabla_{g}}
\nc{\nabh}{\nabla_{h}}
\nc{\nabi}{\nabla_{i}}
\nc{\nabj}{\nabla_{j}}
\nc{\nabk}{\nabla_{k}}
\nc{\nabl}{\nabla_{l}}
\nc{\nabm}{\nabla_{m}}
\nc{\nabn}{\nabla_{n}}
\nc{\nabo}{\nabla_{o}}
\nc{\nabp}{\nabla_{p}}
\nc{\nabq}{\nabla_{q}}
\nc{\nabr}{\nabla_{r}}
\nc{\nabs}{\nabla_{s}}
\nc{\nabt}{\nabla_{t}}
\nc{\nabu}{\nabla_{u}}
\nc{\nabv}{\nabla_{v}}
\nc{\nabw}{\nabla_{w}}
\nc{\nabx}{\nabla_{x}}
\nc{\naby}{\nabla_{y}}
\nc{\nabz}{\nabla_{z}}

\nc{\ola}{\ol{a}}
\nc{\olb}{\ol{b}}
\nc{\olc}{\ol{c}}
\nc{\old}{\ol{d}}
\nc{\ole}{\ol{e}}
\nc{\olf}{\ol{f}}
\nc{\olg}{\ol{g}}
\nc{\olh}{\ol{h}}
\nc{\oli}{\ol{i}}
\nc{\olj}{\ol{j}}
\nc{\olk}{\ol{k}}
\nc{\oll}{\ol{l}}
\nc{\olm}{\ol{m}}
\nc{\oln}{\ol{n}}
\nc{\olo}{\ol{o}}
\nc{\olp}{\ol{p}}
\nc{\olq}{\ol{q}}
\nc{\olr}{\ol{r}}
\nc{\ols}{\ol{s}}
\nc{\olt}{\ol{t}}
\nc{\olu}{\ol{u}}
\nc{\olv}{\ol{v}}
\nc{\olw}{\ol{w}}
\nc{\olx}{\ol{x}}
\nc{\oly}{\ol{y}}
\nc{\olz}{\ol{z}}

\title[]{Shrinking K\"ahler-Ricci solitons on toric fano fibrations}

\begin{document}
\begin{abstract}
We prove that a smooth toric Fano fibration admits a complete gradient shrinking K\"ahler-Ricci soliton.
\end{abstract}
\maketitle
\tableofcontents
\section{Introduction}

A complete gradient K\"ahler--Ricci shrinker is a quadruple
$(X,J,\omega,f)$ consisting of a complete K\"ahler manifold $(X,J,\omega)$ and a real-valued
function $f$ such that $J\nabla f$ is a real holomorphic Killing field and,
after normalization,
\begin{equation}
  Ric(\omega)+\sqrt{-1}\,\partial\bar\partial f=\omega.
  \label{eq:KRS}
\end{equation}
Equivalently, in a commonly used Riemannian normalization,
\begin{equation}
  Ric(g)+\nabla^2 f=\frac12 g,
  \qquad \nabla^{1,0}f\ \text{holomorphic}.
\end{equation}
If $f$ is constant, \eqref{eq:KRS} is the positive
K\"ahler--Einstein equation.  Generally, Ricci shrinkers are self-similar solutions of the Ricci flow and hence play a fundamental role in the singularity analysis of the Ricci flow, a viewpoint that goes back to Hamilton's original work on singularity formation \cite{Hamilton1995} and  Perelman's subsequent work \cite{Perelman2002}.  Indeed, by work of Naber \cite{Naber} and Enders-M\"uller-Topping \cite{EMT2011}, it is known that a parabolic blow-up of the Ricci flow at a Type~I singular point converges subsequentially to a nontrivial gradient shrinking soliton \cite{EMT2011}. Thus K\"ahler--Ricci shrinkers
serve simultaneously as canonical metrics and as local models for Type~I
singularities of the (K\"ahler)--Ricci flow.

Let us give a brief summary of the known constructions of shrinking K\"ahler-Ricci solitons. If $X$ is compact, taking cohomology in \eqref{eq:KRS} gives
$[\omega]=2\pi c_1(X)$ up to normalization; hence $X$ is Fano. When $f=0$ in \eqref{eq:KRS}, the K\"ahler-Ricci soliton equation reduces to the K\"ahler-Einstein equation.  The existence of K\"ahler-Einstein metrics on compact Fano manifolds is equivalent to the algebro-geometric notion of $K$-stability, thanks to Chen-Donaldson-Sun's resolution of the Yau-Tian-Donaldson conjecture \cite{ChenDonaldsonSun2015I,ChenDonaldsonSun2015II,ChenDonaldsonSun2015III}. K\"ahler--Ricci shrinkers therefore extend positive K\"ahler--Einstein metrics precisely in the setting where
holomorphic vector fields may obstruct the Einstein equation.  Tian and Zhu \cite{TianZhu2000,TianZhu02}
introduced the relevant modified Futaki invariant, identified a
distinguished soliton vector field, and proved uniqueness modulo holomorphic
automorphisms \cite{TianZhu2000,TianZhu02}.   Wang-Zhu \cite{WangZhu2004} (see also Berman-Berndtsson \cite{BermanBerndtsson13}) established that every toric Fano manifold admits a (possibly trivial) K\"ahler-Ricci soliton.  Datar-Sz\'ekelyhidi \cite{DatarSzekelyhidi2016}, in combination with work of Berman-Witt-Nystr\"om \cite{ BermanWittNystrom2014} extended the resolution of the Yau-Tian-Donaldson conjecture to solitons, showing that a Fano variety admits a K\"ahler-Ricci soliton if and only if it satisfies a version of $K$-stability relative to the soliton vector field; see also the recent work of Boucksom-Jonsson \cite{BJ} for a non-Archimedean approach.   

Non-Einstein examples of non-compact, shrinking, gradient K\"ahler-Ricci solitons were constructed by Koiso \cite{Koiso1990} and Cao \cite{Cao1996, Cao1997}, and Feldman-Ilmanen-Knopf \cite{FIK2003}.   For non-compact K\"ahler manifolds, a systematic classification of gradient shrinking (and expanding) K\"ahler-Ricci solitons with quadratic curvature decay was initiated by Conlon-Deruelle-Sun \cite{CDS2024}.  In complex dimension $2$, a complete classification of gradient, shrinking K\"ahler-Ricci solitons is now available thanks to the combined works of Conlon-Deruelle-Sun \cite{CDS2024},  Cifarelli--Conlon--Deruelle \cite{CCDTypeI2024}, Bamler--Cifarelli--Conlon--Deruelle \cite{BCCD2024} and Li-Wang \cite{LiWang2026}.  The complete list of non-trivial gradient, shrinking K\"ahler-Ricci solitons, up to biholomorphism, is as follows:
\begin{enumerate}
  \item a compact Fano surface with its K\"ahler--Ricci shrinker;
  \item the Gaussian shrinker on $\C^2$;
  \item the FIK shrinker on $Bl_0\C^2$;
  \item the standard product shrinker on $\C P^1\times\C$; and
  \item the Bamler-Cifarelli-Conlon-Deruelle (BCCD) shrinker on $Bl_p(\C P^1\times\C)$.
\end{enumerate}

It is interesting to note that, in contrast to the other non-compact examples in the above list, the BCCD soliton is constructed indirectly, making use of Bamler's structure theory for the Ricci flow \cite{Bam20a, Bam20b, Bam20c}.  In addition, note that all the non-compact examples in the above list are toric. 

Recently, Sun-Zhang \cite{SZ} proposed a far-reaching generalization of the Yau-Tian-Donaldson conjecture concerning the existence of K\"ahler-Ricci solitons on polarized Fano fibrations (see Definition~\ref{defn: polarizedFanoFibration}).
\begin{conj}[Sun-Zhang Conjecture]\label{conj sun zhang}
A polarized Fano fibration $(\pi\colon X\longrightarrow Y,\xi)$ admits a complete gradient, shrinking K\"ahler--Ricci soliton unique up to the action of \(\operatorname{Aut}(X,\xi)\), if and only if
the polarized Fano fibration is \(K\)-polystable.  
\end{conj}

Recent work of Hallgren-Zhang \cite{HallgrenZhang2026} shows that polarized Fano fibrations with K\"ahler-Ricci solitons arise naturally as singularity models of the K\"ahler-Ricci flow.  To motivate Conjecture~\ref{conj sun zhang} it is helpful to consider the two extreme cases; namely when $Y$ is a point, and when $X=Y$.  When $Y$ is a point the conjecture reduces to the claim that a compact Fano variety $X$ admits a K\"ahler-Ricci soliton if and only if $X$ is modified $K$-polystable, which was established in \cite{DatarSzekelyhidi2016, BermanWittNystrom2014, BJ}, as discussed above.  When $X=Y$, the existence of a K\"ahler-Ricci soliton is implied by the existence of a conical Calabi-Yau metric on $Y$ (or equivalently a Sasaki-Einstein metric on the link of $Y$), which was established by Collins-Sz\'ekelyhidi \cite{CoSz, CoSz1} building on the work of Chen-Donaldson-Sun \cite{ChenDonaldsonSun2015I,ChenDonaldsonSun2015II,ChenDonaldsonSun2015III}, Sz\'ekelyhidi \cite{Szekelyhidi}, and Datar-Szekelyhidi \cite{DatarSzekelyhidi2016}.  Sun-Zhang \cite[Section 6]{SZ} note that the notion of $K$-stability relevant to existence of Calabi-Yau cone metrics is equivalent to the notion of $K$-polystability of the assocoaited Fano fibrations $(\pi:Y\rightarrow Y, \xi)$, and hence conjecture that all K\"ahler-Ricci solitons in this setting arise from Calabi-Yau cone metrics. Conjecture~\ref{conj sun zhang} has recently generated a great deal of interest. Beyond the  extreme cases discussed above there has been some progress on the Sun-Zhang conjecture, primarily in the setting of asymptotically conical solitons. In this setting it was proved by Esparza in \cite{Esparza2025} that K\"ahler-Ricci shrinkers are unique. Cifarelli-Esparza in \cite{CifarelliEsparza2025} showed that the existence of an asymptotically conical K\"ahler shrinker implies K-polystability.  Cifarelli-Conlon-Deruelle \cite{CCD2024Aubin} developed an Aubin continuity path towards constructing K\"ahler-Ricci solitons on certain toric Fano fibrations and proved existence at the initial time. The algebraic theory of $K$-stability for Fano fibrations was further developed by Odaka \cite{Odaka}.

Our main result is the following

\begin{thm}\label{thm main 2}
    Let $(\pi:X\rightarrow Y, \bT,\xi)$ be a toric Fano fibration with $\xi \in {\rm Lie}(\bT)$ the volume minimizing Reeb vector field.  Then $(\pi:X\rightarrow Y, \bT,\xi)$ admits a complete, shrinking K\"ahler-Ricci soliton.  Furthermore, we have the volume growth estimate
    \[
    {\rm Vol}(B_{R}(p)) \sim R^{\dim_{\bR}Y} \quad \text{ for } R \gg 1.
    \]
\end{thm}

In particular, our results give a direct PDE construction of the BCCD soliton. In fact, our results are slightly more general and imply existence of K\"ahler-Ricci solitons on certain toric Fano fibrations with mild singularities.  

Our proof of Theorem~\ref{thm main 2} is by real Monge-Amp\`ere techniques.  We consider the following problem.  Let $P\subset \mathbb{R}^n$ be a non-compact polytope.  Throughout the paper, we make the following assumption;

\begin{assumption}\label{ass 1}
A polytope $P \subset \mathbb{R}^n$, which may  be unbounded, is assumed to satisfy the following properties:
\begin{enumerate}
       \item
    $P\subset \mathbb{R}^n$ is a full-dimensional simple
    rational polytope with an irredundant facet presentation
    \[
        P
        =
        \left\{
            x\in \mathbb{R}^n :
            \langle x,\nu_i\rangle\geq -1,
            \quad i=1,\ldots,d
        \right\},
    \]
    where $\nu_1,\ldots,\nu_d\in \mathbb{R}^n$ are primitive inward
    facet normals. In particular,
    $0\in\operatorname{int}(P)$.

    \item
    The recession cone
    \[
        C:=\operatorname{rec}(P)
        =
        \left\{
            u\in \mathbb{R}^n :
            \langle u,\nu_i\rangle\geq 0,
            \quad i=1,\ldots,d
        \right\}
    \]
    is strongly convex, namely $C\cap(-C)=\{0\}$.
\end{enumerate}
\end{assumption}
For a polytope satisfying Assumption~\ref{ass 1} we consider the following real Monge-Amp\`ere equation for a convex function $u:P \rightarrow \mathbb{R}$.
 \begin{equation}\label{e main equ}
        \det (D^2u) = e^{\langle y, \nabla u (y) \rangle -u(y) -\langle \xi, y \rangle }, \quad \text{ and }\quad  \nabla u(P)= \R^n.
    \end{equation}
where $\xi \in \R^n$ is the unique vector satisfying the barycenter condition $\int_P y e^{-\langle \xi, y \rangle }dy=0$.  Equation~\eqref{e main equ} can be equivalently framed in terms of the Legendre transform $u^*:=\varphi:\mathbb{R}^n\rightarrow \mathbb{R}$ as

\begin{equation}\label{eq: mainLegendre}
\det(D^2\varphi)= e^{\langle \xi, \nabla \varphi \rangle -\varphi},\quad \text{ and }\quad \nabla\varphi(\mathbb{R}^n)=P
\end{equation}

In fact, it follows from work of Cordero-Erausquin-Klartag \cite{CEK} (which itself builds on the work of Berman-Berndtsson \cite{BermanBerndtsson13}) that a weak solution of~\eqref{e main equ} exists.  The challenging question, addressed in this paper, concerns the boundary regularity of the solution.  By work of Guillemin \cite{Guillemin1994} it is known that if $P$ is a polytope associated to a smooth toric manifold $X$, then a convex function $u \in C^{0}(\overline{P})$ defines a smooth, toric K\"ahler metric on $X$ if and only if it satisfies the {\bf Guillemin boundary conditions}: 
    \begin{itemize}
        \item[(i)]the restriction of $u$ to $P$, or to the interior of any face of $P$, is smooth and strictly convex,
        \item[(ii)] For suitable linear functions $L_i(y)$ defining the boundary of $P$, if we set
        \[
        u_{0}:= \sum_{i=1}^{k} L_i(y)\log(L_i(y))
        \]
        then $u-u_0$ is the restriction of a smooth function defined on an open neighborhood of $\overline{P}$.
        \end{itemize}

We refer the reader to Section~\ref{sec back} for more detailed discussion.  We summarize the main analytic challenges as follows:
\begin{itemize}
    \item Prove that the solution $u$ of~\eqref{e main equ} satisfies the Guillemin boundary conditions, and therefore extends to a smooth gradient K\"ahler-Ricci soliton on the toric variety $X$.
    \item Prove that the resulting metric is complete.
\end{itemize}

Both of these difficulties are new in the non-compact setting.  For compact Fano manifolds Wang-Zhu \cite{WangZhu2004} and Berman-Berndtsson \cite{BermanBerndtsson13} derive a $C^0$ estimate for the potential $u$ and then deduce higher-order estimates by passing to the compact complex manifold and applying suitable modifications of Yau's global estimates for the complex Monge-Amp\`ere equation \cite{Yau1978}. In the non-compact case Yau's estimates are not available. As a result, we develop new PDE techniques to derive local higher order estimates directly on $P$. In this direction, the second author \cite{Huang2023}, in combination with earlier work of Rubin \cite{Rubin}, established the Guillemin boundary condition for a related real Monge-Amp\`ere equation on compact polyhedra in two-dimensions. This was generalized by the second author and Shen \cite{HuangShen2026} to higher dimension.

The strategy of the proof of the Theorem \ref{thm main 2} is as follows: we approximate the unbounded polytope $P$ by a sequence of bounded polytopes $P_R$. By a result of Legendre \cite{Leg} we obtain a solution $u_R$ to (\ref{e main equ}) on $P_R$, satisfying the Guillemin boundary conditions. We derive locally uniform $C^2$ estimate for $\varphi_R=u_{R}^*$, the Legendre dual of $u_R$. The key step is to derive locally uniform weighted $C^{1,1}$ estimate for $u_R$.  This step employs an argument by induction on the codimension of faces of $P$. An Evans-Krylov type estimate implies uniform higher order estimates for $u_{R}$ with suitable weights. This allows us to take a convergent subsequence of $u_R$ as $R\rightarrow \infty$ which converges to a solution $u$ to (\ref{e main equ}) which satisfies the Guillemin boundary condition. Finally, to show that the limiting K\"ahler-Ricci shrinker defined by $u$ is complete, we prove a growth rate for components of the inverse Hessian $(D^{2}u)^{-1}$ in the direction of the critical Reeb field $\xi$.

The key point is that the proof strategy does not require passing to the K\"ahler manifold and applying global estimates.  In particular, our estimates are proved directly on the polytope $P$ bypassing the difficult problem of establishing purely local estimates for complex Monge-Amp\'ere equation.   As a result, we expect such a framework can be useful for the study of other Monge-Amp\`ere type equations on non-compact toric manifolds and toric manifolds with singularities.

The outline of the paper is as follows: Section~\ref{sec back} contains relevant background needed throughout the paper, including the general description of Fano fibrations, as well as some basic toric geometry.  We reformulate Theorem~\ref{thm main 2} in terms of the real Monge-Amp\`ere equation~\eqref{e main equ}.  Section~\ref{sec: momentMeasures} recalls the work of Cordero-Erausquin-Klartag \cite{CEK} on moment measures.  We show that the weak solution of the Monge-Amp\`ere equation~\eqref{e main equ} can be approximated in suitable norms by solutions of the same problem on truncated compact polytopes $P_{R}$.  Section~\ref{sec: completeness} establishes the growth estimate that leads eventually to the completeness of the K\"ahler-Ricci solitons we construct.  Section~\ref{sec: C2estimateRn} establishes uniform $C^2$ estimates for the Legendre dual potential $\varphi_{R}$ along the approximation. Section~\ref{sec: EKest} establishes an Evans-Krylov type estimate, which is used in Section~\ref{sec: Guillemin} to establish uniform control on the Guillemin boundary conditions along the approximation. Using these results we finally give a proof of the main existence and regularity theorem for the Monge-Amp\`ere equation~\eqref{e main equ}.  Section~\ref{sec: volumeGrowth} proves a general result concerning the volume growth of toric gradient K\"ahler-Ricci shrinkers, and finally we prove Theorem~\ref{thm main 2}.  Section~\ref{sec sta rel} discusses the relationship between the Sun-Zhang volume minimization procedure, which is used to select the critical Reeb vector field, and the notion of weighted stability, as defined by Lahdili \cite{Lahdili} and studied in the present context by Cifarelli \cite{Cifarelli2024}.

\bigskip

\noindent{\bf Acknowledgements:} T.C.C.~is supported in part by NSERC Discovery grant RGPIN-2024-03853.  
G. Huang is  supported in part by National Key R$\&$D Program of China 2025YFA1017600 and
National Natural Science Foundation of China under Grants 12526202.
F.T.~is supported in part by NSERC Discovery grant RGPIN-2025-06760. We thank Ronan Conlon for very helpful discussions. 

\bigskip

\noindent{\bf Tool and computational resource disclosure:}
ChatGPT 5.6-Sol was used to perform literature review and as an adversarial reader.  In the course of review, ChatGPT suggested some simplifications to the proofs of Proposition~\ref{prop: linGrowthApprox}, Lemma~\ref{lem:corner-domain} and Proposition~\ref{prop hig dim c2} that were incorporated into the final draft.  The authors retain exclusive responsibility for the correctness, accuracy, and validity of the results presented in this work.

\section{Background}\label{sec back}

\subsection{Toric Fano Fibrations}
We follow the terminology of Sun-Zhang \cite{SZ},
\begin{defn}\label{defn: FanoFibration}
     A {\bf Fano fibration} is a surjective projective morphism
\[
\pi:X\longrightarrow Y
\]
between normal varieties, with $\pi_*\mathcal{O}_X=\cO_Y$, such that $X$ has klt singularities and $-K_X$ is relatively ample.
\end{defn}

Now assume that there is an algebraic torus $\mathbb{T}^{\mathbb{C}}$ acting on $X$ and $Y$ such that $\pi:X\rightarrow Y$ is $\mathbb{T}^{\C}$-equivariant.  Let $\mathbb{T}$ denote the maximal compact sub-torus of $\bT^{\C}$.  Let $R_{X}$ (resp. $R_{Y}$) denote the ring of regular functions on $X$ (resp. $Y$)
\begin{defn}\label{defn: ReebField}
    We say that $\xi \in \mathfrak{t} = {\rm Lie}(\bT)$  is a {\bf Reeb vector field} if, under the action of $\bT$ on $Y$ we have
    \[
    R_{Y} = \bigoplus_{\alpha\in \mathfrak{t}^*}(R_{Y})_{\alpha}
    \]
    and $-i\alpha(\xi)>0$ for all $\alpha \in \mathfrak{t}^*\setminus \{0\}$ for which the weight space $(R_{Y})_{\alpha} \ne  \emptyset$.  We denote by $\mathcal{R}(P_{X})$ the convex, polyhedral cone of all Reeb vector fields.
\end{defn}

\begin{defn}\label{defn: polarizedFanoFibration}
A polarized Fano fibration is a triple  $(\pi: X\rightarrow Y, \bT, \xi)$ consisting of a Fano fibration with an equivariant $\bT^{\C}$-action, together with a Reeb vector field $\xi\in  {\rm Lie}(\bT)$.
\end{defn}

Polarized Fano fibrations interpolate between Fano varieties, obtained when $Y$ is a point, and Fano cones, obtained when $\pi$ is the identity. Sun-Zhang \cite{SZ} propose that K-polystable polarized Fano fibrations should correspond to complete shrinking K\"ahler–Ricci solitons.

From now on we assume that $(\pi: X\rightarrow Y, \bT, \xi)$ is toric.

\begin{defn}
    Let $(\pi:X \rightarrow Y, \bT, \xi)$ be a polarized Fano fibration with $\dim_{\mathbb{C}}X=n_{X}$ and $\dim_{\mathbb{C}}Y=n_{Y}$.  We say that $(\pi:X \rightarrow Y, \xi)$ is {\bf toric} if:
    \begin{itemize}
    \item[(i)]there is an algebraic torus $\mathbb{T}_{Y}^{\C}\subset {\rm Aut}(Y)$ with $\dim_{\C}\bT^{\C}_{Y}=n_{Y}$ 
        \item[(ii)] there is an algebraic torus $\mathbb{T}_{X}^{\C}\subset {\rm Aut}(X)$ with $\dim_{\C}\bT_{X}^{\C}=n_{X}$ 
        \item[(iii)] $\bT_{X}^{\C}$ (resp. $\bT_{Y}^{\C})$ acts effectively on $X$ (resp. $Y$), and for a generic point $p_{X}\in X$ (resp. $p_{Y}\in Y$) the orbit $\bT_{X}^{\C}\cdot p_{X}$ is dense in $X$ (resp. $\bT_{Y}^{\C}\cdot p_{Y}$ is dense in $Y$),
        \item[(iv)]There are homomorphisms
\[
\rho:\bT_X^{\mathbb C}\longrightarrow\bT_Y^{\mathbb C}
\]
and
 \[
        \iota_{X}: \bT\rightarrow \bT_{X}, \quad \text{ and }\quad \iota_{Y}:\bT \rightarrow \bT_{Y}
\]
inducing the given $\bT$-actions on $X$ and $Y$, such that $\pi$ is $\rho$-equivariant; ie. for any $t\in \bT_{X}^{\C}$ we have $\pi(t\cdot x) = \rho(t)\cdot \pi(x)$.
    \end{itemize}
\end{defn}

Roughly speaking, a toric polarized Fano fibration is a polarized Fano fibration for which the generic fiber is compact toric Fano variety and the polarized affine cone $Y$ is also toric.  

Toric polarized Fano fibrations can be described in terms of certain convex sets together with a distinguished direction.  From now on we assume that $\dim_{\C}X >\dim_{\C}Y >0$. Let 
\[
\begin{aligned}
M_{X}:= {\rm Hom}(\bT_{X}^{\C}, \C^*), \qquad M_{Y}:= {\rm Hom}(\bT_{Y}^{\C}, \C^*)\\
N_{X}:= {\rm Hom}(\C^*,\bT_{X}^{\C}), \qquad N_{Y}:= {\rm Hom}(\C^*, \bT_{Y}^{\C})
\end{aligned}
\]
denote the character and co-character lattices of $\bT_{X}^{\C}$ and $\bT_{Y}^{\C}$ respectively, and let $M_{X,\bR}= M_{X}\otimes_{\mathbb{Z}}\bR$ and $M_{Y,\bR} =M_{Y}\otimes_{\mathbb{Z}}\bR $ (and similarly for $N_{X,\R}, N_{Y,\R})$.  The map $\rho: \bT_{X}^{\C}\rightarrow \bT_{Y}^{\C}$ induces an inclusion $M_{Y,\R}\hookrightarrow M_{X, \bR}$ and a projection $N_{X,\R}\rightarrow N_{Y, \bR}$.

Let $P_{Y}\subset M_{Y,\bR}$ denote the moment cone of the toric polarized affine variety $Y$. $P_{Y}$ is a $n_{Y}$-dimensional, strongly convex rational polyhedral cone with vertex at the origin.  The toric polarized Fano fibration
\[
\pi:X\to Y
\]
is described by an unbounded, rational moment polytope
\[
P_{X}\subset M_{X,\mathbb R}.
\]
That is, $P_{X}$ can be written as
\[
P_{X}=\{y\in M_{X,\mathbb R}:\langle y,\nu_i\rangle\geq -1,\ i=1,\dots,d\},
\]
where the facet normals $\nu_i\in N_X= M_{X}^{\vee}$.  In particular, we see that $0\in P_{X}$.

\begin{defn}
    The polytope $P_{X}$ is simple if exactly $\dim_{\mathbb C}X$ facets meet at every vertex.
\end{defn}

$P_{X}$ being simple is equivalent to the normal fan being simplicial and hence to $X$ being $\mathbb Q$-factorial \cite{CLS}. In the toric setting, simplicial singularities are finite abelian quotient singularities. If $X$ is smooth, then $P$ satisfies the Delzant condition

\begin{defn}
      The polytope $P_{X}$ is Delzant if, at every vertex $v\in P_{X}$, the primitive inward normals to the facets meeting at $v$ form a $\mathbb Z$-basis of $N_X$.
\end{defn}

The equivariant fibration $\pi:X\rightarrow Y$ implies that $P_{Y}$ is precisely the recession cone of $P_{X}$, where we view $P_{Y}\subset M_{X,\R}$ by the inclusion $M_{Y,\R}\subset M_{X,\R}$.  In particular we can write $P_{X}$ as a Minkowski sum
\[
P_{X} = Q + P_{Y}
\]
where $Q\subset M_{X,\R}$ is a compact polytope.

In this setting the Reeb cone $\mathcal{R}(P_{X})$ is a convex polyhedral cone in $N_{X,\bR}$. A vector field $\xi \in N_{X,\bR}$ is a Reeb vector field if $\langle \xi, p \rangle >0$ for all non-zero $p\in P_{Y}\subset P_{X}$.  Equivalently, since $P_{Y}$ is the recession cone of $P_{X}$, we can characterize Reeb vector fields by the condition that 
\[
P_{\xi, R}:= \{ p \in P_{X} : \langle \xi, p \rangle \leq R\}
\]
is compact. 
We define the Sun-Zhang weighted-volume function by
\[
\mathcal{R}(P_{X}) \ni \xi \mapsto {\rm Vol}(\xi):=  \int_{P_{X}}e^{-\langle \xi, y\rangle} dy
\]
Evidently this function is finite on $\mathcal{R}(P_{X})$ and
\[
\lim _{\xi \rightarrow \del \mathcal{R}(P_{X})} {\rm Vol}(\xi) = +\infty.
\]
Furthermore, since $0$ is an interior point of $P_{X}$ we have
\[
\lim_{t\rightarrow 0} \int_{P_{X}}e^{-t\langle \xi, y\rangle} dy= +\infty
\]
and so ${\rm Vol}(\xi)$ is proper on the open cone $\mathcal{R}(P_{X})$.  By direct calculation one verifies that ${\rm Vol}(\xi)$ is strictly convex and hence we obtain

\begin{lem}
    The function
    \[
    \mathcal{R}(P_{X}) \ni \xi \mapsto {\rm Vol}(\xi)
    \]
    has a unique minimum in $\mathcal{R}(P_{X})$.  Let $\xi_{*}$ denote the minimizer then we have
    \[
    \int_{P_{X}} \vec{y}e^{-\langle \xi_*, y\rangle} dy=0
    \]
\end{lem}

The following definition gives an equivalent formulation of a toric Fano fibration in terms of combinatorial data encoded by a non-compact polytope.

\begin{defn}
\label{defn: toricFanoPoly}
Let $M_X$ be a lattice of rank $n_X$, let
$N_X=\operatorname{Hom}(M_X,\mathbb Z)$, and write
$M_{X,\mathbb R}=M_X\otimes_{\mathbb Z}\mathbb R$ and
$N_{X,\mathbb R}=N_X\otimes_{\mathbb Z}\mathbb R$.
A polyhedral presentation of a toric Fano fibration with
Reeb field consists of a pair $(P_X,\xi)$ satisfying the
following conditions:
\begin{enumerate}
    \item
    $P_X\subset M_{X,\mathbb R}$ is a full-dimensional
    rational polytope with an irredundant facet presentation
    \[
        P_X
        =
        \left\{
            x\in M_{X,\mathbb R} :
            \langle x,\nu_i\rangle\geq -1,
            \quad i=1,\ldots,d
        \right\},
    \]
    where $\nu_1,\ldots,\nu_d\in N_X$ are primitive inward
    facet normals. In particular,
    $0\in\operatorname{int}(P_X)$.

    \item
    The recession cone
    \[
        C:=\operatorname{rec}(P_X)
        =
        \left\{
            u\in M_{X,\mathbb R} :
            \langle u,\nu_i\rangle\geq 0,
            \quad i=1,\ldots,d
        \right\}
    \]
    is strongly convex, namely $C\cap(-C)=\{0\}$.
    Set
    \[
        M_Y
        :=
        M_X\cap\operatorname{span}_{\mathbb R}(C),
        \qquad
        N_Y:=\operatorname{Hom}(M_Y,\mathbb Z),
        \qquad
        n_Y:=\dim C.
    \]

    \item
    $\xi\in N_{X,\mathbb R}$ satisfies
    \[
        \langle u,\xi\rangle>0
        \qquad
        \text{for every }u\in C\setminus\{0\}.
    \]
\end{enumerate}
The associated toric Fano fibration is
\[
    \pi:X_{\Sigma_{P_X}}
    \longrightarrow
    Y:=\operatorname{Spec}\mathbb C[C\cap M_Y],
\]
where $\Sigma_{P_X}$ is the inward normal fan of $P_X$.
The morphism $\pi$ is induced by the surjective lattice map
\[
    q:N_X\longrightarrow N_Y
\]
dual to the inclusion $M_Y\hookrightarrow M_X$.
The vector $\xi$ induces the Reeb field
$\bar\xi:=q_{\mathbb R}(\xi)$ on $Y$.

The Reeb domain of $P_X$ is
\[
    \mathcal R(P_X)
    :=
    \left\{
        \xi\in N_{X,\mathbb R} :
        \langle u,\xi\rangle>0
        \text{ for all }u\in C\setminus\{0\}
    \right\}.
\]
Equivalently, $\xi\in\mathcal R(P_X)$ if and only if
\[
    P_{\xi,R}
    :=
    \left\{
        x\in P_X:\langle x,\xi\rangle\leq R
    \right\}
\]
is compact for every $R\in\mathbb R$.
\end{defn}

\begin{rk}
To restrict the Reeb field to a chosen lift of the base
torus, fix a lattice splitting
\[
    s:N_Y\longrightarrow N_X,
    \qquad q\circ s=\operatorname{id}_{N_Y},
\]
and take $\xi=s_{\mathbb R}(\bar\xi)$, where
\[
    \bar\xi\in
    \operatorname{int}_{N_{Y,\mathbb R}}(C^\vee),
    \qquad
    C^\vee
    :=
    \left\{
        v\in N_{Y,\mathbb R}:
        \langle u,v\rangle\geq 0
        \text{ for all }u\in C
    \right\}.
\]
This splitting is not canonically determined by $P_X$.
\end{rk}
\bigskip

\noindent {\bf Notational Convention}:
\\
Throughout the paper we shall identify $N_{X,\bR}\simeq \bR^n$ and similarly for the dual space $M_{X,\bR}$, which we identify with $(\bR^n)^*$.  With these identifications we view the polytope $P\subset (\mathbb{R}^n)^*$. We denote functions defined on subsets of  $(\mathbb{R}^n)^*$ by Roman symbols (e.g $u,v,w$ etc).  On $\mathbb{R}^n$, we use Greek letters for functions; e.g. $\vphi, \psi$ etc.
\bigskip

\subsection{Guillemin Boundary Conditions}

\begin{defn}
    A labeled polytope  $(P,\nu, b)$ is a triple consisting of an open, bounded, convex polytope $P\subset M_{X,\bR}$ with $k$ facets and $\nu= \{\nu_1,\ldots,\nu_k\} \subset N_{X,\bR}$ a collection of vectors, and $b=\{b_1, \ldots, b_k\} \in \mathbb{R}^k$ such that
        \[
        P = \bigcap_{i=1}^{k} \big\{y \in M_{X,\bR} : \langle y, \nu_i\rangle  >b_i\big\}.
        \]
    A labeled polytope is said to be {\bf monotone} with preferred center $0$ if $0\in P$ and $b_1=\cdots =b_k$.
    \end{defn}

\begin{lem}\label{lem: goodTruncation}
    Let $P_{X}$ be a polytope as before with inward pointing normal vectors $\nu_1,\ldots, \nu_{k-1}$, and $\xi$ a Reeb vector field.  Let $\nu_k= -\frac{\xi}{R}$.  Set
    \[
    P_{\xi, R}:= \{ p \in P_{X} : \langle \xi, p \rangle \leq R\} \Subset P_{X}
    \]
    Then, for all $R$ sufficiently large, $(P_{\xi, R}, \{\nu_1,,\ldots, \nu_{k-1}, -\frac{\xi}{R}\}, \{-1, \ldots, -1\})$ is a simple, labeled monotone polytope with preferred center $0$.
\end{lem}
\begin{proof}
$P_{\xi, R}$ will be simple as long as $\{\langle \xi,p\rangle =R\} \cap \overline{P_{X}}$ does not contain a vertex of $P_{X}$.  Since all vertices of $P_{X}$ lie in a compact set, we conclude that $P_{\xi, R}$ is simple for all $R$ sufficiently large.  That $(P_{\xi, R}, \{\nu_1,,\ldots, \nu_{k-1}, -\frac{\xi}{R}\}, \{-1, \ldots, -1\})$ is a labeled monotone polytope with preferred center $0$ is clear.
\end{proof}

\begin{defn}\label{def: GuilleminBoundary}
    Let $(P, \nu, b)$ be a labeled polytope. A symplectic potential is a continuous function $u \in C^0(\overline{P})$ such that:
    \begin{itemize}
        \item[(i)]the restriction of $u$ to $P$, or to the interior of any face of $P$, is smooth and strictly convex,
        \item[(ii)] if we set $L_i= \langle \nu_i, y\rangle - b_i$ and
        \[
        u_{0}:= \sum_{i=1}^{k} L_i(y)\log(L_i(y))
        \]
        then $u-u_0$ is the restriction of a smooth function defined on an open neighborhood of $\overline{P}$.
    \end{itemize}
    We denote by $S(P, \nu,b)$ the set of all symplectic potentials on $(P,\nu, b)$. Note that the definition of usual Guillemin boundary condition uses $\frac{1}{2}\sum_{i=1}^{k} L_i(y)\log(L_i(y))$ instead, which is essentially the same.
\end{defn}

From now on we denote $P_{\xi, R}=P_{R}$, with the dependence on the Reeb field understood.  We also denote
\[
(P_{R}, \nu_{R}, b) = ((P_{\xi, R}, \{\nu_1,,\ldots, \nu_{k-1}, -\frac{\xi}{R}\}, \{-1, \ldots, -1\})
\]
the simple, labeled monotone polytope with preferred center $0$ provided by Lemma~\ref{lem: goodTruncation}.  The following result is due to Tian-Zhu \cite[Lemma 2.2]{TianZhu02}; see also \cite[\S3.8]{BermanBerndtsson13}.

\begin{lem}
 Suppose that \(R\geq R_{0}\) is such that
\[
    P_R\subset M_{X,\mathbb R}
\]
is a compact, full-dimensional convex polytope satisfying
\(0\in\operatorname{int}P_R\). Then the function
\[
\operatorname{Vol}_R\colon N_{X,\mathbb R}\longrightarrow
\mathbb R_{>0},
\qquad
\operatorname{Vol}_R(\eta)
   =\int_{P_R}e^{-\langle\eta,y\rangle}\,dy,
\]
is smooth, strictly convex, and proper. Consequently, it admits a
unique minimizer \(\xi_R\in N_{X,\mathbb R}\), characterized by
\[
    \int_{P_R}y\,e^{-\langle\xi_R,y\rangle}\,dy=0.
\]
\end{lem}

Apostolov--Calderbank--Gauduchon--T{\o}nnesen-Friedman \cite{ApostolovCalderbank2004} gave the
following alternative description of the boundary condition.

\begin{prop}[Proposition 1, \cite{ApostolovCalderbank2004}]\label{prop H gui}
Given a labelled polytope $(P,\nu)$, a strictly convex
function $u\in C^\infty(P)$ is a symplectic potential of $(P,\nu)$ if
and only if, denoting
\[
\mathbf{H}=(\operatorname{Hess}\frac{u}{2})^{-1},
\]
the following conditions hold:
\begin{itemize}
    \item $\mathbf{H}$ is the restriction to $P$ of a smooth
    $S^2\mathfrak{t}^*$-valued function on $\overline{P}$;

    \item for every $k=1,\ldots,d$ and every $y$ in the interior of
    the facet $F_k$,
    \begin{equation}
    \mathbf{H}_y(\nu_k,\cdot)=0
    \qquad\text{and}\qquad
    d\mathbf{H}_y(\nu_k,\nu_k)=2\nu_k;
    \tag{7}
    \end{equation}

    \item the restriction of $\mathbf{H}$ to the interior of any face
    $F\subset P$ is a positive-definite
    $S^2(\mathfrak{t}/\mathfrak{t}_F)^*$-valued function.
\end{itemize}
\end{prop}

\subsection{Local complexification of the polytope}

Given a labelled polytope $(P,\nu)$, we denote the set of (closed)
faces of $P$ by $\mathcal{F}(P)$. The facets of $P$ are still denoted
by
\[
F_1,\ldots,F_d\in\mathcal{F}(P).
\]
For $F\in\mathcal{F}(P)$, denote by
\[
I_F\subset\{1,\ldots,d\}
\]
the set of indices such that
\[
F=\bigcap_{k\in I_F}F_k.
\]
For example, $P\in\mathcal{F}(P)$ and $I_{\overline{P}}=\varnothing$.
For a vertex $p$, $I_{\{p\}}$ has $n$ elements, and
\[
\Lambda_p
=
\operatorname{span}_{\mathbb{Z}}
\{\nu_k\mid k\in I_{\{p\}}\}
\]
is a lattice in $\mathfrak{t}$. Given a vertex $p$ of $P$, we call $\mathcal{F}_p(P)$ the set of
faces containing $p$. For a face
$F\in\mathcal{F}_p(P)$,
\[
T_F
=
\frac{
  \operatorname{span}_{\mathbb{R}}
  \{\nu_k\mid k\in I_F\}
}{
  \Lambda_p\cap
  \operatorname{span}_{\mathbb{R}}
  \{\nu_k\mid k\in I_F\}
}
\]
is a subtorus of
\[
T_p=\mathfrak{t}/\Lambda_p
\]
if $p\in F$.

For $F\in\mathcal{F}_p(P)$, we denote by
$s_p(F)$ the subset of $F$ obtained by removing all the subfaces
which do not contain $p$; that is,
\[
s_p(F)
=
\left\{
x\mathrel{}\middle|\mathrel{}
x\in\mathring{E},\ p\in E,\ E\subset F
\right\},
\]
where $\mathring{E}$ is the interior of the face $E$ (in $E$). In
particular, the interior of a vertex is the vertex itself. Thus,
\[
\bigcup_{F\in\mathcal{F}_p(P)}s_p(F)
=
\bigcup_{F\in\mathcal{F}_p(P)}\mathring{F}
\]
is an open neighborhood of $p$ in $\overline{P}$. Set
\[
M_p
=
\left.
\bigsqcup_{F\in\mathcal{F}_p(P)}
\left(s_p(F)\times T_p/T_F\right)
\middle/\sim
\right.,
\]
where, for
\[
(x,\theta)\in F\times T_p/T_F
\quad\text{and}\quad
(x',\theta')\in F'\times T_p/T_{F'},
\]
we define $(x,\theta)\sim(x',\theta')$ if

\begin{enumerate}
    \item $x=x'$, and
    \item the equivalence classes of $\theta$ and $\theta'$ in
    $T_p/T_{F\cap F'}$ coincide.
\end{enumerate}

Here, the first condition implies that $F\cap F'\neq\varnothing$, so
$F\cap F'\in\mathcal{F}_p(P)$ and $T_{F\cap F'}$ contains $T_F$ and
$T_{F'}$ as subgroups. The second condition refers to the fact that
$T_p/T_{F\cap F'}$ is the quotient of $T_p/T_F$ by
$T_{F\cap F'}/T_F$ and the quotient of $T_p/T_{F'}$ by
$T_{F\cap F'}/T_{F'}$.

Ordering the normals
\[
\nu_{k_1},\ldots,\nu_{k_n}
\qquad
(k_i\in I_{\{p\}}),
\]
we obtain an identification
\[
T_p\simeq\mathbb{T}^n=\mathbb{R}^n/\mathbb{Z}^n,
\]
via which $T_p$ acts on $\mathbb{C}^n$. For an equivariant
neighborhood $U_p$ of $0\in\mathbb{C}^n$, the map
\[
\phi_p\colon U_p\longrightarrow M_p,
\]
defined by
\begin{equation}
\phi_p(z)
=
\left[
\left(
p+|z_i|^2\nu_{k_i}^{*},
\left(
e^{2\pi\sqrt{-1}\theta_1},
\ldots,
e^{2\pi\sqrt{-1}\theta_n}
\right)
\right)
\right],
\tag{5}
\end{equation}
where
\[
z=
\left(
|z_1|e^{2\pi\sqrt{-1}\theta_1},
\ldots,
|z_n|e^{2\pi\sqrt{-1}\theta_n}
\right),
\]
is a well-defined (i.e., it does not depend on the choice of
$e^{2\pi\sqrt{-1}\theta_i}$ when $|z_i|=0$) equivariant
homeomorphism. The chart $(U_p,\phi_p)$ provides a smooth
differential structure on $M_p$.

In the case that there exists a fixed lattice $\Lambda$ such that $\Lambda_p = \Lambda$ for any vertex $p$ (this is also called Delzant), Duistermaat-Pelayo used the above construction in \cite{DuistermaatPelayo2009} (alternative to the so-called Delzant construction \cite{Delzant1988}) to build a toric manifold $(M,\omega, T)$ from the data $(P,\nu,\Lambda)$, see also \cite{Donaldson2008Toric}. Legendre \cite{Leg} adapted such construction to more general labelled polytope as is stated in this section. Using this construction, she proved the following Theorem by using the equivalence of Guillemin boundary condition and the smoothness in local complexification.   

\begin{thm}[Legendre \cite{Leg}, Theorem 1.6]\label{thm: LegendreKRS}
For each $R \geq 1$, let $\xi_{R} \in N_{X,\bR}$ denote the unique minimizer of ${\rm Vol}_{R}(\eta)= \int_{P_{R}}e^{-\langle \eta, y \rangle}$.  Then there exists a convex function $\vphi_{R}:\mathbb{R}^n\rightarrow \mathbb{R}$, unique up to translation and addition of a constant, such that
\[
\begin{aligned}
\det D^2\vphi_{R} &=e^{\langle \xi_{R}, \nabla \vphi_{R}\rangle -\vphi_{R} +C}\\
\nabla \vphi_{R}(\bR^n)&=P_{R}
\end{aligned}
\]
and $\vphi_{R}^*:= u_{R} \in S(P_{R}, \nu_{R}, b)$.  That is, the symplectic potential $u_{R}$ has Guillemin boundary conditions in the sense of Definition~\ref{def: GuilleminBoundary} and is unique up to the addition of an affine function.
\end{thm}

\subsection{Local complexification of the polytope, II}

For any $p \in \partial P$, take a coordinate 
\[
  y=(y',z),\qquad
  y'=(y_1,\ldots,y_k),\qquad
  z=(y_{k+1},\ldots,y_n)\in\R^m
\]
centered at $p$ such that locally near $p$, $\partial P$ is given by 
\begin{equation*}
\cup_{i=1}^k \{(y',z): y_i=0\}.
\end{equation*}

Let
$\Omega'\Subset\Omega\subset\R^m$ and let $\delta>0$.  Consider
\[
  \mathcal Q=(0,\delta)^k\times\Omega.
\]
The original polytope coordinates are
\begin{equation}
  \mathbf x(y)
  =(y_1-1,\ldots,y_k-1,y_{k+1},\ldots,y_n).
  \label{eq:x-of-y}
\end{equation}

Let $U\in C^\infty(\mathcal Q)$ be strictly convex and satisfy the Guillemin boundary condition.  Write
\begin{equation}
  U(y)=\sum_{i=1}^k  y_i\log y_i+V(y).
  \label{eq:guillemin-decomposition}
\end{equation}

Let
\begin{equation}
  p=DU(y),
  \qquad
  \Phi(p)=y\cdot p-U(y)
  \label{eq:legendre-transform}
\end{equation}
be the full Legendre transform.  In the interior,
\begin{equation}
  D\Phi=y,
  \qquad
  \operatorname{Hess}\Phi=(\operatorname{Hess} U)^{-1}.
  \label{eq:legendre-identities}
\end{equation}
For $i\leq k$ and $a>k$,
\begin{equation}
  p_i= (\log y_i+1)+V_i,
  \qquad
  p_a=V_a.
  \label{eq:p-components}
\end{equation}
Define
\begin{equation}
  \rho_i=e^{p_i}\quad(i\leq k),
  \qquad
  \eta_a=p_a\quad(a>k).
  \label{eq:rho-eta}
\end{equation}

Define
\begin{equation}
  F(\rho,\eta)
  =\Phi(\log\rho_1,\ldots,\log\rho_k,\eta).
  \label{eq:F-definition}
\end{equation}
Introduce complex coordinates $\zeta=(\zeta_1,\ldots,\zeta_n)\in\C^n$ by
\begin{equation}
  \rho_i=|\zeta_i|^2\quad(i\leq k),
  \qquad
  \eta_a=\zeta_a+\bar\zeta_a\quad(a>k),
  \label{eq:complex-coordinates}
\end{equation}
and set
\begin{equation}
  \Psi(\zeta)
  =F\bigl(
       |\zeta_1|^2,\ldots,|\zeta_k|^2,
       \zeta_{k+1}+\bar\zeta_{k+1},\ldots,
       \zeta_n+\bar\zeta_n
     \bigr).
  \label{eq:Psi-definition}
\end{equation}

Then, we have the following Proposition:
\begin{prop}\label{prop comp ext 2}
    We use the notations introduced above. Suppose that $U$ is strictly convex and satisfies the Guillemin boundary condition on $P$. Then, $\Psi$ can be extended smoothly as a plurisubharmonic function on a neighbourhood of $(0,\ldots,0, \frac{1}{2}\nabla_{k+1} U(p), \ldots, \frac{1}{2}\nabla_{n} U(p))$.
\end{prop}
\begin{proof}
    We have that:
    \begin{equation*}
        \rho_i = e^{p_i}= y_i e^{1+V_i(x)}
    \end{equation*}
    for $i \le k$. 
    Consider the map $A(y)= (y_1 e^{1+V_1(y)},\ldots, y_k e^{1+V_k(y)}, \frac{1}{2}V_{k+1}(y),\ldots, \frac{1}{2} V_n (y))$. Since $V$ is smooth up to the boundary, we have that $A$ can be extended to be a smooth function defined in a neighbourhood of $p$. Since $U$ is strictly convex, we have that $(V_{ij})_{i,j \ge k+1}$ is positive definite.
    Then, we have that $DA(0)$ is invertible. The smooth inverse function theorem then gives a smooth local inverse $y=Y(\rho_1,\ldots, \rho_k, \eta_{k+1},\ldots, \eta_n)$. Then, we compute that:
    \begin{equation*}
        F(\rho.\eta) = \sum_{i=1}^k y_i \big((\log y_i +1)+ V_i \big) + \sum_{i=k+1}^n y_i V_i -\sum_{i=1}^k y_i \log y_i - V(y)= \sum_{i=1}^n y_i V_i +\sum_{i=1}^k y_i -V
    \end{equation*}
    Since we have shown that $y$ is a smooth function of $(\rho,\eta)$, which is a smooth function of $(\zeta)$, the above formula shows that $\Psi$ can be extended smoothly to a neighbourhood of  $(0,\ldots,0, \frac{1}{2}\nabla_{k+1} U(p), \ldots, \frac{1}{2}\nabla_{n} U(p))$.

    Next, we show that $\Psi$ is a PSH function. We can compute that:
    \begin{equation*}
        \frac{\partial F}{\partial \rho_i} = \frac{\partial \Phi}{\partial p_i} \frac{\partial p_i}{\partial \rho_i}=\frac{\partial \Phi}{\partial p_i} \frac{1}{ \rho_i}= \frac{1}{e^{1+V_i}}>0
    \end{equation*}
    for $i \le k$. Since $\Phi=U^*$ whose Hessian is positive definite, we have that $(F_{a\bar{b}})_{a,b \ge k+1}$ is positive definite. Note that:
    \begin{equation*}
        \Psi_{\zeta_{i}\bar{\zeta}_j}= \delta_{ij} F_{\rho_i}+ \bar{\zeta}_i \zeta_j F_{\rho_i \rho_j}
    \end{equation*}
    for $i,j \le k$, and
    \begin{equation*}
        \Psi_{\zeta_i \bar{\zeta}_a}= \bar{\zeta}_i F_{\rho_i \eta_a},
    \end{equation*}
    for $i \le k <a$, and
    \begin{equation*}
        \Psi_{\zeta_a \bar{\zeta}_b} = F_{\eta_a \eta_b},
    \end{equation*}
    for $a,b >k$. As a result, we have that $(\Psi_{i\bar{j}})$ is strictly positive at $(0,\ldots,0, \frac{1}{2}\nabla_{k+1} U(p), \ldots, \frac{1}{2} \nabla_{n} U(p))$. As a result, by the smoothness of $\Psi$, it is positive in a neighbourhood of $(0,\ldots,0, \frac{1}{2}\nabla_{k+1} U(p), \ldots, \frac{1}{2}\nabla_{n} U(p))$.
\end{proof}

\subsection{Reduction to a Real Monge-Amp\`ere Equation}
Let $(\pi:X\rightarrow Y, \mathbb{T}, \xi)$ be a toric polarized Fano fibration.  By Definition~\ref{defn: toricFanoPoly}, and following a standard local calculation (see, for example, \cite{WangZhu2004}), Theorem~\ref{thm main 2} is reduced to proving the following existence theorem for a real Monge-Amp\`ere equation

\begin{thm}\label{main thm 1}
    For any polytope $P\subset \mathbb{R}^n$ satisfying Assumption \ref{ass 1}, there exists a smooth function $u$, unique up to the addition of linear functions, satisfying the Guillemin boundary condition on $P$ and solving the real Monge-Amp\`ere equation
    \begin{equation}\label{e main equ1}
        \det (D^2u) = e^{\langle y, \nabla u (y) \rangle -u(y) -\langle \xi, y \rangle }, \quad \text{ and } \quad \nabla u(P)= \R^n,
    \end{equation}
    where $\xi \in \R^n$ is the unique vector satisfying $\int_P y e^{-\langle \xi, y \rangle }dy=0$.
\end{thm}


\section{Moment Measures}\label{sec: momentMeasures}

\subsection{Construction of Moment Measures}
Let $\vphi:\mathbb{R}^n \rightarrow \mathbb{R}$ be a convex function such that $0< \int_{\mathbb{R}^n} e^{-\vphi(x)}dx <+\infty$
\begin{defn}
The moment measure of $\vphi$, denoted by $\mu_{\vphi}$, is the Borel measure on $\mathbb{R}^n$ obtained by pushing forward $e^{-\vphi(x)}dx$ by the map $\nabla \vphi$.  That is
\[
\int_{\mathbb{R}^n} b(y)d\mu_{\vphi}(y)= \int_{\mathbb{R}^n}b(\nabla \vphi(x)) e^{-\vphi(x)}dx
\]
for every Borel function $b \in L^{1}(\mu_{\vphi})$ or such that $b \geq 0$.
\end{defn}

We have the following theorem.

\begin{thm}[Cordero-Erausquin-Klartag \cite{CEK}, Theorem 2]\label{thm: CEKMoment}
Let $\nu$ be a Borel measure on $\mathbb{R}^n$ such that
\begin{itemize}
    \item[(i)] $0<\nu(\mathbb{R}^n) < +\infty$,
    \item[(ii)] The measure $\nu$ is not supported in a lower dimensional subspace of $\mathbb{R}^n$, and
    \item[(iii)] The barycenter of $\nu$ lies at the origin, in the sense that
    \[
    \int_{\mathbb{R}^n} y d\nu(y)=0.
    \]
\end{itemize}
Then, there exists an essentially continuous convex function $\vphi:\mathbb{R}^n \rightarrow \mathbb{R}\cup \{+\infty\}$ such that $\nu$ is the moment measure of $\vphi$.  That is $\mu_{\vphi}=\nu$. Moreover, $\vphi$ is unique up to translation.
\end{thm}

Let us recall briefly some of the ideas of the proof of Theorem~\ref{thm: CEKMoment}.  Given a probability measure $\nu$ on $\mathbb{R}^n$ satisfying the assumptions of Theorem~\ref{thm: CEKMoment}. For a $\nu$-integrable function $f:\mathbb{R}^n \rightarrow \mathbb{R}\cup\{+\infty\}$ we set
\begin{equation}\label{eq: Ifunctional}
I_{\nu}(f):= \log \int_{\mathbb{R}^{n}} e^{-f^*} - \int_{\mathbb{R}^n}f d\nu.
\end{equation}
Cordero-Erausquin-Klartag \cite[Proposition 12]{CEK} prove that there exists a $\nu$-integrable, convex function $u:\mathbb{R}^n \rightarrow \mathbb{R}\cup \{+\infty\}$ such that
\[
I_{\nu}(u) = \sup_{f \in L^1_{+}(\nu)}I_{\nu}(f) \quad \text{ and } \quad \int_{\mathbb{R}^n}e^{-u^*}=1
\]
where $L^1_{+}(\nu)= \{ f: \mathbb{R}^n \rightarrow \mathbb{R}\cup \{+\infty\} : f \text{ is } \nu-\text{integrable}\}$.  \cite{CEK} then shows that by setting $\vphi=u^*$ we obtain the desired convex function $\vphi$ in Theorem~\ref{thm: CEKMoment}.  The uniqueness follows from Prekopa's theorem which implies that $t \mapsto I_{\nu}((1-t)f_0+tf_1)$ is concave, and characterizes precisely the linear case.

As a corollary, we obtain

\begin{cor}\label{cor: momentMeasureKRS}
    Let $\xi$ be the unique Reeb vector field on $P_{X}$ minimizing the volume function.  Then there exists a convex function $\vphi$ on $\mathbb{R}^n$, unique up to translation, such that
    \[
    \mu_{\vphi} = e^{-\langle \xi, y\rangle+C} \chi_{P_{X}}(y) dy
    \]
    where $C=-\log {\rm Vol}(\xi)$.
\end{cor}

\begin{lem}\label{lem varphi bdd}
Let $\varphi:\mathbb R^n\to\mathbb R\cup\{+\infty\}$ be an
essentially continuous convex function such that
$0<\int_{\mathbb R^n}e^{-\varphi}<\infty$.
Let $\nu=(D\varphi)_{\#}(e^{-\varphi}\,ds)$ be its moment measure.
If, for some $a>0$,
\[
\int_{\mathbb R^n}e^{a|x|}\,d\nu(x)<\infty,
\]
then $\varphi$ is finite on all of $\mathbb R^n$.
\end{lem}

\begin{proof}
Put $\Omega=\operatorname{int}\{\varphi<+\infty\}$.
The positive mass implies that $\Omega$ is nonempty, and the
moment-measure identity gives
\[
\int_\Omega e^{a|D\varphi|-\varphi}\,ds<\infty.
\]
Suppose that $\Omega\ne\mathbb R^n$.
A separating hyperplane gives a unit vector $\theta$ and a finite
constant $B$ such that $s\cdot\theta\leq B$ on $\Omega$.
By Fubini's theorem and the linewise continuity statement in
\cite[Lemma~3]{CEK}, we may choose a line $z+\mathbb R\theta$
meeting $\Omega$ on which the restriction
$h(t)=\phi(z+t\theta)$ is continuous as an extended-real-valued
function and
\[
E:=\int_I e^{a|h'|-h}\,dt<\infty.
\]
Here $I=(\alpha,b)$ is its interval of interior finiteness and
$b<\infty$. Continuity implies $h(t)\to+\infty$ as $t\uparrow b$.
In dimension one, use the same argument on the whole line;
essential continuity is then ordinary extended-real continuity.

Fix $t\in I$ and $p>1$. Applying the fundamental theorem of
calculus first on compact subintervals, and then taking the
right endpoint to $b$, gives
\[
e^{-h(t)/p}
\leq\frac1p\int_t^b |h'|e^{-h/p}\,d\tau
\leq\frac{(b-t)^{1-1/p}}p
     \left(\int_t^b |h'|^p e^{-h}\,d\tau\right)^{1/p}.
\]
Since $q^p\leq (p/(a\mathrm e))^p e^{aq}$ for $q\geq0$,
\[
e^{-h(t)/p}
\leq\frac{(b-t)^{1-1/p}}{a\mathrm e}\,E^{1/p}.
\]
Letting $p\to\infty$ yields $1\leq(b-t)/(a\mathrm e)$.
This is impossible for $t$ sufficiently close to $b$.
Thus $\Omega=\mathbb R^n$, as required.
\end{proof}

\begin{cor}\label{cor fin varphi}
    The function $\varphi$ obtained by the Corollary \ref{cor: momentMeasureKRS} is finite on $\mathbb{R}^n$
\end{cor}
\begin{proof}
    By the definition of $\xi$, there exist constants $C_1, C_2>0$ such that $\langle \xi,y\rangle \ge C_1 |y| -C_2 $ for $y \in P_X$. Then, we can take $a_1 = \frac{C_1}{2}$ to get that:
    \begin{equation*}
        \int_{\R^n} e^{a_1|y|} e^{-\langle \xi, y \rangle + C} \chi_{P_X}(y)dy \le \int_{P_X} e^{-\frac{C_1}{2}|y| +C_2} dy < +\infty.
    \end{equation*}
    Then, we can use the Lemma \ref{lem varphi bdd} to get that $\varphi$ is finite on $\mathbb{R}^n$.
\end{proof}

Note that formally, the convex function $\vphi$ produced by Corollary~\ref{cor: momentMeasureKRS} satisfies the equation
\[
\begin{aligned}
\det D^2\vphi &=e^{\langle \xi, \nabla \vphi\rangle -\vphi -C}\\
\nabla \vphi(\bR^n)&=P
\end{aligned}
\]
and hence $\vphi$ defines a K\"ahler-Ricci soliton on $\mathbb{R}^n$.  The question is whether this K\"ahler -Ricci soliton extends to the toric manifold $X$ smoothly.  If $X$ is compact, this follows from the regularity theory of \cite{WangZhu2004} \cite{BermanBerndtsson13}.  These results depend in an essential way on the maximum principle on compact K\"ahler manifolds, and hence do not carry over to the non-compact setting.  

Our strategy is to obtain the regularity of $\vphi$ through approximation by the functions $\vphi_{R}= u_{R}^*$ produced by Legendre's result, Theorem~\ref{thm: LegendreKRS}.

\subsection{Approximation by K\"ahler-Ricci solitons}

For $R \geq 1$, we define measures on $\mathbb{R}^n$ by
\[
d\mu_{R} = e^{-\langle \xi_{R}, y\rangle+C_{R}}\, \chi_{P_{R}}(y)dy, \quad \text{ and }  \quad d\mu_{\infty} = e^{-\langle \xi, y\rangle+C_{\infty} } \chi_{P} dy
\]
where:
\begin{itemize}
\item[$(i)$] $\xi_{R}$ is the unique minimizer of ${\rm Vol}_{R}: N_{X,\bR}\rightarrow \mathbb{R}_{>0}$, and $\xi \in \mathcal{R}(P_X)$ is the unique minimizer of ${\rm Vol}: \mathcal{R}(P_X)\rightarrow \mathbb{R}_{>0}$.
\item [$(ii)$] $C_{R}=-\log{\rm Vol}_{R}(\xi_{R})$ and $C_{\infty} = -\log{\rm Vol}(\xi)$ are chosen so that
\[
\int_{\mathbb{R}^n} d\mu_{R} =1= \int_{\mathbb{R}^n} d\mu_{\infty}
\]
\end{itemize}

Our goal is to show that the function $\vphi$ produced by Corollary~\ref{cor: momentMeasureKRS} is the limit as $R\rightarrow +\infty$ of the K\"ahler-Ricci solitons produced by applying Theorem~\ref{thm: LegendreKRS} to the compact polytopes $P_{R}$. This is quite natural, since the K\"ahler-Ricci solitons $\vphi_{R}$ produced by Theorem~\ref{thm: LegendreKRS} are precisely the functions whose moment measure $\mu_{\vphi_{R}} = \mu_{R}$.  The first step is to prove convergence of the minimizing vector fields $\xi_{R}$ to $\xi$, which implies strong convergence of $\mu_{R}$ to $\mu_{\infty}$.

\begin{lem}\label{lem: convergenceOfReebFields}
    As $R\rightarrow +\infty$ we have $\xi_{R}\rightarrow \xi$, where $\xi$ is the unique minimizer of ${\rm Vol}: \mathcal{R}(P_X) \rightarrow \mathbb{R}_{>0}$.  Furthermore, we have the quantitative rate estimate
    \[
   \sup_{P_{R}}\big|e^{-\langle \xi_{R}-\xi, y \rangle}-1|  \le C R^k e^{-R}
    \]
    for some constants $k$ and $C$.
\end{lem}
\begin{proof}
Since $\mathcal{C}\subset N_{X,\bR}$, the variational characterization of $\xi_{R}$ yields
\[
{\rm Vol}_{R}(\xi_{R}) \leq {\rm Vol}_{R}(\xi) = \int_{P_{R}}e^{-\langle \xi, y \rangle}dy \leq  \int_{P_{X}}e^{-\langle \xi, y \rangle}dy = {\rm Vol}(\xi) < +\infty
\]
 Thus ${\rm Vol}_{R}(\xi_{R})$ is uniformly bounded.  Suppose that $|\xi_{R}| = M_{R} \gg 1$.  Consider the cone in the direction of $-\xi_{R}$ defined by 
 \[
 \Upsilon = \bigg\{ y\in M_{X,\bR} : \langle y, \xi_{R}\rangle  \leq -\frac{|y||\xi_{R}|}{2}\bigg\}.
 \]
 Since $0\in P_{R}$ there are $\epsilon,\delta >0$, independent of $R$, such that $ |\Upsilon\cap P_{R}\cap B_{\epsilon}^{c}| > \delta >0$.  Now we have
 \[
 {\rm Vol}(\xi)\geq {\rm Vol}_{R}(\xi_{R}) \geq \int_{\Upsilon \cap P_{R}}e^{-\langle \xi_{R},y \rangle} dy  \geq \int_{\Upsilon \cap P_{R}\cap B_{\epsilon}^c}e^{\frac{\epsilon|\xi_{R}|}{2}} dy \geq \delta e^{\frac{\epsilon|\xi_{R}|}{2}}
 \]
 which proves that $|\xi_{R}|$ is bounded.  Thus we can take a limit $\xi_{R} \rightarrow \xi_{\infty}$. By Fatou's lemma
 \[
 \int_{P_{X}} e^{-\langle \xi_{\infty}, y\rangle} dy \leq \liminf_{R\rightarrow +\infty} {\rm Vol}_{R}(\xi_{R})\leq {\rm Vol}(\xi)<+\infty.
 \]
 This immediately implies that $\xi_{\infty} \in \mathcal{R}(P_X)$. But since $\xi$ is the unique minimizer of ${\rm Vol}:\mathcal{R}(P_X)\rightarrow \bR_{>0}$, this implies $\xi_{\infty}=\xi$.

 We now establish the rate estimate, which will be needed later in the argument.  Consider the functions ${\rm Vol}_{R}: N_{X,\bR}\rightarrow \mathbb{R}_{>0}$, and ${\rm Vol}: \mathcal{R}(P_X)\rightarrow \mathbb{R}_{>0}$.
 Recall that, by definition,  $\xi, \xi_{R}$ are the unique minimizers of ${\rm Vol}$ and ${\rm Vol}_{R}$ respectively.  In particular, we have
 \[
 \int_{P}y e^{-\langle \xi, y \rangle}dy=0, \quad \text{ and } \quad  \int_{P_{R}}y e^{-\langle \xi_{R}, y \rangle}dy=0.
 \]
 Note that
 \[
 \int_{P_{R}}y e^{-\langle \xi, y \rangle}dy= - \int_{P\setminus P_{R}}y e^{-\langle \xi, y \rangle}dy
 \]
 which yields
 \[
 \big| \int_{P_{R}}y e^{-\langle \xi, y \rangle}dy\big| \leq CR^{N}e^{-R}
 \]
 And so
 \[
 |D {\rm Vol}_{R}(\xi)| \leq CR^{N}e^{-R}
 \]
 On the other hand, for any $\eta \in N_{X,\bR}$ we have
 \[
 D^2{\rm Vol}_{R}(\eta) \geq \lambda I
 \]
 for a uniform constant $\lambda$ independent of $R$. By the fundamental theorem of calculus we obtain
 \[
 |\xi_{R}-\xi| \leq \lambda^{-1}|D {\rm Vol}_{R}(\xi)| \leq C\lambda^{-1}R^{N}e^{-R}
 \]
 Now we have
 \[
 \sup_{P_{R}}\big|e^{-\langle \xi_{R}-\xi, y \rangle}-1| \le C R^{N+1} e^{-R}\rightarrow 0
 \]
 and the result is proved.
\end{proof}

By Lemma~\ref{lem: convergenceOfReebFields} we have $\mu_{R} \rightarrow \mu_{\infty}$ in the sense of total variation.  The next several lemmas establish the convergence of $\vphi_{R}$ to $\vphi$.  These results are due to Corder-Erausquin-Klartag \cite{CEK}, but since we need to establish suitable uniformity in the parameter $R$, we recall the proofs for the reader's convenience.

The first lemma we prove is based on \cite[Lemma 14]{CEK}.

\begin{lem}\label{lem: prodIntegralLowBound}
    There exists a constant $C_0>0$ such that, for any $R\geq 1$ we have
    \[
    \left(\int_{\mathbb{R}^n}e^{-v}dy\right)^{1/n}\left(\int_{\mathbb{R}^n}(v -\inf v)\, d\mu_{R} +1\right) \geq C_0
    \]
    where $v:\mathbb{R}^n \rightarrow \mathbb{R}\cup\{+\infty\}$ is any  convex function   with the following properties:
    \begin{itemize}
        \item[(i)] $v$ is $\mu_{R}$ integrable.
        \item[(ii)] $e^{-v}$ is Lebesgue integrable on $\mathbb{R}^n$ and $v(0)=0$.
    \end{itemize}
\end{lem}
\begin{proof}
By the Lemma \ref{lem: convergenceOfReebFields}, we can find $C_1>0$ independent of $R\geq 1$ such that, for any $\theta \in S^{n-1}$ we have
\[
\int_{\mathbb{R}^n} |\langle y, \theta\rangle |\,d\mu_{R}(y) \geq C_1.
\]
We claim that the result holds with
\[
C_0:= \frac{\omega_n^{1/n}}{4e^{1/n}} C_1
\]
where $\omega_n= |B_1|$.

Let $K=\{v \leq 1\}$, and let
\[
\widehat{K}: =K-K = \{ y_0-y_1 : y_0,y_1 \in K\}.
\]
The set $\widehat{K}$ is convex and centrally symmetric in the sense that $y \in \widehat{K}$ implies that $-y \in \widehat{K}$.  Let
\[
\ell:= \sup\{ \rho >0 : B_{\rho}(0) \subset \widehat{K}\}
\]
denote the inner radius of $K$. By The Rogers-Shephard inequality we have
\[
|\widehat{K}| \leq \binom{2n}{n} |K| \leq 4^{n}|K|
\]
and so, if we choose $L>0$ so that
\[
\omega_nL^n=  4^{n}|K|
\]
then $L \geq \ell$.  Let $\theta_0\in S^{n-1}$ be a direction such that $y=\ell\theta_0 \in \del \widehat{K}$, then we have
\[
\sup_{y \in \widehat{K}} \langle y, \theta_0\rangle \leq \ell \leq L
\]
On the other hand, since $\widehat{K}$ is centrally symmetric we conclude that
\[
\sup_{y \in \widehat{K}} |\langle y, \theta_0\rangle| \leq L.
\]
Since $0\in K$ and have that $K\subset \widehat{K}$, and so if $y\in \mathbb{R}^n$ has $|\langle y, \theta_0\rangle|\geq L$ then $y\in \mathbb{R}^n\setminus K$ and so $v(y) \geq 1$. 
For any $y\in \mathbb{R}^n$ with $|y\cdot \theta_0| \geq  L$ we have $y \frac{L}{|y\cdot \theta_0|} \notin K$. Using that $v$ is convex and $v(0)=0$ we deduce that 
\begin{equation}\label{eq: psiLowBoundCone}
1 \leq v\left(\frac{L}{|\langle y, \theta_0\rangle|}y\right) \leq \frac{L}{|\langle y, \theta_0 \rangle|}v(y) \quad \text{ for any $y$ with $|y\cdot \theta_0| \geq  L$}
\end{equation}
Now we claim that
\[
v \geq \frac{|\langle y, \theta_0 \rangle|}{L} + \inf v -1 \quad \text{ for all } y \in \mathbb{R}^n.
\]
To see this, note that since $\inf v \leq 0$ the lower bound holds for any $y$ with $|y\cdot \theta_0| \geq  L$ by~\eqref{eq: psiLowBoundCone}.  On the other hand, if $|y\cdot \theta_0| \leq  L$, then 
\[
\frac{|\langle y, \theta_0 \rangle|}{L} + \inf v -1 \leq \inf v
\]
and so the bound holds trivially.  Thus we conclude that
\[
\int_{\mathbb{R}^n} (v -\inf v +1) \,d\mu_{R} \geq \frac{1}{L}\int_{\mathbb{R}^n}|\langle y, \theta_0 \rangle|\, d\mu_{R}(y) \geq \frac{C_0}{L}.
\]
On the other hand
\[
\int_{\mathbb{R}^{n}}e^{-v} \geq \int_{K}e^{-v} \geq e^{-1}|K| =\frac{ \omega_nL^n}{4^{n}e}
\]
Then 
\[
\left(\int_{\mathbb{R}^{n}}e^{-v} \right)^{1/n}\left(\int_{\mathbb{R}^n} (v -\inf v +1) \,d\mu_{R}\right) \geq \left(\frac{ \omega_n^{1/n}}{4e^{1/n}}\right)C_0
\]
which is the desired bound.
\end{proof}

The next result is based on \cite[Lemma 15]{CEK}

\begin{lem}\label{lem: meanLowerBoundLegendre}
    Fix $R \geq 1$.  Suppose that $v:\mathbb{R}^n\rightarrow \mathbb{R}\cup\{+\infty\}$ convex, $\mu_{R}$-integrable with $v(0)=0$ and $0 \in {\rm int}(\{v<+\infty\})$.  Then we have
    \[
    \int_{\mathbb{R}^n} v\,d\mu_{R} \geq \frac{C_0}{2\pi}\left(\int_{\mathbb{R}^n} e^{-v^*(x)}dx\right)^{1/n} - (n+1)
    \]
    where $C_0$ is the constant in Lemma~\ref{lem: prodIntegralLowBound}.
\end{lem}
\begin{proof}
    Let $\psi= v^*$ for simplicity.  First note that since $v(0)=0$ and $0 \in {\rm int}(\{v<+\infty\})$ we have that $\psi \geq 0$ and $e^{-\psi}$ is Lebesgue integrable on $\mathbb{R}^n.$ we claim that, without loss of generality, we may assume that
    \begin{equation}\label{eq: e-psiBarycenter0}
    \int_{\mathbb{R}^n}xe^{-\psi(x)}dx =0
    \end{equation}
    Indeed if $\widetilde{v}(y) = v(y)-\langle x_0,y\rangle$, then $\tilde{v}(0)=0$ and  $\widetilde{\psi}(x)= \widetilde{v}^*(x)= \psi(x+x_0)$.  Thus
    \[
    \int_{\mathbb{R}^n}e^{-\widetilde{\psi}(x)}dx =\int_{\mathbb{R}^n}e^{-\psi(x)}dx.
    \]
    On the other hand, since $d\mu_{R}$ has barycenter at the origin we have
    \[
    \int_{\mathbb{R}^n}\widetilde{v}(y)\, d\mu_{R}(y) = \int_{\mathbb{R}^n} v\,d\mu_{R}.
    \]
    Thus, by choosing $x_0$ appropriately it suffices to prove the desired inequality assuming~\eqref{eq: e-psiBarycenter0}.  Apply Lemma~\ref{lem: prodIntegralLowBound} and rearranging we have
    \begin{equation}\label{eq: intLowboundPreSantalo}
    \int_{\mathbb{R}^n}v\, d\mu_{R} \geq C_0\left(\int_{\mathbb{R}^n} e^{-v(y)}dy\right)^{-1/n}+\inf v-1
    \end{equation}
    Since the barycenter of $e^{-\psi(x)}dx$ is at the origin we can apply the functional Santal\'o inequality \cite[Theorem 1.3]{AAKM} which yields
    \[
    \left(\int_{\mathbb{R}^n} e^{-v(y)} dy\right)\left(\int_{\mathbb{R}^n} e^{-\psi(x)}dx\right)\leq (2\pi)^n,
    \]
or in other words
\[
 \left(\int_{\mathbb{R}^n} e^{-v(y)} dy\right)^{-1/n} \geq \frac{1}{2\pi}\left(\int_{\mathbb{R}^n} e^{-\psi(x)}dx\right)^{1/n}.
\]
   Substituting this into~\eqref{eq: intLowboundPreSantalo} yields
   \[
    \int_{\mathbb{R}^n}v\, d\mu_{R} \geq \frac{C_0}{2\pi}\left(\int_{\mathbb{R}^n} e^{-\psi(x)}dx\right)^{1/n}+\inf v-1.
   \]
   By Lemma~\ref{lem: entropyEst} we have $\inf v \geq -n$, and the result follows.
\end{proof}

In the proof of Lemma~\ref{lem: meanLowerBoundLegendre} we used the following lemma due to Fradelizi \cite{Frad}; we recall the proof for the readers' convenience.

\begin{lem}\label{lem: entropyEst}
Suppose $\psi:\mathbb{R}^n\rightarrow \mathbb{R}_{\geq 0}\cup\{+\infty\}$ is proper, convex, $\inf \psi =0$ and the measure $e^{-\psi(x)}dx$ has barycenter at the origin.  Then $\inf_{\mathbb{R}^n} \psi^* \geq -n$.
\end{lem}
\begin{proof}
    We write
    \[
    \begin{aligned}
    -\psi^*(y) = \inf_{x\in \mathbb{R}^n}(\psi(x)-\langle x,y\rangle ) &\leq \frac{1}{\int_{\mathbb{R}^n}e^{-\psi(x)}dx} \int_{\mathbb{R}^n}(\psi(x)-\langle x,y \rangle) e^{-\psi(x) dx}\\
    &=  \frac{1}{\int_{\mathbb{R}^n}e^{-\psi(x)}dx} \int_{\mathbb{R}^n}\psi(x)e^{-\psi(x)} dx
    \end{aligned}
    \]
    where we used the barycenter condition in the final equality.  Thus it suffices to prove
    \[
   \int_{\mathbb{R}^n}\psi(x)e^{-\psi(x)} dx \leq n\int_{\mathbb{R}^n}e^{-\psi(x)}dx
    \]
    Choose $x_0$ such that $\psi(x_0)=0$.  Then we have
    \[
    \psi((1-t)x_0+tx) \leq t\psi(x)
    \]
    and so if $z= x_0+t(x-x_0)$ we have, for all $t\in [0,1]$
    \[
    \begin{aligned}
    \int_{\mathbb{R}^n}e^{-\psi(z)}dz &= \int_{\mathbb{R}^n}e^{-\psi(x_0+t(x-x_0))}t^ndx\\
    & \geq t^{n}\int_{\mathbb{R}^n}e^{-t\psi(x)}dx.
    \end{aligned}
    \]
    We conclude from this inequality that function $f(t)=t^{n}\int_{\mathbb{R}^n}e^{-t\psi(x)}dx$ is clearly increasing as $t\rightarrow 1^{-}$, and so
    \[
    0 \leq n -\frac{1}{\int_{\mathbb{R}^n}e^{-\psi(x)}dx} \int_{\mathbb{R}^n}\psi(x)e^{-\psi(x)}
    \]
    which is the desired inequality.
\end{proof}

The next result is based on \cite[Lemma 16]{CEK}.

\begin{lem}\label{lem: pointwiseToLocalIntegral}
Let $y_0 \in {\rm int}(P)$.  Then, there exists constants $R_0, C_{y_0}$ depending on $y_0$ with the following effect: for all $R \geq R_0$ and for any non-negative convex function $v:\mathbb{R}^n\rightarrow \mathbb{R}_{\geq 0}\cup\{+\infty\}$ that is $\mu_{R}$-integrable we have
\[
v(y_0)\leq C_{y_0}\int_{P_{R}}v \, d\mu_{R}
\]
\end{lem}
\begin{proof}
    Let $R_0=R_0(y_0)$ be chosen so that $y_0\in P_{R_0-1}\setminus P_{R_0-2}$.
    \[
    m(y_0) = \inf_{R \geq R_0} \inf_{\theta \in S^{n-1}}\mu_{R}(\{y\in \mathbb{R}^n : (y-y_0)\cdot \theta\geq 0\})
    \]
    Since the measures $\mu_{R}, \mu$ lie in a bounded family of strictly positive measures on $P$, we know that $m(y_0)>0$.  So if $x_0$ is a sub-differential of $v$ at $y_0$ we have
    \[
    v(y) \geq v(y_0) + \langle x_0, (y-y_0)\rangle.
    \]
    If $x_0=0$ then the result is clear so we may assume $x_0\ne 0$. Since $v$ is non-negative we have
    \[
    \begin{aligned}
    \int_{\mathbb{R}^n}v(y)\, d\mu_{R} &\geq \int_{\left\{y\in \mathbb{R}^n : (y-y_0)\cdot \frac{x_0}{|x_0|}\geq 0\right\}}v(y)\, d\mu_{R}\\
    & \geq v(y_0)m(y_0)
    \end{aligned}
    \]
    and the result follows by taking $C_{y_0} = \frac{1}{m(y_0)}$.
\end{proof}

Let $\vphi_{R}:\mathbb{R}^n\rightarrow \bR$ be the convex functions satisfying $\mu_{\vphi_{R}} = \mu_{R}$.  By translating we may assume that $\vphi_{R}(0)= \inf_{\mathbb{R}^n} \vphi_{R}$. Under this normalization we have
\[
\inf u_{R} = -\varphi_{R}(0) = -\inf_{\mathbb{R}^n}\varphi_{R} = u_{R}(0)
\]

\begin{lem}\label{lem: infupperandlowerbds}
    Suppose $\vphi_{R}:\mathbb{R}^n\rightarrow \bR$ is convex, satisfies $\inf_{\mathbb{R}^n}\vphi_{R}=\vphi_{R}(0)$, $\int_{\bR^n}e^{-\vphi_{R}} =1$ and $\vphi_{R}$ has moment measure $\mu_{\vphi_{R}} = \mu_{R}$.  Then there exists a constant $C$, independent of $R$ such that
    \[
    -C\leq \vphi_{R}(0) \leq C
    \]
\end{lem}
\begin{proof}
We first prove the lower bound. Since $\vphi_{R}$ has moment measure $\mu_{R}$, which is supported on $P_{R}$, we have
\[
\vphi_{R} \leq \vphi_{R}(0) +\phi_{P_{R}}
\]
where $\phi_{P_{R}}(x) = \sup_{y\in P_{R}}\langle x,y\rangle$ is the convex support function of $P_{R}$.  Now since $\int_{\mathbb{R}^n}e^{-\vphi_{R}} =1$ we have
\[
1=\int_{\mathbb{R}^n}e^{-\vphi_{R}}  \geq e^{-\varphi_{R}(0)}\int_{\mathbb{R}^n}e^{-\phi_{P_{R}}} = c(n)e^{-\varphi_{R}(0)}|P_{R}^{\circ}|
\]
where $P_{R}^{\circ}$ denotes the polar dual of $P_{R}$. Since $P_{R}$ Hausdorff converges to $P_{X}$ on compact sets, and $P_{X}$ does not contain a line, we have $|P_{R}^{\circ}| \geq c' >0 $ for a constant $c'>0$ independent of $R$.  This proves the lower bound. 

For the upper bound, fix a $\mu$-integrable function $f_0: \mathbb{R}^n \rightarrow \mathbb{R} \cup\{+\infty\}$.  Then $f_0$ is also $\mu_{R}$ integrable for all $R$.  Since $u_{R}= \vphi_{R}^*$ maximizes the functional $I_{\mu_{R}}$ defined in~\eqref{eq: Ifunctional}, we have
\begin{equation}\label{eq: integraluRbd}
-\int_{\mathbb{R}^n} u_{R}d\mu_{R} =I_{\mu_{R}}(u_{R}) \geq I_{\mu_{R}}(f_0) = \log \int_{\mathbb{R}^n} e^{-f_0^*(x)}dx - \int_{\mathbb{R}^n}f_0d\mu_{R} \geq -C
\end{equation}
In particular, we have $\int_{\mathbb{R}^n} u_{R} d\mu_{R} \leq C$ independent of $R$.  We now apply Lemma~\ref{lem: meanLowerBoundLegendre} to $u_{R}-\inf u_{R}$.  Indeed we have
\[
u_{R}(0)= -\inf \vphi_{R}= -\vphi_{R}(0)= \inf u_{R}
\]
by our normalization.  So $u_{R}-u_{R}(0) \geq 0$ and we have
\[
\begin{aligned}
\int_{\mathbb{R}^n}u_{R}d\mu_{R} -u_{R}(0)&\geq \frac{C_0}{2\pi}\left(\int_{\mathbb{R}^n} e^{-(\vphi_{R}(x)-\vphi_{R}(0))}\right)^{1/n}-(n+1) \\
&=  \frac{C_0}{2\pi}e^{\frac{\vphi_{R}(0)}{n}}-(n+1).
\end{aligned}
\]
Since $u_{R}(0)= -\vphi_{R}(0)$, we obtain
\[
C \geq \int_{\mathbb{R}^n}u_{R}d\mu_{R} \geq -\vphi_{R}(0)+\frac{C_0}{2\pi}e^{\frac{\vphi_{R}(0)}{n}}-(n+1)
\]
which easily yields the upper bound for $\phi_{R}(0)$.
\end{proof}

\begin{lem}\label{lem: uRholderbounds}
For any $R_0 \geq 1$ and any $\alpha \in (0,1)$, there exists a constant $C= C(R_0,\alpha, P_{X})$ such that, for all $R$ sufficiently large depending on $R_0$ we have
\[
\|u_{R}\|_{C^{\alpha}(P_{R_0})} \leq C
\]
\end{lem}
\begin{proof}
  We apply Lemma~\ref{lem: pointwiseToLocalIntegral} to $u_{R}-u_{R}(0)$. Take a small positive constant $\epsilon$.  We see that for all $R \geq R_0$ sufficiently large depending on $R_0$ there is a constant $C_{0}$ so that for any $y \in B_{\epsilon}(0)$,
  \[
  u_{R}(y)-u_{R}(0) \leq C_{0}\int_{P_{R}}(u_{R}-u_{R}(0)) d\mu_{R}
  \]
  By Lemma~\ref{lem: infupperandlowerbds} we have $-C\leq u_{R}(0)= \inf u_{R} \leq C$ for a constant $C$ independent of $R$. Thanks to ~\eqref{eq: integraluRbd} we have $\int_{P_{R}}u_{R} d\mu_{R} \leq C$, and so
  \[
  \|u_{R}\|_{L^{\infty}(B_{\epsilon}(0))} \leq C_{1},
  \]
 for a constant $C_1$ independent of $R$.  Fix $\epsilon >0$ so that $B_{\epsilon}(0) \subset P_{R}.$ Then
  \begin{equation}\label{eq: uniformCoercivityphi}
  \vphi_{R}(x) = \sup_{y\in \bR^{n}} \langle x,y \rangle - u_{R}(y) \geq  \sup_{y\in B_{\epsilon}(0)} \langle x,y \rangle - u_{R}(y)  \geq \epsilon|x|-C_1.
  \end{equation}
   Now using the moment measure equation we have
  \[
  \begin{aligned}
  \int_{P_{R}}|\nabla u_{R}(y)|^{p} \frac{e^{-\langle \xi_{R},y\rangle}}{{\rm Vol}_{R}(\xi_{R})}dy &= \int_{\mathbb{R}^n}|x|^p e^{-\varphi_{R}(x)}dx\\
  &\leq C \int_{\mathbb{R}^n}|x|^p e^{-\epsilon|x|} dx \leq C(R_0, p, \epsilon)
  \end{aligned}
  \]
  On the other hand, by Lemma~\ref{lem: convergenceOfReebFields} we have $\frac{e^{-\langle \xi_{R},y\rangle}}{{\rm Vol}_{R}(\xi_{R})} \geq C^{-1}$ for $y \in P_{R_0}$, where $C$ is a constant depending on $R_0$, so we conclude that for all $R \geq R_0$ sufficiently large we have
  \[
\int_{P_{R_0}}|\nabla u_{R}(y)|^{p}dy \leq C(R_0, P_{X},p).
  \]
  By the Sobolev embedding theorem we conclude that, for any $\alpha \in (0,1)$ and all $R \geq R_0$ sufficiently large we have
  \[
  \|u_{R}\|_{C^{\alpha}(P_{R_0})} \leq C(R_0, P_{X}, \alpha)
  \]
\end{proof}

\begin{cor}\label{cor hol con}
    Fix $\alpha \in (0,1)$ and let $R_{k}\rightarrow +\infty$.  There exists a subsequence $R_{k_{\ell}}$ such that $u_{R_{k_{\ell}}}$ converges in $C^{\alpha}_{loc}(P)$ to $u_{\infty}$, which is a maximizer of $I_{\mu}$.
\end{cor}
\begin{proof}
By Lemma~\ref{lem: uRholderbounds} and a standard diagonal argument we can find a subsequence $R_{k_{\ell}}$ such that $u_{R_{k_{\ell}}}$ converges in $C^{\alpha}_{loc}(P)$ to $u_{\infty}$.  Let us relabel $R_{k_{\ell}}\to R_{k}$ to simplify notation.  We need to show that $u_{\infty}$ maximizes $I_{\mu}$.  
By Fatou's Lemma we get
\[
\int_{\bR^n}u_{\infty} d\mu \leq \liminf_{k\rightarrow \infty}\int_{\bR^n}u_{R_k} d\mu_{R_{k}} \leq C
\]
and hence $u_{\infty}$ is $\mu$-integrable.  Furthermore, since  $d\mu_{R}, d\mu$ are probability measures we get
\[
\int_{\bR^n}u_{\infty}d\mu \leq \liminf_{k\rightarrow \infty}\int_{\bR^n}u_{R_k}d\mu_{R_{k}}
\]
Let $\vphi_{\infty}= u_{\infty}^*$.  We claim that $\int e^{-\vphi_{\infty}}\geq 1$.  To see this we write
\[
\varphi_{R_k}(x) =\sup_{y\in \bR^n}\, \langle x,y \rangle -u_{R_k}(y) \geq \langle x,y_0 \rangle -u_{R_k}(y_0)
\]
for any $y_0\in \bR^n$.  Taking the $\liminf$ over $k$ and using that $u_{R_k}(y_0)$ converges pointwise to $u_{\infty}(y_0)$ yields
\[
\liminf_{k\rightarrow \infty} \varphi_{R_{k}}(x) \geq \langle x,y_0 \rangle -u_{\infty}(y_0)
\]
for any $y_0\in \bR^n$.  This implies $\liminf_{k\rightarrow \infty} \varphi_{R_{k}}(x) \geq \varphi_{\infty}(x)$.  Thanks to~\eqref{eq: uniformCoercivityphi} we can apply Fatou's lemma to obtain
\[
1= \limsup_{k\rightarrow \infty}\int_{\bR^n} e^{-\vphi_{R_k}} \leq \int_{\bR^n} e^{-\varphi_{\infty}}.
\]
We conclude that
\begin{equation}\label{e i mu u inf}
I_{\mu}(u_{\infty}) = \log\int_{\bR^n} e^{-\varphi_{\infty}}-\int_{\bR^n}u_{\infty} d\mu \geq \limsup_{k\rightarrow \infty} I_{\mu_{R_k}}(u_{R_{k}})
\end{equation}
To show that $u_{\infty}$ is a maximum point of $I_{\mu}$, we need to show that $I_{\mu}(u_{\infty}) \ge I_{\mu}(u)$. Here $u=\varphi^*$, where $\varphi$ is given in the Corollary \ref{cor: momentMeasureKRS}.
Since $u_{R_k}$ maximizes $I_{\mu_{R_k}}$ we have
\begin{equation}\label{e urk low}
\limsup_{k\rightarrow \infty} I_{\mu_{R_k}}(u_{R_{k}}) \geq \limsup_{k\rightarrow \infty} I_{\mu_{R_k}}(u) =\limsup_{k \rightarrow \infty} \log \int_{\R^n} e^{- u^*} dy - \int_{P_{R_k}} u d\mu_{R_k}.
\end{equation}
Using the Lemma \ref{lem: convergenceOfReebFields} and the integrability of $u$ with respect to $d\mu$, we have that:
\begin{equation}\label{e limuk u}
   \lim_{k \rightarrow \infty} \int_{P_{R_k}} u d\mu_{R_k} = \int_P u d\mu.
\end{equation}
Combining (\ref{e i mu u inf}), (\ref{e urk low}) and (\ref{e limuk u}), we can get:
\begin{equation*}
    I_{\mu}(u_{\infty})\ge I_\mu(u)
\end{equation*}
Since $u$ maximizes $I_{\mu}$, it follows that $u_{\infty}$ maximizes $I_{\mu}$ and so, by \cite{CEK}, $\varphi_{\infty}= u_{\infty}^*$ has moment measure $\mu$.
\end{proof}

We now establish higher order regularity and convergence.

\begin{lem}\label{lem ur con}
    We have convergence in $C^{k,\alpha}_{loc}(P)$ for all $k$ and similarly on $\bR^n$
\end{lem}
\begin{proof}
From the moment measure construction, $\nabla \varphi_{R}$ (resp. $\nabla \varphi_{\infty}$) defines an optimal transport map, in the sense of Brenier-McCann \cite{Brenier1991,McCann1995} from $(\bR^n, e^{-\varphi_{R}}dx)$ to $(P_{R}, d\mu_{R})$ (resp. from $(\bR^n, e^{-\varphi_{\infty}}dx)$ to $(P_{X}, d\mu)$).   Caffarelli's interior regularity theory for optimal transport maps \cite{Caffarelli1992Mappings} has been adapted to this setting by Alesker-Dar-Milman \cite{ADM} and Cordero-Erausquin-Figalli \cite{CEF}.  It follows from a simple bootstrap that:
\begin{itemize}
    \item $\vphi_{R}, \vphi_{\infty} \in C^{k,\alpha}_{loc}(\mathbb{R}^n)$ for all $k$, 
    \item $\nabla \varphi_{R} : \bR^n \rightarrow P_{R}$ and $\nabla \varphi_{\infty}: \bR^n \rightarrow P_{X}$ are smooth diffeomorphisms,
    \item  $u_{R} \in C^{k,\alpha}_{loc}(P_{R})$ and $u_{R}$ are strictly convex.
    \item $\vphi_{\infty}^*= u_{\infty} \in C^{k,\alpha}_{loc}(P)$ for all $k$.
\end{itemize}

The only thing to prove is that we have convergence in similarly strong norms.  Recall that if $u$ is convex and differentiable at $p$, then the section of height $h$ at $p$ is given by
\[
S_{h}(u,p) = \{ y \in \bR^n: u(y) \leq u(p) + (y-p)\cdot \nabla u(p) +h \}
\]
Fix compact, convex sets
\[
K_0 \Subset K_1 \Subset K_2 \Subset {\rm int}P_{X}
\]
Choose $R_0$ large so that $K_{2}\Subset P_{R_0}$.  We consider $R \geq R_0$.  Since $u_{R}\rightarrow u_{\infty}$ in $C^{\alpha}(K_2)$, we have
\[
{\rm osc}_{K_{2}} u_{R} \leq C_{K_2}
\]
It follows from convexity that
\[
\sup_{K_1}|\nabla u_{R}| \leq C(K_2, K_1)
\]
Since $u_{R}$ is smooth and strictly convex the moment measure equation gives
\begin{equation}\label{eq: bootstrapEquation}
\det D^2u_{R}= e^{\langle y, \nabla u_R(y)\rangle-u_{R}(y)-\langle \xi_{R}, y\rangle + C_R}
\end{equation}
By the gradient bound and Lemma~\ref{lem: convergenceOfReebFields} the right hand side is uniformly bounded above and below on $K_1$. Fix $y\in K_0$.  Since $u_{\infty}$ is smooth and strictly convex we can find $\rho_1, \rho_2$ so that
\[
B_{\rho_1}(y) \subset S_{\frac{h}{2}}(u_{\infty}, y) \subset S_{2h}(u_{\infty},y) \subset B_{\rho_2}(y) \subset K_1
\]
Since $u_{R}, u_{\infty}$ are smooth, and convex and $u_{R}$ converges uniformly to $u_{\infty}$ on $K_2$, it is straightforward to show that $\nabla u_{R}(y)$ converges to $\nabla u_{\infty}(y)$.  Thus for $R$ large we have
\[
B_{\rho_1}(y)\subset S_{\frac{3}{4}h}(u_{R}, y) \subset S_{h}(u_{R}, y) \subset B_{\rho_2}(y) \subset K_1
\]
Now the classical $C^{1,\alpha}$-regularity of Caffarelli \cite{Caffarelli1991Regularity} implies that
\[
\|u_{R}\|_{C^{1,\alpha}(B_{\rho_1}(y))} \leq C(K_2,K_1,K_0).
\]
With this estimate~\eqref{eq: bootstrapEquation} has $C^{\alpha}$ right-hand side, and so Caffarelli's interior $C^{2,\alpha}$ regularity \cite{Caffarelli1990W2p} yields
\[
\|u_{R}\|_{C^{2,\alpha}(B_{\frac{\rho_1}{2}}(y))} \leq C(K_2,K_1,K_0).
\]
All higher order estimates follow from the Schauder theory.  Thus we obtain the $C^{k,\alpha}_{loc}(P)$ convergence of $u_{R}$ to $u_{\infty}$ for all $k$.  The $C^{k,\alpha}_{loc}(\bR^n)$ convergence of $\varphi_{R}$ to $\varphi_{\infty}= u_{\infty}^*$ follows easily from this and the strict convexity of $u_{\infty}$. 
\end{proof}

\section{Completeness}\label{sec: completeness}

For any vector $v\in \R^n$, we denote $u^{vv}= v^{\top} \cdot (u^{ij}) \cdot v$.
\begin{prop}\label{prop: linGrowthImplesComplete}
    Suppose that there exists a constant $C$ such that \begin{equation}\label{e uxx}
    u^{\xi \xi}\le C(\langle \xi, y\rangle+C).    
    \end{equation}
    Then the corresponding K\"ahler-Ricci soliton is complete.
\end{prop}
\begin{proof}
   In order to prove $(M,g)$ is complete, it suffices to construct a proper function $\rho$ such that $|\nabla \rho|_g \le C$ for some constant $C$. In this case, we take $\rho=\sqrt{\langle\xi,y\rangle+C}$. Then we can compute that:
   \begin{equation*}
       |\nabla \rho|_g^2 = (\rho_\xi)^2 u^{\xi \xi}= \frac{u^{\xi \xi}}{4(\langle\xi,y\rangle+C)}\le C/4
   \end{equation*}
   by our assumption.
\end{proof}

Here we need a Theorem by Minkowski-Weyl:

\begin{lem}
    For any closed nonempty polytope $P$, there exists a compact closed polytope $Q$ such that:
    \begin{equation*}
        P=Q+ rec(P).
    \end{equation*}
\end{lem}

\begin{lem}\label{lem uni c}
    There exists a constant $C_1$ such that for any $R$, $P_R \subset \{z: \langle\xi_R,z\rangle\ge C_1\}$.
\end{lem}
\begin{proof}
    Let $C_0$ be a constant such that $\{z: \langle z,\xi\rangle=C_0\}$ is a supporting plane of $\partial P$. We claim that $\{z: \langle z,\xi\rangle=C_0\}\cap P$ is a compact set. In fact, if $\{z: \langle z,\xi \rangle=C_0\}\cap P$ is not a compact set, then it must contain a ray $\{y_0+tv: t \ge 0\}$ with $\langle v,\xi \rangle=0$. Then, we can take a point $y_1 \in \inte P$ and a $(n-1)$-dimensional disc $D_{\epsilon}$ centered around $y_1$ such that for any $y\in D_{\epsilon}$, $\langle y-y_1, v \rangle =0$. Using the convexity of $P$, we know that for any $y\in D_{\epsilon}$ and $s\ge 0$, 
    \begin{equation*}
        y+ sv = \lim_{\lambda \rightarrow 0} (1-\lambda) y + \lambda (y_0 + \frac{s}{\lambda}v) \in \overline{P}.
    \end{equation*}
    $\bar{P}$ contains an infinite cylinder $\{y+sv: y \in D_{\epsilon}, s\ge 0 \}$. The integral of $e^{-\langle\xi,y\rangle}$ over this cylinder is infinite. This contradicts the fact that  $\int_P e^{-\langle\xi,y\rangle}dy$ is bounded. This concludes the proof of the claim. 
    
    Using the claim, we can see that when we perturb $\xi$, the corresponding supporting plane does not intersect with the vertices of $P$ which are not in $\{z: \langle z,\xi\rangle=C_0\}\cap P$. As a result, the supporting plane would move continuously when we perturb $\xi$. Then, we can use the fact that $\lim_{R \rightarrow \infty}\xi_R =\xi$ to find the constant $C_1$ as is required in this lemma.
\end{proof}

We want to first establish (\ref{e uxx}) for the approximation sequence we defined before:
\begin{prop}\label{prop: linGrowthApprox}
    There exists a constant $C$ independent of $R$ such that for any $R \ge 1$, we have that:
    \begin{equation}
    u_R^{\xi_R \xi_R}\le C(\langle\xi_R, y\rangle+C).    
    \end{equation}
\end{prop}
\begin{proof}
    Note that $\varphi_R$ satisfies the equation:
    \begin{equation*}
        det(\varphi_R)_{ij}= e^{-\varphi_R + \langle\xi_R, \nabla \varphi_R\rangle-C_R}.
    \end{equation*}
    Take logarithm of the above equation and take derivative with respect to $\xi_R$ direction. We can get that:
    \begin{equation}\label{e 1deri}
        (\varphi_R)^{ij} (\varphi_R)_{ij \xi_R}= - (\varphi_R)_{\xi_R} +\langle\xi_R, \nabla (\varphi_R)_{\xi_R}\rangle.
    \end{equation}
    We define the drifted linearized Monge-Amp\`ere operator as
    \begin{equation*}
        L u = \varphi_R^{ij} u_{ij} - \langle\xi_R, \nabla u\rangle.
    \end{equation*}
    Then (\ref{e 1deri}) becomes 
    \begin{equation*}
        L (\varphi_R)_{\xi_R} = -(\varphi_R)_{\xi_R}.
    \end{equation*}
    Take one more derivative of (\ref{e 1deri}) with respect to $x$. We can get that:
    \begin{equation*}
        (\varphi_R)^{ij}(\varphi_R)_{ij \xi_R \xi_R}-(\varphi_R)^{ia}(\varphi_R)^{jb}(\varphi_R)_{ij \xi_R}(\varphi_R)_{ab \xi_R} =-  (\varphi_R)_{\xi_R \xi_R} + \langle\xi_R, \nabla (\varphi_R)_{\xi_R \xi_R}\rangle,
    \end{equation*}
    which can be written as 
    \begin{equation*}
        L (\varphi_R)_{\xi_R \xi_R}= (\varphi_R)^{ia}(\varphi_R)^{jb}(\varphi_R)_{ij \xi_R} (\varphi_R)_{ab \xi_R} - (\varphi_R)_{\xi_R \xi_R}.
    \end{equation*}
    Then we can compute that
    \begin{equation}\label{e Lphi11}
    \begin{split}
           L \log (\varphi_R)_{\xi_R \xi_R}&= \frac{L (\varphi_R)_{\xi_R \xi_R}}{(\varphi_R)_{\xi_R \xi_R}} - \frac{(\varphi_R)^{ij}(\varphi_R)_{\xi_R \xi_R i}(\varphi_R)_{\xi_R \xi_R j}}{(\varphi_R)_{\xi_R \xi_R}^2}  \\
           &= \frac{ (\varphi_R)^{ia}(\varphi_R)^{jb}(\varphi_R)_{ij\xi_R} (\varphi_R)_{ab\xi_R} - (\varphi_R)_{\xi_R \xi_R}}{(\varphi_R)_{\xi_R \xi_R}} - \frac{(\varphi_R)^{ij}(\varphi_R)_{\xi_R \xi_Ri}(\varphi_R)_{\xi_R \xi_Rj}}{(\varphi_R)_{\xi_R \xi_R}^2} \\
           & \ge \frac{(\varphi_R)^{ij}(\varphi_R)_{\xi_R \xi_Ri}(\varphi_R)_{\xi_R \xi_Rj}}{(\varphi_R)^2_{\xi_R \xi_R}} -1 - \frac{(\varphi_R)^{ij}(\varphi_R)_{\xi_R \xi_Ri}(\varphi_R)_{\xi_R \xi_Rj}}{(\varphi_R)_{\xi_R \xi_R}^2}
           = -1.
    \end{split}
    \end{equation}
    In the last line above, we use $(\varphi_R)^{ia}(\varphi_R)^{jb} (\varphi_R)_{ij\xi_R} (\varphi_R)_{ab \xi_R} \ge \frac{(\varphi_R)^{ij}(\varphi_R)_{\xi_R \xi_R i} (\varphi_R)_{\xi_R \xi_R  i}}{H_{\xi_R \xi_R}}$ which comes from the Cauchy-Schwarz inequality.

    Let $C_1$ be the constant as in the Lemma \ref{lem uni c}. We additionally require that $C_2$ is big enough such that $C_2 \ge 1-C_1$. This ensures that for any $z\in P_R$, we have that:
\begin{equation*}
    \langle \xi_R,y \rangle + C_2  \ge C_1 + C_2  \ge 1. 
\end{equation*}
    Then, we compute that
    \begin{equation}\label{e Lx}
    \begin{split}
         L \log (\langle\xi_R,y\rangle+C_2) &= \frac{L (\langle\xi_R,y\rangle+C_2)}{(\langle\xi_R,y\rangle+C_2)} -\frac{\varphi_R^{ij}(\langle\xi_R,y\rangle+C_2)_i (\langle\xi_R,y\rangle+C_2)_j}{(\langle\xi_R,y\rangle+C_2)^2} \\
         &= \frac{L ((\varphi_R)_{\xi_R})}{(\langle\xi_R,y\rangle+C_2)} -\frac{(\varphi_R)^{ij}(\varphi_R)_{i \xi_R} (\varphi_R)_{j \xi_R }}{(\langle\xi_R,y\rangle+C_2)^2} \\
         & =-\frac{(\varphi_R)_{\xi_R}}{C_2+(\varphi_R)_{\xi_R}} -\frac{(\varphi_R)_{\xi_R \xi_R}}{(C_2+(\varphi_R)_{\xi_R})^2}.
    \end{split}
    \end{equation}
Then, we combine (\ref{e Lphi11}) and (\ref{e Lx}) to get that:
\begin{equation}\label{e Lphi11x}
 \begin{split}
      L \Big( \log (\frac{(\varphi_R)_{\xi_R \xi_R }}{(\langle\xi_R,y\rangle+C_2)}) -\epsilon \varphi_R \Big) &\ge -1 +\frac{(\varphi_R)_{\xi_R}}{C_2+(\varphi_R)_{\xi_R}} +\frac{(\varphi_R)_{\xi_R \xi_R}}{(C_2+(\varphi_R)_{\xi_R})^2} -n\epsilon +\epsilon (\varphi_R)_{\xi_R}\\
     & =-\frac{C_2}{C_2+(\varphi_R)_{\xi_R}} +\frac{(\varphi_R)_{\xi_R \xi_R}}{(C_2+(\varphi_R)_{\xi_R})^2} +\epsilon ( (\varphi_R)_{\xi_R}-n).
 \end{split}   
\end{equation}
Using the Lemma \ref{lem uni c}, we have that $P_R \subset \{z: \langle\xi_R,z\rangle\ge C_1\}$. Since $\nabla \varphi_R \subset P_R$, we have that:
\begin{equation}\label{e phir low}
(\varphi_R)_{\xi_R}=\langle\nabla \varphi_R, \xi_R\rangle\ge C_1.
\end{equation}
Plugging (\ref{e phir low}) into (\ref{e Lphi11x}), we can get that
\begin{equation}\label{e phi11 up1}
\begin{split}
    (\varphi_R)_{\xi_R \xi_R}\le& C_2 ((\varphi_R)_{\xi_R} +C_2) + \epsilon(n-(\varphi_R)_{\xi_R})_+ ((\varphi_R)_{\xi_R} +C_2)^2\\
    &+((\varphi_R)_{\xi_R}+C_2)^2L \Big( \log (\frac{(\varphi_R)_{\xi_R \xi_R}}{(\langle\xi_R,y\rangle+C_2)}) -\epsilon \varphi_R \Big).
    \end{split}
\end{equation}
From now on, we fix the constant $C$ by taking $C=C_2$. By the Proposition \ref{prop H gui}, we have that $\log (\frac{u_R^{\xi_R \xi_R}}{(C_2 + \langle\xi_R,y\rangle)})$ is bounded from above. Since $-\epsilon \varphi_R \circ \nabla u_R \rightarrow-\infty$ near $\partial P_R$, we have that 
\begin{equation*}
    \log (\frac{u_R^{\xi_R \xi_R}}{(C_2+ \langle\xi_R,y\rangle)})- \epsilon \varphi_R \circ \nabla u_R \rightarrow -\infty
\end{equation*}
near $\partial P_R$. As a result, the maximum point of $ \log (\frac{u_R^{\xi_R \xi_R}}{(C_2+ \langle\xi_R,y\rangle)})- \epsilon \varphi_R \circ \nabla u_R $ is taken at some point $p\in Int (P_R)$. At $\nabla u_R (p)$, (\ref{e phi11 up1}) implies that:
\begin{equation}\label{e phi11 uper p}
\begin{split}
    (\varphi_R)_{\xi_R \xi_R}\le &C_2 ((\varphi_R)_{\xi_R} +C_2) + \epsilon(n-(\varphi_R)_{\xi_R})_+ ((\varphi_R)_{\xi_R} +C_2)^2\\
    \le & C_2 ((\varphi_R)_{\xi_R} +C_2) + \epsilon(n-C_1) (n +C_2)^2\\
    \le & C_3 ((\varphi_R)_{\xi_R} +C_2).
    \end{split}
\end{equation}
Then we have that for any $ p_1 \in P_R$:
\begin{equation*}
\begin{split}
&\log (\frac{u_R^{\xi_R \xi_R}}{(C_2 + \langle\xi_R,y\rangle)}) (p_1)=  \log (\frac{u_R^{\xi_R \xi_R}}{(C_2 + \langle\xi_R,y\rangle)}) (p_1) - \epsilon \varphi_R (p_1) + \epsilon \varphi_R (p_1) \\
& \le \log (\frac{u_R^{\xi_R \xi_R}}{(C_2 + \langle\xi_R,y\rangle)}) (p)- \epsilon \varphi_R (p) + \epsilon \varphi_R (p_1) \le \log C_3 - \epsilon C_4 + \epsilon \varphi_R (p_1)
\end{split}
\end{equation*}
$C_4$ is a lower bound on $\varphi_R$ which exists since $\varphi_R$ is proper.  In the end, we let $\epsilon$ go to zero  to get that 
\begin{equation*}
    \log (\frac{u_R^{\xi_R \xi_R}}{(C_2 + \langle\xi_R,y\rangle)}) (p_1) \le \log C_3
\end{equation*}
for any $p_1\in P$. In the end, we take $C= \max\{C_2,C_3\}$ such that:
\begin{equation*}
 u_R^{\xi_R \xi_R} \le C(C + \langle\xi_R,y\rangle)
\end{equation*}
\end{proof}

\section{$C^2$ estimate on the $\R^n$ side}\label{sec: C2estimateRn}

\begin{prop}\label{prop gra loc}
    For any $M>0$, there exist constants $R_M >M$ and $\epsilon_M>0$ such that for any $R \ge R_M$,  if $\rho_0$ satisfies $\nabla \varphi_{R} (\rho_0) \in P_M$, then we have that:
    \begin{equation*}
        \nabla \varphi_R (\{\rho : \varphi_R (\rho)-\varphi_R(\rho_0) - \langle \nabla \varphi_R (\rho_0), \rho- \rho_0 \rangle <\epsilon_M \}) \subset P_{R_M}.
    \end{equation*}
\end{prop}
\begin{proof}
    We prove this Proposition by contradiction argument. Suppose that there exist $M$ and sequences $R_k \rightarrow \infty$, $ \frac{1}{k} \rightarrow 0$, $\widetilde{R}_k \ge R_k$ and $\rho_k, \bar{\rho}_k \in \R^n$, such that:
    $\nabla \varphi_{\widetilde{R}_k} (\rho_k) \subset P_M$ and
    \begin{equation}\label{e yk big}
        \nabla \varphi_{\widetilde{R}_k} (\bar{\rho}_k) \subset P \setminus P_{R_k}
    \end{equation}
    and
    \begin{equation}\label{e sec con}
        \varphi_{\widetilde{R}_k}(\bar{\rho}_k)-\varphi_{\widetilde{R}_k}(\rho_k) -\langle \nabla \varphi_{\widetilde{R}_k}(\rho_k),\bar{\rho}_k-\rho_k \rangle < \frac{1}{k}.
    \end{equation}
    Denote $y_k = \nabla \varphi_{\widetilde{R}_k}(\rho_k)$ and $\bar{y}_k =\nabla \varphi_{\widetilde{R}_k} (\bar{\rho}_k)$. (\ref{e sec con}) implies that:
    \begin{equation*}
        \bar{y}_k \cdot \nabla u_{\widetilde{R}_k} (\bar{y}_k) - u_{\widetilde{R}_k}(\bar{y}_k) - (y_k \cdot \nabla u_{\widetilde{R}_k}(y_k) - u_{\widetilde{R}_k}(y_k))-y_k \cdot (\nabla u_{\widetilde{R}_k} (\bar{y}_k) - \nabla u_{\widetilde{R}_k} (y_k))<\frac 1k.
    \end{equation*}
    This implies that:
    \begin{equation}\label{e ur conv}
        u_{\widetilde{R}_k}(y_k) - u_{\widetilde{R}_k}(\bar{y}_k) - \nabla u_{\widetilde{R}_k} (\bar{y}_k)\cdot (y_k - \bar{y}_k) <\frac{1}{k}.
    \end{equation}
    Let $L_k$ be the segment between $y_k$ and $\bar{y}_k$.
    Since $y_k \in P_M$, we can take a subsequence of $y_k$, converging to a point $y_0$. We can also assume that $\frac{\bar{y}_k- y_k}{|\bar{y}_k - y_k|}$, up to a subsequence, converge to a vector $v$. Using the Lemma \ref{lem ur con}, we can take a subsequence of $u_{\widetilde{R}_k}$, still denoted as $u_{\widetilde{R}_k}$, converging to a function $u$. We claim that $u$ is linear along the ray $L \triangleq \{y_0 + t v : t \ge 0\}$. Take any $\hat{y}, y^*, y' \in L$, where $y'$ lies on the ray shooting from $\hat{y}$ to $y^*$. Let $\lambda \in (0,1)$ be the constant such that $y^* = \lambda \hat{y} + (1-\lambda) y'$. There exist sequences $\hat{y}_k, y_k^*, y_k' \in L_k$ such that $\lim_{k \rightarrow \infty} \hat{y}_k=\hat{y}$, $ \lim_{k \rightarrow \infty} y_k^* = y^*$ and $\lim_{k \rightarrow \infty} y'_k = y'$. Since by (\ref{e yk big}) $\bar{y}_k$ goes to $\infty$ as $k$ goes to $\infty$, the relative position between $y_k, \hat{y}_k, y_k^*, y_k', \bar{y}_k$ can be expressed by the following picture:
    
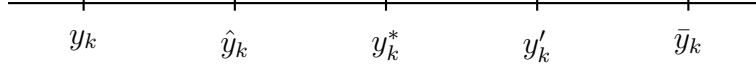
\begin{center}
\begin{tikzpicture}
    \draw[thick] (0,0) -- (10,0);
    \foreach \x/\label in {
        1/ y_k,
        3/{\hat{y}_k},
        5/{y_k^*},
        7/{y_k'},
        9/{\bar{y}_k}
    } {
        \draw[thick] (\x,-0.1) -- (\x,0.1);
        \node[below=6pt] at (\x,0) {$\label$};
    }
\end{tikzpicture}
\end{center}
    Using the fact that $u_{\widetilde{R}_k}$ is convex and (\ref{e ur conv}), we can get that:
    \begin{equation*}
    \begin{split}
          0 & \le u_{\widetilde{R}_k}(y_k^*)- u_{\widetilde{R}_k}(y'_k)- \nabla u_{\widetilde{R}_k}(y'_k) \cdot (y_k^* - y'_k)  
\le u_{\widetilde{R}_k}(\hat{y}_k) - u_{\widetilde{R}_k}(y'_k) - \nabla u_{\widetilde{R}_k}(y'_k) \cdot (\hat{y}_k - y'_k)\\
          &\le  u_{\widetilde{R}_k}(y_k) - u_{\widetilde{R}_k}(\bar{y}_k) - \nabla u_{\widetilde{R}_k} (\bar{y}_k)\cdot (y_k - \bar{y}_k) \le \frac{1}{k}.
    \end{split}
    \end{equation*}
    Then, we can compute that:
    \begin{equation*}
    \begin{split}
        & u_{\widetilde{R}_k}(y^*_k) - \lambda u_{\widetilde{R}_k} (\hat{y}_k)- (1-\lambda) u_{\widetilde{R}_k} (y'_k) \\
        & = u_{\widetilde{R}_k}(y_k^*)- u_{\widetilde{R}_k}(y'_k)- \nabla u_{\widetilde{R}_k} (y'_k)\cdot (y_k^* - y'_k) - \lambda \Big( u_{\widetilde{R}_k}(\hat{y}_k) - u_{\widetilde{R}_k}(y'_k) - \nabla u_{\widetilde{R}_k}(y'_k) \cdot (\hat{y}_k - y'_k) \Big) \le \frac{1}{k}
    \end{split}
    \end{equation*}
    and
    \begin{equation*}
    \begin{split}
    & u_{\widetilde{R}_k}(y^*_k) - \lambda u_{\widetilde{R}_k} (\hat{y}_k)- (1-\lambda) u_{\widetilde{R}_k} (y'_k) \\
        & = u_{\widetilde{R}_k}(y_k^*)- u_{\widetilde{R}_k}(y'_k)- \nabla u_{\widetilde{R}_k} (y'_k)\cdot (y_k^* - y'_k) - \lambda \Big( u_{\widetilde{R}_k}(\hat{y}_k) - u_{\widetilde{R}_k}(y'_k) - \nabla u_{\widetilde{R}_k}(y'_k) \cdot (\hat{y}_k - y'_k) \Big) \ge -\lambda \frac{1}{k}.
    \end{split}
    \end{equation*}
    Let $k$ goes to infinity, we can use the Corollary \ref{cor hol con} to get that:
    \begin{equation*}
        u(y^*)-\lambda u(\hat{y}) -(1-\lambda)u(y')=0.
    \end{equation*}
    As a result, $u$ is linear along the ray $L$. By adding an affine function to $u$, we can assume that $u=0$ along $L$. Take an arbitrary point $p \in \inte P$ and take a small constant $\epsilon>0$ such that $p-\epsilon v \in \inte P$. Denote $y_t = y_0+ tv$. We can choose constants $\epsilon_t$ such that $p_t \triangleq p+ \epsilon_t (p-y_t)$ satisfies:
    \begin{equation*}
        \lim_{t \rightarrow \infty} p_t = p-\epsilon v.
    \end{equation*}
    Using the convexity of $u$, we can get that:
    \begin{equation}\label{e con u}
        u(p)\le \frac{1}{1+\epsilon_t} u(p_t)+ \frac{\epsilon_t}{1+\epsilon_t} u(y_t)= \frac{1}{1+\epsilon_t} u(p_t).
    \end{equation}
    It is clear from the definition of $\epsilon_t$ that $\lim_{t \rightarrow \infty}\epsilon_t=0$. As a result, we can let $t\rightarrow \infty$ in (\ref{e con u}) to get:
    \begin{equation*}
        u(p)\le u(p-\epsilon v).
    \end{equation*}
    Since we have the freedom to choose $p$ and $\epsilon$, the above formula implies that:
    \begin{equation*}
        \nabla_v u \le 0
    \end{equation*}
    on $P$. This contradicts with the fact that $\nabla u (P)=\R^n$.
\end{proof}

The above proposition means that when we do the Pogorelov estimate for $\varphi_R$ in a section, we can assume that $\nabla \varphi_R$ is uniformly bounded in this section.
\begin{prop}\label{prop rn c2}
For any $M$, there exist constants $R_M$ and $C(M)>0$ such that for any $R \ge R_M$ and $y\in P_M$,
\[
  (D^2 u_R)(y) \ge C(M)I.
\]
\end{prop}

\begin{proof}
Fix $\rho_0 \in \nabla \varphi_R^{-1} (P_M)$ and normalize $\varphi_R$ by setting
\begin{equation}
  \widetilde{\varphi}_R(\rho)
  =\varphi_R(\rho)-\varphi_R(\rho_0)
   -\nabla\varphi_R(\rho_0)\cdot(\rho-\rho_0).
\end{equation}
Since $\varphi_R$ is strictly convex, the section
\[
  S_{\epsilon_M}=\{\rho\in\R^n:\widetilde{\varphi}_R(\rho)<\epsilon_M\}
\]
is compact, and $\nabla\widetilde{\varphi}_R$ is uniformly bounded on $S_{\epsilon_M}$. Moreover,
\begin{equation}\label{eq:tilde-varphi-R-logdet}
\begin{split}
   \log\det D^2\widetilde{\varphi}_R
  &=\log\det D^2\varphi_R \\
  &=-\widetilde{\varphi}_R(\rho)-\varphi_R(\rho_0)
    -\nabla\varphi_R(\rho_0)\cdot(\rho-\rho_0)
    +\Bigl\langle\xi_R, \nabla\widetilde{\varphi}_R(\rho)+\nabla\varphi_R(\rho_0)
     \Bigr\rangle+C\\
&=:F(\rho,\widetilde{\varphi}_R,\nabla\widetilde{\varphi}_R).  
\end{split}
\end{equation}

Let
\[
  t=\frac12|\nabla\widetilde{\varphi}_R|^2,
  \qquad
  H(\rho,\zeta)=(\epsilon_M-\widetilde{\varphi}_R)\eta(t)
  \widetilde{\varphi}_{R,\zeta\zeta},
\]
where $\eta>0$ will be chosen later. Suppose that
\[
  H(\bar\rho, \bar \zeta)=\max_{(\rho, \zeta) \in S_{\epsilon_M} \times \mathbb S^{n-1}}H(\rho, \zeta).
\]
After rotating coordinates, we may also assume that
$(\widetilde{\varphi}_{R,ij}(\bar\rho))$ is diagonal and $\bar \zeta = e_1$.
Denote 
\begin{equation}
    G(\rho)= (\epsilon_M-\widetilde{\varphi}_R)\eta(t)
  \widetilde{\varphi}_{R,11}.
\end{equation}
At $\bar\rho$,
\begin{align}
  0=(\log G)_i
  &=\frac{\widetilde{\varphi}_{R,i}}{\widetilde{\varphi}_R-\epsilon_M}
    +\frac{\eta_i}{\eta}
    +\frac{\widetilde{\varphi}_{R,11i}}{\widetilde{\varphi}_{R,11}},
  \label{eq:first-G}\\
  0\geq(\log G)_{ii}
  &=\frac{\widetilde{\varphi}_{R,ii}}{\widetilde{\varphi}_R-\epsilon_M}
    -\frac{\widetilde{\varphi}_{R,i}^2}{(\widetilde{\varphi}_R-\epsilon_M)^2}
    +\frac{\eta_{ii}}{\eta}-\frac{\eta_i^2}{\eta^2}
    +\frac{\widetilde{\varphi}_{R,11ii}}{\widetilde{\varphi}_{R,11}}
    -\frac{\widetilde{\varphi}_{R,11i}^2}{\widetilde{\varphi}_{R,11}^2}.
  \label{eq:second-G}
\end{align}

Differentiating \eqref{eq:tilde-varphi-R-logdet} gives
\begin{equation}
  \widetilde{\varphi}_R^{ij}\widetilde{\varphi}_{R,ijk}
  =-\widetilde{\varphi}_{R,k}-\varphi_{R,k}(\rho_0)
   +\xi_{R,i}\widetilde{\varphi}_{R,ik},
  \label{eq:logdet-first}
\end{equation}
and differentiating once more yields
\begin{equation}
\widetilde{\varphi}_R^{ij}\widetilde{\varphi}_{R,ijkk}
  =\widetilde{\varphi}_R^{ia}\widetilde{\varphi}_R^{bj}
    \widetilde{\varphi}_{R,ijk}\widetilde{\varphi}_{R,abk}
   -\widetilde{\varphi}_{R,kk}
   +\xi_{R,i}\widetilde{\varphi}_{R,ikk}.
  \label{eq:logdet-second}
\end{equation}
At the diagonal point $\bar\rho$, taking $k=1$ in
\eqref{eq:logdet-second} gives
\begin{equation}  \label{eq:hash2}
\frac{\widetilde{\varphi}_R^{ii}\widetilde{\varphi}_{R,ii11}}
       {\widetilde{\varphi}_{R,11}}
  =\widetilde{\varphi}_R^{ii}\widetilde{\varphi}_R^{jj}\widetilde{\varphi}_R^{11}
    \widetilde{\varphi}_{R,ij1}^{\,2}
   -1+\frac{\xi_{R,i}\widetilde{\varphi}_{R,i11}}
            {\widetilde{\varphi}_{R,11}}.
\end{equation}

Contracting \eqref{eq:second-G} with $\widetilde{\varphi}_R^{ii}$, we obtain
\begin{align}
  0\geq\widetilde{\varphi}_R^{ii}(\log G)_{ii}
  ={}&\frac{n}{\widetilde{\varphi}_R-\epsilon_M}
    -\frac{\widetilde{\varphi}_R^{ii}\widetilde{\varphi}_{R,i}^2}
          {(\widetilde{\varphi}_R-\epsilon_M)^2}
    +\frac{\eta_{ii}\widetilde{\varphi}_R^{ii}}{\eta}
    -\frac{\eta_i^2\widetilde{\varphi}_R^{ii}}{\eta^2}+\frac{\widetilde{\varphi}_R^{ii}\widetilde{\varphi}_{R,ii11}}
          {\widetilde{\varphi}_{R,11}}
    -\frac{\widetilde{\varphi}_R^{ii}\widetilde{\varphi}_{R,i11}^{\,2}}
          {\widetilde{\varphi}_{R,11}^2}.
  \label{eq:contracted-G}
\end{align}
Assume that
\[
  \widetilde{\varphi}_{R,11}(\epsilon_M-\widetilde{\varphi}_R)\geq1;
\]
otherwise the desired bound is immediate. For $i\geq2$, equation
\eqref{eq:first-G} gives
\begin{equation}\label{eq:hash1}
    \begin{split}
-\sum_{i=2}^n
\frac{\widetilde{\varphi}_R^{ii}\widetilde{\varphi}_{R,i}^2}
       {(\widetilde{\varphi}_R-\epsilon_M)^2}
 &=-\sum_{i=2}^n\widetilde{\varphi}_R^{ii}
   \left(\frac{\eta_i}{\eta}
        +\frac{\widetilde{\varphi}_{R,i11}}{\widetilde{\varphi}_{R,11}}\right)^2\\
 &= -\sum_{i=2}^n\widetilde{\varphi}_R^{ii} (\frac{\eta_i^2}{\eta^2}+ 2 \frac{\eta_i \widetilde{\varphi}_{R,i11}}{\eta \widetilde{\varphi}_{R,11}} + \frac{\widetilde{\varphi}_{R,i11}^2}{\widetilde{\varphi}_{R,11}^2})\\ 
& =  -\sum_{i=2}^n\widetilde{\varphi}_R^{ii} (\frac{\eta_i^2}{\eta^2}- 2 \frac{\eta_i^2}{\eta^2} -\frac{2\eta_i \widetilde{\varphi}_{R,i}}{\eta (\widetilde{\varphi_R}-\epsilon_M)} + \frac{\widetilde{\varphi}_{R,i11}^2}{\widetilde{\varphi}_{R,11}^2})\\
&=\sum_{i=2}^n\frac{\widetilde{\varphi}_R^{ii}\eta_i^2}{\eta^2}
   +2\sum_{i=2}^n
      \frac{\widetilde{\varphi}_R^{ii}\eta_i}{\eta}
      \frac{\widetilde{\varphi}_{R,i}}{\widetilde{\varphi}_R-\epsilon_M}
   -\sum_{i=2}^n \frac{\widetilde{\varphi}_R^{ii}\widetilde{\varphi}_{R,i11}^{\,2}}
           {\widetilde{\varphi}_{R,11}^2}.
    \end{split}
\end{equation}

Combining \eqref{eq:hash1} and \eqref{eq:hash2}, and using the positive
terms with $i\neq j$ in \eqref{eq:hash2}, gives
\begin{align}
  & \sum_{i=1}^n \frac{\widetilde{\varphi}_R^{ii}\widetilde{\varphi}_{R,ii11}}
       {\widetilde{\varphi}_{R,11}} -\sum_{i=2}^n
\frac{\widetilde{\varphi}_R^{ii}\widetilde{\varphi}_{R,i}^2}
       {(\widetilde{\varphi}_R-\epsilon_M)^2}
  -\sum_{i=1}^n \frac{\widetilde{\varphi}_R^{ii}\widetilde{\varphi}_{R,i11}^{\,2}}
        {\widetilde{\varphi}_{R,11}^2}
  \notag\\
  &\qquad\geq
  \sum_{i=2}^n\frac{\widetilde{\varphi}_R^{ii}\eta_i^2}{\eta^2}
  +2\sum_{i=2}^n
       \frac{\widetilde{\varphi}_R^{ii}\eta_i}{\eta}
       \frac{\widetilde{\varphi}_{R,i}}{\widetilde{\varphi}_R-\epsilon_M}
  -1+\sum_{i=1}^n\frac{\xi_{R,i}\widetilde{\varphi}_{R,i11}}
                         {\widetilde{\varphi}_{R,11}}.
  \label{eq:hash-combined}
\end{align}

At the diagonal point,
\begin{align}
  \eta_i&=\eta'\widetilde{\varphi}_{R,k}\widetilde{\varphi}_{R,ki}
         =\eta'\widetilde{\varphi}_{R,i}\widetilde{\varphi}_{R,ii},
  \label{eq:eta-first}\\
  \eta_{ii}
  &=\eta''\widetilde{\varphi}_{R,i}^2\widetilde{\varphi}_{R,ii}^2
    +\eta'\widetilde{\varphi}_{R,ii}^2
    +\eta'\widetilde{\varphi}_{R,k}\widetilde{\varphi}_{R,kii}.
  \label{eq:eta-second}
\end{align}
Substitution into \eqref{eq:contracted-G} gives
\begin{align}
  0\geq{}&\frac{n}{\widetilde{\varphi}_R-\epsilon_M}
   +\frac{\widetilde{\varphi}_{R,1}^2}{\widetilde{\varphi}_R-\epsilon_M}
   +\sum_{i=1}^n\left\{
       \frac{\widetilde{\varphi}_R^{ii}}{\eta}
       \left(
          \eta''\widetilde{\varphi}_{R,i}^2\widetilde{\varphi}_{R,ii}^2
         +\eta'\widetilde{\varphi}_{R,ii}^2
         +\eta'\widetilde{\varphi}_{R,k}\widetilde{\varphi}_{R,kii}
       \right)
       +\frac{\xi_{R,i}\widetilde{\varphi}_{R,i11}}
              {\widetilde{\varphi}_{R,11}}
     \right\}
  \notag\\
  &+2\sum_{i=2}^n
       \frac{\widetilde{\varphi}_R^{ii}\eta'
             \widetilde{\varphi}_{R,i}\widetilde{\varphi}_{R,ii}}{\eta}
       \frac{\widetilde{\varphi}_{R,i}}{\widetilde{\varphi}_R-\epsilon_M}
   -1-\frac{(\eta')^2}{\eta^2}
       \widetilde{\varphi}_{R,1}^2\widetilde{\varphi}_{R,11}^2
       \widetilde{\varphi}_R^{11}.
  \label{eq:before-eta-choice}
\end{align}
Using \eqref{eq:logdet-first} and \eqref{eq:first-G}, the terms containing
third derivatives satisfy
\begin{align}
 &\sum_{i=1}^n\left( \sum_{k=1}^n
    \frac{\widetilde{\varphi}_R^{ii}}{\eta}\eta'
      \widetilde{\varphi}_{R,k}\widetilde{\varphi}_{R,kii}
    +\frac{\xi_{R,i}\widetilde{\varphi}_{R,i11}}
           {\widetilde{\varphi}_{R,11}}
  \right)
 \notag\\
 &=  \sum_{k=1}^n \frac{\eta'}{\eta}\widetilde{\varphi}_{R,k} (-\widetilde{\varphi}_{R,k} - \varphi_{R,k} (\rho_0)+ \xi_{R,k} \widetilde{\varphi}_{R,kk})+ \sum_{i=1}^n \frac{\xi_{R,i} \widetilde{\varphi}_{R,i11}}{\widetilde{\varphi}_{R,11}} \\
 & = \sum_{k=1}^n - \frac{\eta'}{\eta} \widetilde{\varphi}_{R,k}^2 -\frac{\eta'}{\eta} \widetilde{\varphi}_{R,k} \varphi_{R,k}(\rho_0) + \frac{\xi_{R,k} \eta_k}{\eta} + \frac{\xi_{R,k} \widetilde{\varphi}_{R,k11}}{\widetilde{\varphi}_{R,11}}\\
 & =
 -\frac{\eta'}{\eta}|\nabla\widetilde{\varphi}_R|^2
 -\sum_{k=1}^n\frac{\eta'}{\eta}\widetilde{\varphi}_{R,k} \varphi_{R,k}(\rho_0)
 -\sum_{k=1}^n \frac{\xi_{R,k}\widetilde{\varphi}_{R,k}}
        {\widetilde{\varphi}_R-\epsilon_M}.
 \label{eq:third-order-cancel}
\end{align}

Choose
\[
  \eta(t)=e^{\alpha t},
  \qquad \frac{\eta'}{\eta}=\alpha.
\]
Then \eqref{eq:before-eta-choice} becomes
\begin{align}
  0\geq{}&\frac{n}{\widetilde{\varphi}_R-\epsilon_M}
   +\frac{\widetilde{\varphi}_{R,1}^2}{\widetilde{\varphi}_R-\epsilon_M}
   +\alpha^2\sum_i\widetilde{\varphi}_{R,ii}\widetilde{\varphi}_{R,i}^2
   +\alpha\sum_i\widetilde{\varphi}_{R,ii}
   -\alpha\sum_i\widetilde{\varphi}_{R,i}^2
   -\alpha\sum_i\widetilde{\varphi}_{R,i} \varphi_{R,i}(\rho_0)
  \notag\\
  &-\sum_{k=1}^n \frac{\xi_{R,k}\widetilde{\varphi}_{R,k}}
           {\widetilde{\varphi}_R-\epsilon_M}
   +2\alpha\sum_{i=2}^n
       \frac{\widetilde{\varphi}_{R,i}^2}{\widetilde{\varphi}_R-\epsilon_M}
   -1-\alpha^2\widetilde{\varphi}_{R,1}^2\widetilde{\varphi}_{R,11}.
  \label{eq:alpha-inequality}
\end{align}
For $\alpha>0$ sufficiently small such that $\alpha \widetilde{\varphi}_{R,1}^2 \le \frac{1}{2}$, this yields
\begin{equation}
  \frac{\alpha}{2}\widetilde{\varphi}_{R,11}
  \leq \frac{C\bigl(1+\|\nabla\widetilde{\varphi}_R\|_{L^\infty(S_{\epsilon_M})} + \|\nabla\widetilde{\varphi}_R\|_{L^\infty(S_{\epsilon_M})}^2 \bigr)}
              {\epsilon_M-\widetilde{\varphi}_R}.
\end{equation}
According to the Proposition \ref{prop gra loc}, there exists a constant $C'(M)$ such that $$\|\nabla\widetilde{\varphi}_R\|_{L^\infty(S_{\epsilon_M})} \le C'(M).$$ Evaluating at $\rho=\rho_0$ proves $|D^2\varphi_R(\rho_0)|\leq C''(M)$. Since $((u_R)^{ij})= ((\varphi_R)_{ij})$, this concludes the proof of this Lemma.
\end{proof}

\section{The Evans-Krylov Estimate}\label{sec: EKest}

For any $p \in \partial P$, take a coordinate 
\[
  y=(y',z),\qquad
  y'=(y_1,\ldots,y_k),\qquad
  z=(y_{k+1},\ldots,y_n)\in\R^m.
\]
such that locally near $p$, $\partial P$ is given by 
\begin{equation*}
\cup_{i=1}^k \{(y',z): y_i=0\}.
\end{equation*}

Let
$\Omega'\Subset\Omega\subset\R^m$ and let $\delta>0$.  Consider
\[
  \mathcal Q=(0,\delta)^k\times\Omega.
\]
The original polytope coordinates are
\begin{equation}
  \mathbf x(y)
  =(y_1-1,\ldots,y_k-1,y_{k+1},\ldots,y_n).
\end{equation}

Let $U\in C^\infty(\mathcal Q)$ be strictly convex and solve the soliton
Monge-Amp\`ere equation
\begin{equation}
  \det \operatorname{Hess} U
  =\exp\bigl(
      \mathbf x(y)\cdot\nabla U-U+\ell(\mathbf x(y))
    \bigr),
  \qquad \ell(x)=b\cdot x+c,
  \label{eq:soliton-MA1}
\end{equation}
for fixed $b\in\R^n$ and $c\in\R$.  Write
\begin{equation}
  U(y)=\sum_{i=1}^k y_i\log y_i+V(y).
  \label{eq:guillemin-decomposition1}
\end{equation}
Throughout this section, we assume that $U$ satisfies Guillemin boundary condition.
The main Proposition that we will prove is the following Evans-Krylov type estimate:
\begin{prop}\label{prop evans}
  Let $V$, $\mathcal Q$, $y$, $\Omega$ and $\Omega'$ be given as above.  Suppose that there exists $K$ such that 
    \begin{equation}
 |V|\le K,\,\,\, |DV|\leq K,
  \label{eq:gradient-bound}
\end{equation}
\begin{equation}
  K^{-1}\leq 1+y_iV_{ii}\leq K,
  \qquad i=1,\ldots,k,
  \label{eq:singular-diagonal-bounds}
\end{equation}
and
\begin{equation}
  K^{-1}\leq V_{aa}\leq K,
  \qquad a=k+1,\ldots,n,
  \label{eq:tangential-diagonal-bounds}
\end{equation}
in $\mathcal Q$.
Then for any $l\ge 1$, there exists a constant $C(l)$ depending on $K$, $l$, $\Omega$, $\Omega'$, $\delta$,  $|b|$ and $c$ such that:
\begin{equation*}
    ||D^{\alpha} V||_{L^{\infty}\big((0,\frac{\delta}{2} )^k \times \Omega'\big)}\le C(l)
\end{equation*}
for any $\alpha \in (\mathbb{Z}_{\ge 0})^n$ with $|\alpha| \triangleq \sum_{i=1}^n \alpha_i=l$.
\end{prop}

We postpone the proof of Proposition~\ref{prop evans} to the end of this section.

Set
\begin{equation}
S=\operatorname{diag}(\sqrt{y_1},\ldots,\sqrt{y_k},1,\ldots,1),
  \qquad
  M=S(\operatorname{Hess} U)S.
  \label{eq:S-and-M}
\end{equation}
Because $U$ is strictly convex, $M$ is positive definite.  Its diagonal
entries are
\begin{equation}
  M_{ii}=1+y_iV_{ii}\quad(i\leq k),
  \qquad
  M_{aa}=V_{aa}\quad(a>k).
  \label{eq:M-diagonal}
\end{equation}

\begin{lem}\label{lem upp imp low}
    There exists positive constants $\lambda$, $\Lambda$ and $C$ such that:
    \begin{equation}
  \lambda I\leq M\leq\Lambda I
  \label{eq:M-uniform-ellipticity}
\end{equation}
and
\begin{align}
  |\sqrt{y_i y_j}\,V_{ij}|&\leq C,
    &&1\leq i,j\leq k,\ i\neq j,
  \label{eq:mixed-singular}\\
  |\sqrt{y_i}\,V_{ia}|&\leq C,
    &&1\leq i\leq k<a\leq n,
  \label{eq:mixed-singular-tangential}\\
  |V_{ab}|&\leq C,
    &&k<a,b\leq n.
  \label{eq:mixed-tangential}
\end{align}
in $\mathcal Q$.
\end{lem}
\begin{proof}
Multiplying \eqref{eq:soliton-MA1} by $\prod_{i=1}^k y_i$ and using
\eqref{eq:guillemin-decomposition1}, all logarithmic terms cancel.  Thus
\begin{equation}
  \det M
  =\left(\prod_{i=1}^k y_i\right)\det \operatorname{Hess} U
  =e^G,
  \label{eq:det-M}
\end{equation}
where
\begin{equation}
  G=
  \sum_{i=1}^k (y_i-1)(1+V_i)
  +\sum_{a=k+1}^n y_aV_a
  -V+\ell(\mathbf x(y)).
  \label{eq:G}
\end{equation}
Inside $\mathcal Q$, the bound on $DV$, together
with the value of $V$ at one point, bounds $G$.  Consequently,
\begin{equation}
  0<c_0\leq\det M\leq C_0<\infty.
  \label{eq:det-M-bounds}
\end{equation}
On the other hand, \eqref{eq:singular-diagonal-bounds} and
\eqref{eq:tangential-diagonal-bounds} give $\tr M\leq nK$.  Since $M>0$, we can get (\ref{eq:M-uniform-ellipticity})
for some constants $\lambda$ and $\Lambda$. Then, this implies (\ref{eq:mixed-singular}), (\ref{eq:mixed-singular-tangential}) and (\ref{eq:mixed-tangential}).
\end{proof}

Let
\begin{equation}
  p=DU(y),
  \qquad
  \Phi(p)=y\cdot p-U(y)
\end{equation}
be the full Legendre transform.  In the interior,
\begin{equation}
  D\Phi=y,
  \qquad
  \operatorname{Hess}\Phi=(\operatorname{Hess} U)^{-1}.
\end{equation}
For $i\leq k$ and $a>k$,
\begin{equation}
  p_i=\log y_i+1+V_i,
  \qquad
  p_a=V_a.
\end{equation}
Define
\begin{equation}
  \rho_i=e^{p_i}\quad(i\leq k),
  \qquad
  \eta_a=p_a\quad(a>k).
\end{equation}
The gradient bound gives
\begin{equation}
  \frac{y_i}{\rho_i}=e^{-1-V_i},
  \qquad
  c\,y_i\leq\rho_i\leq C\,y_i.
  \label{eq:rho-comparison}
\end{equation}

We need a relative product neighborhood in the dual variables.  Introduce
\begin{equation}
  r_i=2\sqrt{y_i},
  \qquad
  R_i=2\sqrt{\rho_i},
  \label{eq:square-root-coordinates}
\end{equation}
and consider
\begin{equation}
  \mathcal T:(r,z)\longmapsto(R,\eta).
  \label{eq:T-map}
\end{equation}

\begin{lem}
\label{lem:corner-domain}
Suppose that (\ref{eq:gradient-bound}) and (\ref{eq:M-uniform-ellipticity}) hold for $V$ in $\mathcal Q$. Then, for any $z_0 \in \Omega'$ there exists constants $\epsilon_1$ and $\epsilon_2$ depending on $n,k$, $K$, $\lambda$, $\Lambda$, $\delta, \Omega, \Omega'$ such that  $\mathcal T$ is a injective map from $B_{\epsilon_1}(0,z_0)\cap (\cap_{i=1}^k \{r_i \ge 0\})$ to its image and 
\begin{equation}\label{e ima con bal}
    B_{\epsilon_2}(0,\eta_0)\cap ( \cap_{i=1}^k \{R_i \ge 0\} ) \subset  \mathcal T \Big( B_{\epsilon_1}(0,z_0)\cap (\cap_{i=1}^k \{r_i \ge 0\}) \Big).
\end{equation}
Here $(0,\eta_0)= \mathcal T (0,z_0).$
\end{lem}
\begin{proof}
\textbf{Step 1: Estimate for $\mathcal T$.}
    By the definition of $R$, we have that 
    \begin{equation*}
        \mathcal T (r,z)= (r_1 e^{\frac{1}{2}(1+V_{y_1})},\ldots, r_k e^{\frac{1}{2}(1+ V_{y_k})}, D_z V).
    \end{equation*}
    We can compute in the coordinates $(r,z)$, $(R,\eta)$ that:
    \begin{equation*}
    \begin{split}
         D_j \mathcal T^i &= (\delta_{ij}+ \sqrt{y_i y_j} V_{y_i y_j})e^{\frac{1}{2}(1+ V_{y_i})} \\
         D_b \mathcal T^i &= \sqrt{y_i} V_{y_i z_b} e^{\frac{1}{2}(1+ V_{y_i})} \\
         D_j \mathcal T^a &= V_{z_{a} y_j} \sqrt{y_j} \\
         D_b \mathcal T^a &= V_{z_a z_b},
    \end{split}
    \end{equation*}
    for $i,j=1,\ldots,k$ and $a,b = k+1,\ldots, n$. As a result we have that 
    \begin{equation*}
        D \mathcal T= Q M,
    \end{equation*}
    where $Q= \diag(e^{\frac{1}{2}(1+ V_{y_1})},\ldots, e^{\frac{1}{2}(1+ V_{y_k})},1, \ldots, 1)$ and $M$ is defined in (\ref{eq:S-and-M}). Using (\ref{eq:gradient-bound}) and (\ref{eq:M-uniform-ellipticity}), we have that there exists a uniform constant $C_1$ such that:
    \begin{equation}\label{e est Dt Dt1}
        |D \mathcal T |+ |(D \mathcal T)^{-1}|\le C_1
    \end{equation}
    on $\mathcal Q$.

\textbf{Step 2: Injectivity of $\mathcal T$}. Suppose that there are two points $(r^*,z^*),(\widetilde{r},\widetilde{z}) \in \mathcal Q$ such that $\mathcal T(r^*,z^*) = \mathcal T (\widetilde{r},\widetilde{z})$. Define $(R^*, \eta^*)=\mathcal T (r^*,z^*)=\mathcal T (\widetilde{r},\widetilde{z})$. Let $J \subset \{1,\ldots, k\}$ be the set of the indices of zero coordinates of $r^*$, and denote $J=(J_1,\ldots, J_l)$ and $I=(I_1,\ldots, I_{k-l})$ which is the complement of $J$ in $(1,\ldots, k)$. By (\ref{eq:rho-comparison}), $J$ is the same as the set of the indices of zero coordinates of $R^*$, which is the same as the set of the indices of zero coordinates of $\widetilde{r}$. As a result, $(r^*,z^*)$ and $(\widetilde{r},\widetilde{z})$ are two different points in the interior of the face $\cap_{i=\in J}\{r_i=0\}$. By the strict convexity of $U$ in this statum, we have that 
$$(\nabla_{y_{I_1}}U,\ldots, \nabla_{y_{I_{k-l}}}U, \nabla_z U)(r^*,z^*) \neq (\nabla_{y_{I_1}}U,\ldots, \nabla_{y_{I_{k-l}}}U, \nabla_z U)(\widetilde{r},\widetilde{z}).$$
This implies that $\mathcal T(r^*,z^*) \neq \mathcal T (\widetilde{r},\widetilde{z})$. This is a contradiction.

\textbf{Step 3: Proof of (\ref{e ima con bal})}.  Denote $b_0=(0,\eta_0)$. Take $b=(R^*,\eta^*)\in B_{\epsilon_2}(b_0)\cap (\cap_{i=1}^k \{R_i \ge 0\})$, where $\epsilon_2>0$ is to be determined. Next, we solve the ODE:
\begin{equation*}
    \gamma'(t)=(D \mathcal T (\gamma(t)))^{-1}(b-b_0), \,\,\, \gamma(0)=(0,z_0).
\end{equation*}
This implies that $\frac{d}{dt} \mathcal T (\gamma(t))= b-b_0$, which implies that $\mathcal T (\gamma(t))=b_0 + t (b-b_0)\in B_{\epsilon_2}(b_0)\cap (\cap_{i=1}^k \{R_i \ge 0\})$. This implies that 
\begin{equation}\label{e gammat rip}
    \gamma(t)\in  (\cap_{i=1}^k \{r_i \ge 0\}).
\end{equation}

Using (\ref{e est Dt Dt1}), we can get that:
\begin{equation*}
    |\gamma(t)-(0,z_0)|\le t C_1 |b-b_0|\le C_1 \epsilon_2,
\end{equation*}
for $t\in [0,1]$. Then, we can let $\epsilon_2$ small such that $C_1 \epsilon_2 \le \epsilon_1$ so that 
\begin{equation}\label{e gammat epsi1}
    \gamma(t)\in B_{\epsilon_1} (0,z_0).
\end{equation}
Combining (\ref{e gammat rip}) and (\ref{e gammat epsi1}), we can get that:
\begin{equation*}
    \gamma(t)\in B_{\epsilon_1} (0,z_0) \cap (\cap_{i=1}^k \{r_i \ge 0\}).
\end{equation*}
Note that $\gamma(1)= b$. This concludes the proof of (\ref{e ima con bal}). 
\end{proof}

Because
\[
  \mathbf x(y)\cdot p-U
  =\Phi-\sum_{i=1}^k p_i,
\]
equation \eqref{eq:soliton-MA} and Legendre duality imply
\begin{equation}
  \det\operatorname{Hess}\Phi
  =\exp\left(
      -\Phi+\sum_{i=1}^k p_i-\ell(\mathbf x(y))
    \right).
  \label{eq:dual-real-MA}
\end{equation}

Next, we want to do complexification as mentioned in the Section 2.4 above.
Define
\begin{equation}
  F(\rho,\eta)
  =\Phi(\log\rho_1,\ldots,\log\rho_k,\eta).
\end{equation}
Introduce complex coordinates $\zeta=(\zeta_1,\ldots,\zeta_n)\in\C^n$ by
\begin{equation}
  \rho_i=|\zeta_i|^2\quad(i\leq k),
  \qquad
  \eta_a=\zeta_a+\bar\zeta_a\quad(a>k),
\end{equation}
and set
\begin{equation}
  \Psi(\zeta)
  =F\bigl(
       |\zeta_1|^2,\ldots,|\zeta_k|^2,
       \zeta_{k+1}+\bar\zeta_{k+1},\ldots,
       \zeta_n+\bar\zeta_n
     \bigr).
\end{equation}
By Lemma~\ref{lem:corner-domain}, $\Psi$ is initially defined on a complex
product neighborhood away from
\begin{equation}
  \mathcal D=\bigcup_{i=1}^k\{\zeta_i=0\}.
  \label{eq:coordinate-divisor}
\end{equation}

Set
\begin{equation}
  D_\rho=\diag(\sqrt{\rho_1},\ldots,\sqrt{\rho_k},1,\ldots,1),
  \qquad
  Q=D_\rho^{-1}S.
  \label{eq:D-rho-Q}
\end{equation}
We write $\zeta_i = \sqrt{\rho_i}e^{\sqrt{-1}\theta_i}$. Define:
\begin{equation*}
    \mathcal{U}_{\theta}= \diag (e^{-\sqrt{-1}\theta_1},\ldots, e^{-\sqrt{-1}\theta_k},1, \ldots, 1)
\end{equation*}
A direct calculation gives
\begin{equation}
  H_\Psi:=(\Psi_{\alpha\bar\beta})
  =\mathcal{U}_{\theta} D_\rho^{-1}(\Hess\Phi)D_\rho^{-1} \mathcal{U}_{\theta}^*.
  \label{eq:complex-Hessian-Phi} 
\end{equation}
Since
\[
  \Hess\Phi=(\Hess U)^{-1}=SM^{-1}S,
\]
we obtain
\begin{equation}
  H_\Psi=\mathcal{U}_{\theta}QM^{-1}Q \mathcal{U}_{\theta}^*.
  \label{eq:complex-Hessian-M}
\end{equation}
Equations \eqref{eq:M-uniform-ellipticity} and
\eqref{eq:rho-comparison} therefore give
\begin{equation}
  \lambda_1I\leq H_\Psi\leq\Lambda_1I.
  \label{eq:complex-uniform-ellipticity}
\end{equation}

Furthermore,
\begin{equation}
  \zeta_i\Psi_{\zeta_i}=y_i\quad(i\leq k),
  \qquad
  \Psi_{\zeta_a}=y_a\quad(a>k).
  \label{eq:Psi-first-derivative-identities}
\end{equation}
Dividing \eqref{eq:dual-real-MA} by
$\rho_1\cdots\rho_k=e^{p_1+\cdots+p_k}$ yields
\begin{equation}
  \log\det H_\Psi
  =-\Psi-\ell\bigl(\mathbf x(\zeta,D\Psi)\bigr),
  \label{eq:complex-MA}
\end{equation}
where
\begin{equation}
  \mathbf x_i(\zeta,D\Psi)=\zeta_i\Psi_{\zeta_i}-1
  \quad(i\leq k),
  \qquad
  \mathbf x_a(\zeta,D\Psi)=\Psi_{\zeta_a}
  \quad(a>k).
  \label{eq:x-complex}
\end{equation}
For $\ell(x)=b\cdot x+c$, equation \eqref{eq:complex-MA} is explicitly
\begin{equation}
  \log\det H_\Psi
  =-\Psi
   -\sum_{i=1}^k b_i(\zeta_i\Psi_{\zeta_i}-1)
   -\sum_{a=k+1}^n b_a\Psi_{\zeta_a}-c.
  \label{eq:complex-MA-affine}
\end{equation}

Denote $U_{\epsilon} \triangleq \{(\zeta_1,\ldots \zeta_n): |\zeta_i|< \frac{\epsilon}{2} \text{ for } i \le k,  (Re \zeta_{k+1}, \ldots, Re \zeta_n) \in B_{\epsilon}(\eta_0)\}$. We choose $\epsilon_3$ small enough such that: $$  U_{\epsilon_3} \subset B_{\epsilon_2}(0,\eta_0)\cap ( \cap_{i=1}^k \{R_i \ge 0\} ) $$
\begin{lem}\label{lem c11 psi}
Let $\Psi$, $\epsilon_3$, $\eta_0$ be given above. Then there exists a constant $C$ such that:
\begin{equation*}
    ||\Psi||_{C^{1,1}(U_{\epsilon_3})} \le C.
\end{equation*}
\end{lem}
\begin{proof}
From
\eqref{eq:Psi-first-derivative-identities} and
\eqref{eq:rho-comparison},
\begin{equation}
  |\Psi_{\zeta_i}|
  =\frac{y_i}{|\zeta_i|}
  \leq C|\zeta_i|,
  \qquad i\leq k.
  \label{eq:Psi-first-bound}
\end{equation}
The pure holomorphic second derivatives satisfy, for $i\neq j\leq k$ and
$a,b>k$,
\begin{align}
  \Psi_{\zeta_i\zeta_i}
    &=\frac{\Phi_{p_ip_i}-y_i}{\zeta_i^2},
  \label{eq:pure-ii}\\
  \Psi_{\zeta_i\zeta_j}
    &=\frac{\Phi_{p_ip_j}}{\zeta_i\zeta_j},
  \label{eq:pure-ij}\\
  \Psi_{\zeta_i\zeta_a}
    &=\frac{\Phi_{p_i\eta_a}}{\zeta_i},
  \label{eq:pure-ia}\\
  \Psi_{\zeta_a\zeta_b}
    &=\Phi_{\eta_a\eta_b}.
  \label{eq:pure-ab}
\end{align}
The absolute values of the right-hand sides are bounded by
\eqref{eq:complex-uniform-ellipticity} and
\eqref{eq:rho-comparison}.  Together with the mixed complex derivatives
controlled by \eqref{eq:complex-uniform-ellipticity}, this proves
\begin{equation}
  \|D_{\R}^{2}\Psi\|_{L^\infty}\leq C
  \qquad\text{on the complement of }\mathcal D.
  \label{eq:real-Hessian-bound}
\end{equation}

By the Proposition \ref{prop comp ext 2}, we can extend $\Psi$ smoothly across $\mathcal D$. As a result, (\ref{eq:real-Hessian-bound}) holds on $\mathcal D$ as well.
\end{proof}

\begin{prop}\label{prop psi est}
Let $\Psi$, $\epsilon_2$, $\eta_0$ be given above. Then, for any $k \ge 1$, there exists a constant $C(k)$ such that for any $\alpha \in (\mathbb{Z}_{\ge 0})^n$ with $|\alpha|=k$, we have that:
\begin{equation*}
    ||D^{\alpha} \Psi||_{L^{\infty}(U_{\frac{\epsilon_3}{2}})}\le C(k).
\end{equation*}
\end{prop}
\begin{proof}
    The $C^{2,\alpha}$ estimate for $\Psi$ follows from the Lemma \ref{lem c11 psi} and the Evans-Krylov estimate adapted to complex equations, see \cite{TosattiWangWeinkoveYang2015}. The higher estimates follow by standard bootstrapping arguments.
\end{proof}

We need to use the following Lemma by Whitney \cite{Whitney}:

\begin{lem}\label{lem whi}
For any even, infinitely differentiable function \(f(x)\) on all real \(x\), there exists an infinitely differentiable function \(g(y)\) for \(y \ge 0\) such that \(f(x) = g(x^2)\).
\end{lem}

Next, we can prove the Proposition \ref{prop evans}. 
\begin{proof}
of the Proposition \ref{prop evans}. We will use the estimate for $\Psi$ which is obtained in the Proposition \ref{prop psi est} to get the estimate for $V$.

The function $\Psi$ is invariant under independent rotations of
$\zeta_1,\ldots,\zeta_k$ and under translations of
$\operatorname{Im}\zeta_{k+1},\ldots,\operatorname{Im}\zeta_n$.  These
symmetries hold away from $\mathcal D$ by construction and hence everywhere by
continuity.  The Lemma \ref{lem whi}, applied successively to
$\zeta_1,\ldots,\zeta_k$, therefore shows that the function $F$ in
\eqref{eq:F-definition} extends smoothly to
\begin{equation}
  F\in C^\infty\bigl([0,\frac{\eps_3}{2})^k\times B_\eps(\eta_0)\bigr).
  \label{eq:F-smooth}
\end{equation}

The Legendre identities become
\begin{equation}
  y_i=\rho_iF_{\rho_i}\quad(i\leq k),
  \qquad
  y_a=F_{\eta_a}\quad(a>k).
  \label{eq:F-legendre-identities}
\end{equation}
Before extension,
\begin{equation}
  F_{\rho_i}=\frac{y_i}{\rho_i}=e^{-1-V_i}.
  \label{eq:F-rho-positive}
\end{equation}
Thus the smooth extension satisfies
\begin{equation}
  0<c\leq F_{\rho_i}\leq C
  \qquad\text{at }\rho=0.
  \label{eq:F-rho-bounds}
\end{equation}
Moreover, the lower-right principal block of
\eqref{eq:complex-uniform-ellipticity} gives
\begin{equation}
  D^2_{\eta\eta}F\geq\lambda_1I.
  \label{eq:F-eta-positive}
\end{equation}

At $\rho=0$, the Jacobian of the map
\[
  (\rho,\eta)\longmapsto(y',z)
  =\bigl(
      \rho_1F_{\rho_1},\ldots,\rho_kF_{\rho_k},D_\eta F
    \bigr)
\]
has the block form
\begin{equation}
  \begin{pmatrix}
    \diag(F_{\rho_1},\ldots,F_{\rho_k})&0\\
    D^2_{\eta\rho}F&D^2_{\eta\eta}F
  \end{pmatrix}.
  \label{eq:inverse-Jacobian}
\end{equation}
It is invertible by \eqref{eq:F-rho-bounds} and
\eqref{eq:F-eta-positive}.  The smooth inverse function theorem therefore
gives smooth functions
\begin{equation}
  \rho=\rho(y',z),
  \qquad
  \eta=\eta(y',z)
  \label{eq:smooth-inverse-variables}
\end{equation}
up to $y_1=\cdots=y_k=0$.

The inverse Legendre formula is
\begin{equation}
  U=\sum_{i=1}^k y_i\log\rho_i+z\cdot\eta-F.
  \label{eq:inverse-Legendre}
\end{equation}
Subtracting the Guillemin terms and using
$y_i=\rho_iF_{\rho_i}$ gives
\begin{align}
  V
  &=\sum_{i=1}^k y_i\log\frac{\rho_i}{y_i}
    +z\cdot\eta-F
  \notag\\
  &=-\sum_{i=1}^k y_i\log F_{\rho_i}
    +z\cdot\eta-F.
  \label{eq:V-smooth-formula}
\end{align}
Every term on the right-hand side is smooth by
\eqref{eq:F-smooth}, \eqref{eq:F-rho-bounds}, and
\eqref{eq:smooth-inverse-variables}.  Moreover the estimate for any derivatives of $V$ can be obtained from the Proposition \ref{prop psi est}.
\end{proof}

\section{Uniform Guillemin Boundary Conditions and the Main Existence Result}\label{sec: Guillemin}

For any $p \in \partial P$, we can take an affine transformation of the original coordinate to get a new coordinate system: 
\[
  y=(y',z),\qquad
  y'=(y_1,\ldots,y_k),\qquad
  z=(y_{k+1},\ldots,y_n)\in\R^m.
\]
such that locally near $p$, $\partial P$ is given by 
\begin{equation*}
\cup_{i=1}^k \{(y',z): y_i=0\}.
\end{equation*}
Without loss of generality, we can assume that $p$ is at origin.
Let $r>0$.  Consider
\[
  \mathcal Q=(0,r)^k\times (-r,r)^{n-k}.
\]
The original polytope coordinates are
\begin{equation}
  \mathbf x(y)
  =(y_1-1,\ldots,y_k-1,y_{k+1},\ldots,y_n).
\end{equation}

Let $u_R\in C^\infty(\mathcal Q)$ be strictly convex and solve the soliton
Monge--Amp\`ere equation
\begin{equation}
  \det \operatorname{Hess} u_R
  =\exp\bigl(
      \mathbf x(y)\cdot\nabla u_R-u_R+\ell(\mathbf x(y))
    \bigr),
  \qquad \ell(x)= - \langle \xi_R, y \rangle +C_R,
  \label{eq:soliton-MA}
\end{equation}
for fixed $b\in\R^n$ and $c\in\R$.  Write
\begin{equation}
  u_R(y)=\sum_{i=1}^k y_i\log y_i+v_R(y).
\end{equation}

\begin{defn}
    We say that $\{u_R\}$ satisfies uniform Guillemin boundary condition around a point $x_0 \in \partial P$, if the function $v_R$ given by (\ref{eq:guillemin-decomposition}) satisfies: For any $k \ge 0$, there exist constants $C(k)$ and $r>0$ depending on $x_0$ and $k$ such that for any $\alpha \in \mathbb{Z}_{\ge 0}^n$ with $|\alpha|=k$,
    \begin{equation*}
        ||D^{\alpha} v_R||_{L^{\infty}(\mathcal Q)} \le C(k).
    \end{equation*}
\end{defn}

In this section, we want to prove that the uniform Guillemin boundary condition holds around any point in $\partial P$. This allows us to take a convergence subsequence of $u_R$ such that its limit $u$ also satisfies the Guillemin boundary condition on $P$.

The strategy is to do induction on the codimension of the face of $P$. We start with the estimate on codimension 1 faces. After we prove the uniform Guillemin boundary condition around the interior of such faces, we prove the uniform $C^{1,1}$ estimates around the interior of codimension two faces. Using Evans-Krylov type estimate, we prove the uniform Guillemin boundary condition around the interior of codimension two faces. Keeping doing such induction, we can prove the uniform Guillemin boundary condition around any face of $P$.

The following Proposition is the key ingredient for doing induction on the codimension.
\begin{prop}\label{prop hig dim c2}
    Use the same notation as the beginning of this section. Suppose that $u_R \ge 0$ and $u_R$ satisfies Guillemin boundary condition. Suppose that there is a uniform constant $C_1$ such that:
    \begin{equation}\label{e ur bd}
        ||u_R||_{L^{\infty}(2\mathcal Q)}\le C_1.
    \end{equation}
    Suppose that $u_R$ satisfies the uniform Guillemin boundary condition on $\cup_{i=1}^k \{y_i \neq 0\}$. Suppose that there exists a constant $\Lambda$ such that:
    \begin{equation}\label{e tan c2 bdd}
        u_R^{jj}\le \Lambda
    \end{equation}
    for $j\ge k+1$. Then, there exists a uniform constant $C$ such that:
    \begin{equation*}
        \sum_{i=1}^n e^{u_{R,i}} u_{R,ii} \le C
    \end{equation*}
    in $\mathcal Q$ for some $r>0$.
\end{prop}

\begin{proof}
For simplicity, we just write $u_R$ as $u$ and write $\xi_R$ as $\xi$ and write $v_R$ as $v$. 

\textbf{Step 1: Estimate on normal derivatives.}  Fix a small constant $r>0$. For any $i \le k$, and $y_i \in (0,r)$, we can use the convexity to get that:
    \begin{equation*}
    \begin{split}
u_{i}(y',z) &\le \frac{u(y_1,\ldots, y_{i-1}, 2r, y_{i+1},\ldots y_k, z)- u(y_1,\ldots, y_{i-1}, y_i, y_{i+1},\ldots y_k, z)}{2r-y_i} \\
&\le \frac{u(y_1,\ldots, y_{i-1}, 2r, y_{i+1},\ldots y_k, z)}{r}.
    \end{split}
    \end{equation*}
    Here we use the assumption that $u \ge 0$. By the Lemma \ref{lem: uRholderbounds}, there exists a uniform constant $C$ such that $u(y_1,\ldots, y_{i-1}, 2r, y_{i+1},\ldots y_k, z) \le C$. As a result, we have that:
    \begin{equation}\label{e upper ui}
        u_{i} \le C
    \end{equation}
    for any $i \le k$ on $\mathcal Q$. In particular, we have that:
    \begin{equation*}
        \sum_{i=1}^k e^{u_i}\le C.
    \end{equation*}
     Using the convexity of $u$ and (\ref{e ur bd}), we can get that for any $i \ge k+1$, 
    \begin{equation}\label{e ui tan}
        ||u_i||_{L^{\infty}(\mathcal Q)} \le C,
    \end{equation}
    for some uniform constant $C$.\\
    
    \textbf{Step 2: Linearized Inequalities.}
 Define
\[
L\psi=u^{ij}\psi_{ij}-x_k\psi_k.
\]
Taking logarithms in (\ref{eq:soliton-MA}) and differentiate it once in the $x_p$-direction gives
\begin{equation}\label{e lup}
L(u_p)=-\xi_p.
\end{equation}
Differentiating again gives
\begin{equation}\label{e lupp}
L(u_{pp})
=
u^{ia}u^{jb}u_{ijp}u_{abp}+u_{pp}.
\end{equation}
Equation (\ref{e lupp}) yields
\[
L(\log u_{pp})
=
1+
\frac{u^{ia}u^{jb}u_{ijp}u_{abp}}{u_{pp}}
-
\frac{u^{ij}u_{ppi}u_{ppj}}{u_{pp}^2}.
\]
The tensor Cauchy inequality gives
\begin{equation}
u^{ij}u_{ppi}u_{ppj}
\leq
u_{pp}\,u^{ia}u^{jb}u_{ijp}u_{abp}.
\end{equation}
Indeed, with the metric $D^2u$, the left-hand side is the squared
norm of the one-form obtained by inserting $e_p$ into the symmetric
two-tensor $(u_{ijp})$, while the right-hand side is
$|e_p|_{D^2u}^2$ times the squared norm of that tensor. This proves
\begin{equation}\label{e l log upp}
L(\log u_{pp})\geq 1.
\end{equation}
Combining (\ref{e lup}) and (\ref{e l log upp}), we proves:
\begin{equation}\label{e l log eu upp}
L\bigl(\log(e^{ u_p}u_{pp})\bigr)\geq 1 - \xi_p\ge -C_2.
\end{equation}
    Denote $Y= e^{u_p}u_{pp}$ and $E= \exp(u_p + x_p)$.
    \begin{equation}\label{e le}
    \begin{split}
         L E &= E \Big( L( u_p) +  L x_p + u^{ij}(u_p + x_p)_i (u_p + x_p)_j \Big) \\
         & = E \Big( -\xi_p -x_p +  u^{pp}+  u_{pp} + 2\Big) \\
         & \ge E u_{pp}- CE  \ge C_1 Y- C_2 E.
    \end{split}
    \end{equation}
    For any $i \ge k+1$, we define $\eta_i (z_i)= 1-\frac{z_i^2}{r^2}$. We can compute that:
    \begin{equation*}
        (\log \eta_i)'= -\frac{2z_i}{r^2 -z_i^2}, \,\,\, (\log \eta_i)''= -\frac{2(r^2 + z_i^2)}{(r^2- z_i^2)^2}.
    \end{equation*}
    Using (\ref{e tan c2 bdd}), we can get that:
    \begin{equation}\label{e llog eta}
        L (2\log \eta_i) \ge -\frac{C}{\eta_i^2}.
    \end{equation}
    \textbf{Step 3: Maximum principle.} For $\epsilon \in (0,1)$, we define 
    \begin{equation*}
        \Phi_{\epsilon}= \log Y + E + \epsilon (\sum_{i=1}^k u_i)+ 2\sum_{i=k+1}^n \log \eta_i.
    \end{equation*}
    If $p \le k$, then   $u_p= \log y_p + 1 + v_p$, $u_{pp}=\frac{1}{y_p}+ v_{pp}$. Thus, we have that:
    \begin{equation*}
        Y= e^{1+ v_p}(1+ y_p v_{pp}).
    \end{equation*}
    Since $u$ satisfies the Guillemin boundary condition, we have that $v$ is smooth.  As a result, $\log Y$ is bounded on $\mathcal Q$. If $p \ge k$, then $u_p=v_p$ is smooth, which implies that $\log Y$ is bounded on $\mathcal Q$. This bound is not uniform, because we don't have uniform Guillemin boundary condition on $\cap_{i=1}^k \{y_i=0\}$ yet.
    
    The Guillemin boundary condition of $u$ also implies that $\epsilon u_i$ goes to $-\infty$, as $y_i \rightarrow 0$. By the definition of $\eta_i$, we have that $2\log \eta_i$ goes to $-\infty$, as $y_i \rightarrow \pm r$, for $i \ge k+1$. It follows that either the maximum of $\Phi_{\epsilon}$ is attained in the interior, or is attained at some point in $\partial \mathcal Q$ which has a uniform distance away from $\{ y_1= \ldots, = y_k=0 \}$ where
    its maximum value is already uniformly bounded using the assumption that $u_R$ satisfies the uniform Guillemin boundary condition on $\cup_{i=1}^k \{y_i \neq 0\}$.

    Suppose that the maximum is attained at an interior point $p$. At that point, $L \Phi_{\epsilon} (p)\le 0$. By (\ref{e lup}), (\ref{e l log eu upp}), (\ref{e le}) and (\ref{e llog eta}), we can get that:
    \begin{equation*}
        0 \ge L \Phi_{\epsilon} (p) \ge -C_2 + C_1 Y +\epsilon (-\sum_{i=1}^k \xi_i) -C \sum_{i=k+1}^n \frac{1}{\eta_i^2}.
    \end{equation*}

    This implies that 
    \begin{equation}\label{e eta y upp}
        \Pi_{i=k+1}^n \eta_i^2 Y \le C
    \end{equation}
    for some constant $C$ at the maximum point. By (\ref{e upper ui}), we know that $E$ and the positive part of $\epsilon \sum_{i=1}^k u_i$ are uniformly bounded. Combining this with (\ref{e eta y upp}), we get that:
    \begin{equation}\label{e phi eps upp}
        \max \Phi_{\epsilon}\le C
    \end{equation}
    for a uniform constant $C$.
    Fix a point with $|z|< \frac{r}{2}$. Since $\eta_i \ge 3/4$, (\ref{e phi eps upp}) implies that:
    \begin{equation*}
        \log Y + E + \epsilon \sum_{i=1}^k u_i \le C.
    \end{equation*}
    Dropping $E \ge 0$ and letting $\epsilon \rightarrow 0$, we can get that:
    \begin{equation*}
        e^{u_p}u_{pp}\le C
    \end{equation*}
    for a uniform constant $C$.
\end{proof}

\begin{lem}\label{lem upp low c2}
   Assume the same assumption as in the Proposition \ref{prop hig dim c2}. Then, there exists a constant $K$ depending on $C$ such that:
    \begin{equation*}
        K^{-1}\le 1+ y_i v_{R,ii}\le K,
    \end{equation*}
 for any $i \le k$    and
    \begin{equation*}
        K^{-1}\le v_{R,ii}\le K,
    \end{equation*}
    for any $i \ge k+1$.
\end{lem}
\begin{proof}
\textbf{Step 1: Gradient estimate}
Using the Proposition \ref{prop hig dim c2},we can get that:
\begin{equation*}
    (e^{u_{R,i}})_i = e^{u_{R,i}} u_{R,ii} \le C.
\end{equation*}
Here we use the uniform upper bound on $u_{R,i}$ given by (\ref{e upper ui}). By the Guillemin boundary condition, we have that $u_{R,i}$ goes to $-\infty$ as $x_i$ goes to $-1$. As a result, we can integrate the above inequality in $x_i$ direction from $-1$ to $x_i$ to get:
\begin{equation*}
    e^{u_{R,i}}\le C (x_i+1).
\end{equation*}
This implies that:
\begin{equation}\label{e uri upp}
    u_{R,i}\le \log C + \log (x_i+1),
\end{equation}
for $i \le k$.

By the Lemma \ref{lem: uRholderbounds}, $u_R$ is uniformly bounded in $B_{\epsilon}(y_0)$. Using the convexity of $u_R$ and shrink $\epsilon$ a bit, we also have that 
\begin{equation*}
    |u_{R,i}|\le C_5
\end{equation*}
in $B_{r}(0)$ for some $C_5$ and for $i\ge k+1$. As a result, there exists a constant $C_6$ such that:
\begin{equation}\label{e bdd tan 1}
  |z\cdot\nabla_{z}u_R-u_R-\langle\xi_R,x\rangle-C| \le C_6.
\end{equation}
is bounded.  
Next, we estimate $ u_{R,n}$ in the other direction. We can use (\ref{e bdd tan 1}) and the Hadamard's inequality to have that:
\begin{equation}\label{e det lower}
    u_{R,11}\cdot \ldots \cdot u_{R,nn}  \ge det (u_{R,ij}) = \exp(-\langle \xi_R, x \rangle + \langle x, \nabla u_R \rangle -u_R) \ge C_7 \exp(x' \cdot \nabla_{x'} u_{R})
\end{equation}
Using (\ref{e bdd tan 1}), (\ref{e det lower}) and the upper bound of the tangential second order derivatives of $u_R$ by the Proposition \ref{prop hig dim c2}, we can get that:
\begin{equation}\label{e urnn low}
    \Pi_{l=1}^k u_{R,ll}\ge C exp(x' \cdot \nabla_{x'} u_{R}).
\end{equation}

For $i\le k $, by the convexity of $u_R$ and the boundedness of $u_R$, we can get:
\begin{equation}\label{e 1x1 ui}
   (-1-x_i) u_{R,i}(x_1,\ldots, x_{n})\le u_R(x_1,\ldots, x_{i-1}, -1, x_{i+1},\ldots, x_n) - u_R (x_1,\ldots, x_{n})\le C.
\end{equation}
Using the Proposition \ref{prop hig dim c2}, we can get that $u_{R,ii}\le C e^{-u_{R,i}}$ for $i \le k$. Combining this with (\ref{e 1x1 ui}), we can get that:
\begin{equation}\label{e uii xiui}
    u_{R,ii}\le C e^{x_i u_{R,i}}.
\end{equation}
For any $i \le k$, we use (\ref{e uii xiui}) to cancel $u_{R,ll}$ in (\ref{e urnn low}) except for $l=i$ to get 
\begin{equation*}
    u_{R,ii} \ge C \exp(x_i u_{R,i}).
\end{equation*}

Since in $\mathcal Q$ we have that $x_i <0$ and $-x_i \ge C$, we can use the above formula to get:
\begin{equation*}
    -x_i u_{R,ii}- u_{R,i}\ge C \exp(x_i u_{R,i}).
\end{equation*}
This implies that:
\begin{equation*}
    \Big( \exp(-x_i u_{R,i}) \Big)_i \ge C.
\end{equation*}
Integrating the above formula and use the fact that $-x_i u_{R,i}$ goes to $-\infty$ as $x_i$ goes to $-1$ which is ensured by the Guillemin boundary condition, we can get:
\begin{equation*}
    \exp(-x_i u_{R,i})\ge C(x_i+1).
\end{equation*}
This implies that 
\begin{equation}
    -x_i u_{R,i} \ge \log (x_i+1)+ \log C,
\end{equation}
which implies that
\begin{equation}\label{e low 1 nor}
   u_{R,i}\ge \frac{\log (x_i+1)}{-x_i}-C =\log (x_i+1) + \frac{(x_i+1)\log (x_i+1)}{-x_i}-C \ge \log(x_i+1)-C'.
\end{equation}

\textbf{Step 2: Hessian estimate} 
Combining (\ref{e low 1 nor}) and the Proposition \ref{prop hig dim c2}, we can get that:
\begin{equation*}
    u_{R,ii}\le C e^{x_i u_{R,i}} = C e^{x_i \big( \log (x_i+1)+ C \big)}= C e^{(x_i+1)\log (x_i+1) + C x_i - \log (x_i+1)} \le \frac{C}{x_i+1}.
\end{equation*}
By (\ref{e uri upp}) and (\ref{e low 1 nor}), we know that 
\begin{equation*}
  \frac{1}{C \Pi_{i=1}^k y_i}  e^{x \cdot \nabla u_R - u_R -\langle \xi_R, x \rangle -C}\le \frac{C}{\Pi_{i=1}^k y_i}
\end{equation*}
Then, we can apply the proof of Lemma \ref{lem upp imp low} to conclude the proof of this proposition.
\end{proof}

\begin{thm}\label{thm uni gui}
   The family $u_R$ satisfies the uniform Guillemin boundary condition around any point $p\in \partial P$.
\end{thm}
\begin{proof}
The Lemma \ref{lem: uRholderbounds} provides the local uniform $C^0$ estimate for $u_R$. Proposition \ref{prop rn c2} allows us to apply the Lemma \ref{lem upp low c2} around any point in the interior of a codimension one facet of $P$. Note that although the Lemma \ref{lem upp low c2} requires locally $u_R \ge 0$. This can be ensured by adding $u_R$ by a uniform constant, using the local uniform $C^0$ estimate for $u_R$.  Then, we can use the Proposition \ref{prop evans} to prove that $u_R$ satisfy the uniform Guillemin boundary condition around any point in the interior of a codimension one facet of $P$. Then, we can use the Lemma \ref{lem upp low c2} and the Proposition \ref{prop evans} to prove that $u_R$ satisfy the uniform Guillemin boundary condition around any point in the interior of a codimension two facet of $P$. Using an induction argument by applying the Lemma \ref{lem upp low c2} and Proposition \ref{prop evans} again and again, we can prove that $u_R$ satisfy the uniform Guillemin boundary condition around any point in $\partial P$.
\end{proof}

\subsection{The Main Existence Result}
We now prove the main existence result, Theorem~\ref{main thm 1}.

\begin{proof}[Proof of the Theorem \ref{main thm 1}]
    Using the Lemma \ref{lem ur con}, we know that we can take a sequence $u_{R_k}$ such that it converges in $C_{loc}^{k,\alpha}$ to $u_{\infty}$ for any $k$ with $u_{\infty}$ to be a maximum point of $I_{\mu}$, solving
    $$ det (u_{ij}) = e^{\langle y, \nabla u (y) \rangle -u(y) -\langle \xi, y \rangle + C_{\infty}}.$$
    By \cite{CEK}, we know that $u_{\infty}^*$ is a translation of $\varphi$ obtained in the Corollary \ref{cor: momentMeasureKRS}, up to adding a constant. According to the Corollary \ref{cor fin varphi}, $\nabla u_{\infty}= \R^n$. As a result, $u_{\infty}$ satisfies (\ref{e main equ}). By the Theorem \ref{thm uni gui}, $u_{R}$ satisfies uniform Guillemin boundary condition around any point on $\partial P$. As a result, $u_{\infty}$ also satisfies Guillemin boundary condition around any point on $\partial P$. This concludes the proof of the Theorem \ref{main thm 1}.
\end{proof}

\section{Geometric Properties of the  toric shrinking K\"ahler-Ricci soliton}\label{sec: volumeGrowth}

\begin{lem}
\label{lem:polyhedral-volume-asymptotics}
Let $P\subset\mathbb R^n$ be a nonempty, full-dimensional,
closed convex polytope. Set
\[
C=\operatorname{rec}(P),\qquad
W=\operatorname{span}C,\qquad
\gamma=\dim W.
\]
Suppose that $C$ is pointed and that
$\xi\in(\mathbb R^n)^*$ satisfies
\[
\langle\xi,c\rangle>0
\qquad
\text{for every }c\in C\setminus\{0\}.
\]
Choose Lebesgue measures compatible with the orthogonal
decomposition
\[
\mathbb R^n=W\oplus W^\perp,
\qquad dx=dw\,dz.
\]
Define
\[
K=\operatorname{pr}_{W^\perp}(P),
\qquad
C_\xi=\{w\in C:\langle\xi,w\rangle\leq1\},
\]
and
\[
A(P,\xi)
=
\operatorname{Vol}_{n-\gamma}(K)
\operatorname{Vol}_{\gamma}(C_\xi).
\]
The volumes on $W$ and $W^\perp$ are intrinsic volumes,
with the convention $\operatorname{Vol}_0(\{0\})=1$.

Then $0<A(P,\xi)<\infty$, and
\begin{equation}
\label{eq:polyhedral-volume-asymptotics}
\operatorname{Vol}_n
\{x\in P:\langle\xi,x\rangle\leq t\}
=
A(P,\xi)t^\gamma+o(t^\gamma)
\qquad\text{as }t\to\infty.
\end{equation}
\end{lem}

\begin{proof}
By the Minkowski--Weyl theorem, there exists a compact
convex polytope $Q$ such that
\begin{equation}
\label{eq:polytope-minkowski-decomposition}
P=Q+C.
\end{equation}
If $\gamma=0$, then $C=\{0\}$ and $P=Q$ is compact.
In this case $K=P$ and $C_\xi=\{0\}$, so
\eqref{eq:polyhedral-volume-asymptotics} follows because
the sublevel set equals $P$ for all sufficiently large $t$.
We henceforth assume $\gamma\geq1$.

Since $C\subset W$, we have
\[
K=\operatorname{pr}_{W^\perp}(Q).
\]
Thus $K$ is compact. Moreover, projecting an open ball
contained in $P$ shows that $K$ has nonempty interior in
$W^\perp$. In particular, its intrinsic volume is positive.

The positivity assumption on $\xi$ and compactness of
$C\cap\{w\in W:|w|=1\}$ imply that
\[
\alpha
:=
\min_{\substack{c\in C\\ |c|=1}}
\langle\xi,c\rangle
>0.
\]
Consequently,
\begin{equation}
\label{eq:cone-linear-control}
\langle\xi,c\rangle\geq\alpha|c|
\qquad(c\in C).
\end{equation}
It follows that $C_\xi$ is compact. Since $C$ has
nonempty interior relative to $W$, the set $C_\xi$
has positive $\gamma$-dimensional volume. This proves
$0<A(P,\xi)<\infty$.

For $t>0$, write
\[
P_t=\{x\in P:\langle\xi,x\rangle\leq t\},
\]
and introduce the anisotropic dilation
\[
L_t:W\times W^\perp\longrightarrow\mathbb R^n,
\qquad
L_t(w,z)=tw+z.
\]
Its Jacobian is $t^\gamma$. Define
\[
\Omega_t=L_t^{-1}(P_t)
=
\left\{
(w,z):
tw+z\in P,\quad
\langle\xi,w\rangle+t^{-1}\langle\xi,z\rangle\leq1
\right\}.
\]
Then
\begin{equation}
\label{eq:rescaled-polyhedral-volume}
t^{-\gamma}\operatorname{Vol}_n(P_t)
=
\int_{W^\perp}\int_W
\mathbf{1}_{\Omega_t}(w,z)\,dw\,dz.
\end{equation}

We first show that
\begin{equation}
\label{eq:rescaled-indicator-convergence}
\mathbf{1}_{\Omega_t}(w,z)
\longrightarrow
\mathbf{1}_{C_\xi}(w)\mathbf{1}_K(z)
\qquad\text{for almost every }(w,z).
\end{equation}
If $z\notin K$, then $tw+z\notin P$ for every $t$,
so the assertion is immediate.

Fix $z\in K$. Choose $q\in Q$ with
$\operatorname{pr}_{W^\perp}(q)=z$, and write $q=q_W+z$.
Suppose that
\[
w\in\operatorname{int}_W C,
\qquad
\langle\xi,w\rangle<1.
\]
For all sufficiently large $t$,
\[
w-t^{-1}q_W\in C,
\]
and hence $tw-q_W\in C$. Therefore
\[
tw+z=q+(tw-q_W)\in Q+C=P.
\]
Also,
\[
\langle\xi,w\rangle+t^{-1}\langle\xi,z\rangle<1
\]
for all sufficiently large $t$. Thus
$(w,z)\in\Omega_t$ eventually.

If $w\notin C$ and $(w,z)\in\Omega_{t_j}$ along a sequence
$t_j\to\infty$, then
\[
t_jw+z=q_j+c_j,
\qquad q_j\in Q,\quad c_j\in C.
\]
Projecting onto $W$ gives
\[
w=t_j^{-1}(q_j)_W+t_j^{-1}c_j.
\]
Since $Q$ is compact, $t_j^{-1}c_j\to w$.
But $t_j^{-1}c_j\in C$ and $C$ is closed, which contradicts
$w\notin C$. Thus $(w,z)\notin\Omega_t$ eventually.

Finally, if $\langle\xi,w\rangle>1$, the height condition
defining $\Omega_t$ fails for all sufficiently large $t$.
The exceptional points are therefore contained in
\[
\left(
\partial_W C
\cup
\{w\in W:\langle\xi,w\rangle=1\}
\right)\times K.
\]
Both $\partial_W C$ and the displayed hyperplane have
zero $\gamma$-dimensional measure. This proves
\eqref{eq:rescaled-indicator-convergence}.

To justify dominated convergence, we construct a common
bounded container for $\Omega_t$. Let
\[
D=\max_{q\in Q}|q|,
\qquad
B=\max_{q\in Q}|\langle\xi,q\rangle|.
\]
If $(w,z)\in\Omega_t$, write
$tw+z=q+c$ with $q\in Q$ and $c\in C$. Then
\[
\langle\xi,c\rangle
=
\langle\xi,tw+z\rangle-\langle\xi,q\rangle
\leq t+B.
\]
By \eqref{eq:cone-linear-control},
\[
|c|\leq\frac{t+B}{\alpha}.
\]
Consequently,
\[
|w|
\leq
\frac{|q_W|+|c|}{t}
\leq
\frac Dt+\frac{t+B}{\alpha t}.
\]
For $t\geq1$, the right-hand side is bounded independently
of $t$. Since $z\in K$, there is $R_0>0$ such that
\[
\Omega_t\subset B_W(R_0)\times K
\qquad(t\geq1).
\]
The latter set has finite product measure. Applying
dominated convergence to
\eqref{eq:rescaled-polyhedral-volume}, we obtain
\[
\begin{aligned}
\lim_{t\to\infty}t^{-\gamma}\operatorname{Vol}_n(P_t)
&=
\int_{W^\perp}\int_W
\mathbf{1}_{C_\xi}(w)\mathbf{1}_K(z)\,dw\,dz\\
&=
\operatorname{Vol}_{n-\gamma}(K)
\operatorname{Vol}_{\gamma}(C_\xi),
\end{aligned}
\]
as required.
\end{proof}

The next theorem establishes a general volume growth estimate for toric gradient shrinking K\"ahler-Ricci solitons.

\begin{thm}
\label{thm: volumeGrowth}
Let $(X^{2n},g,\omega,f)$ be a connected complete
non-compact toric gradient shrinking Ricci soliton with
proper moment map
\[
\mu:X\longrightarrow P.
\]
Suppose
\[
\operatorname{Ric}(g)+\nabla^2f=\lambda g,
\qquad
f=a\langle\xi,\mu\rangle+b,
\qquad
\lambda>0,\quad a>0.
\]
Assume that $P$ and $\xi$ satisfy the hypotheses of
Lemma~\ref{lem:polyhedral-volume-asymptotics}, with
$\gamma=\dim\operatorname{rec}(P)\geq1$.

Normalize the action-angle coordinates by
\[
\omega=\sum_{j=1}^n dx_j\wedge d\theta_j,
\qquad
\theta_j\in\mathbb R/(2\pi\mathbb Z),
\]
and use lattice-normalized Lebesgue measure $dx$ on $P$,
together with compatible measures on $W$ and $W^\perp$.

Then, for every fixed $p\in X$,
\begin{equation}
\label{eq:toric-geodesic-volume-asymptotics}
\operatorname{Vol}_g B_g(p,r)
=
(2\pi)^n A(P,\xi)
\left(\frac{\lambda}{2a}\right)^\gamma
r^{2\gamma}
+o(r^{2\gamma})
\qquad\text{as }r\to\infty.
\end{equation}
\end{thm}

\begin{proof}
On the dense torus orbit, the Riemannian volume form is
\[
dV_g=\frac{\omega^n}{n!}
=
dx_1\cdots dx_n\,d\theta_1\cdots d\theta_n.
\]
The complement of the dense orbit has zero Riemannian
volume. Integrating over the angular variables gives
\[
\operatorname{Vol}_g
\{q\in X:\langle\xi,\mu(q)\rangle\leq t\}
=
(2\pi)^n
\operatorname{Vol}_n
\{x\in P:\langle\xi,x\rangle\leq t\}.
\]
By Lemma~\ref{lem:polyhedral-volume-asymptotics},
\begin{equation}
\label{eq:moment-sublevel-volume}
\operatorname{Vol}_g
\{q\in X:\langle\xi,\mu(q)\rangle\leq t\}
\sim
(2\pi)^n A(P,\xi)t^\gamma.
\end{equation}

We next compare the soliton potential with distance.
Set $\bar g=2\lambda g$. Since constant rescaling leaves
the Levi-Civita connection and the Ricci tensor as a
$(0,2)$-tensor unchanged,
\[
\operatorname{Ric}(\bar g)+\nabla_{\bar g}^2f
=\frac12\bar g,
\qquad
d_{\bar g}=\sqrt{2\lambda}\,d_g.
\]
The potential-distance estimate of Cao--Zhou
\cite[Theorem~1.1]{CaoZhou2010}, with an additive
normalization of the potential, therefore implies
\[
f(q)=\frac{\lambda}{2}d_g(p,q)^2
+O\bigl(d_g(p,q)+1\bigr).
\]
Choose a constant $c$ such that $F=f+c>0$ everywhere.
The same asymptotic estimate holds for $F$.
Factoring the difference of squares shows that
\[
\left|
\sqrt{\frac{2F(q)}{\lambda}}-d_g(p,q)
\right|
\]
is bounded outside a fixed ball. It is also bounded on
that ball by continuity and compactness.
Thus there exists $L>0$ such that
\begin{equation}
\label{eq:potential-radius-comparison}
\left|
\sqrt{\frac{2F(q)}{\lambda}}-d_g(p,q)
\right|
\leq L
\qquad(q\in X).
\end{equation}

Write
\[
F=a\langle\xi,\mu\rangle+\widehat b,
\qquad
\widehat b=b+c,
\]
and define
\[
V_F(s)=\operatorname{Vol}_g\{q\in X:F(q)<s\}.
\]
It follows from \eqref{eq:moment-sublevel-volume} that
\begin{equation}
\label{eq:potential-sublevel-asymptotics}
V_F(s)
\sim
(2\pi)^n A(P,\xi)
\left(\frac{s-\widehat b}{a}\right)^\gamma
\sim
(2\pi)^n A(P,\xi)a^{-\gamma}s^\gamma.
\end{equation}

For $r>L$, \eqref{eq:potential-radius-comparison} yields
\[
\left\{
F<\frac{\lambda}{2}(r-L)^2
\right\}
\subset
B_g(p,r)
\subset
\left\{
F<\frac{\lambda}{2}(r+L)^2
\right\}.
\]
Consequently,
\[
V_F\!\left(\frac{\lambda}{2}(r-L)^2\right)
\leq
\operatorname{Vol}_g B_g(p,r)
\leq
V_F\!\left(\frac{\lambda}{2}(r+L)^2\right).
\]
Applying \eqref{eq:potential-sublevel-asymptotics},
dividing by $r^{2\gamma}$, and using
\[
\frac{(r\pm L)^{2\gamma}}{r^{2\gamma}}\longrightarrow1,
\]
we obtain \eqref{eq:toric-geodesic-volume-asymptotics}
by the squeeze theorem.
\end{proof}

Finally, we can give the proof of Theorem~\ref{thm main 2}

\begin{proof}[Proof of Theorem~\ref{thm main 2}]
By Theorem~\ref{main thm 1} we have a solution of the real Monge-Amp\`ere equation
\[
        \det (D^2u) = e^{\langle y, \nabla u (y) \rangle -u(y) -\langle \xi, y \rangle }, \quad \text{ and }\quad  \nabla u(P)= \R^n.
\]
which additionally satisfies the Guillemin boundary conditions.  By passing to the limit in Proposition~\ref{prop: linGrowthApprox} we obtain the estimate
\[
u^{\xi\xi}\leq C(\langle \xi, y \rangle +C)
\]
and hence by Proposition~\ref{prop: linGrowthImplesComplete} we deduce that $u$ defines a complete shrinking K\"ahler-Ricci soliton.    By Theorem~\ref{thm: volumeGrowth} we obtain the desired volume asymptotics.
\end{proof}

\section{Zero Barycenter and Stability}\label{sec sta rel}

Let
\[
P=\bigl\{
y\in\mathbb{R}^n:
\langle \nu_j,y\rangle\geq -1,\quad j=1,\ldots,d
\bigr\}
\]
be a full-dimensional rational polytope.  We assume that the $\nu_j\in\mathbb{Z}^n$ are primitive inward normals, and that the set of normals is minimal in the sense that no normal $\nu_j$ is contained in the convex hull of $\nu_i$ for $i\ne j$. In particular, we note that $0\in\operatorname{int}P$. Let $C=\operatorname{rec}(P)$, and suppose that $\xi\in\mathbb{R}^n$ satisfies
\[
\langle \xi,c\rangle>0
\qquad\text{for every }c\in C\setminus\{0\}.
\]
When $P$ is bounded, this condition is vacuous. Define
\begin{equation}\label{eq: vwDefintionforKstab}
v(y)=e^{-\langle\xi,y\rangle},
\qquad
w(y)=2\bigl(n-\langle\xi,y\rangle\bigr)v(y).
\end{equation}
On the facet
$F_j=\{y\in P:\langle\nu_j,y\rangle=-1\}$, set
\[
d\sigma=\frac{dS}{|\nu_j|},
\]
where $dS$ denotes Euclidean hypersurface measure.

For every piecewise-linear, convex function
\[
f(y)=\max_{1\leq a\leq m}
\bigl(\langle p_a,y\rangle+b_a\bigr)
\]
on $P$, define
\[
\mathcal{F}_{v,w}(f)
=
2\int_{\partial P}fv\,d\sigma
-
\int_P fw\,dy.
\]
For Delzant $P$, condition~{\rm (2)} in the Lemma \ref{lem:weighted-barycenter-stability} below is the
$(v,w)$-K-stability condition for the polytope in the sense of Lahdili \cite{Lahdili}, as explained by Cifarelli \cite{Cifarelli2024}.

\begin{defn}[Cifarelli \cite{Cifarelli2024}]
For Delzant $P$, we say that $P$ is $(v,w)$-K-stable if
\[
\mathcal{F}_{v,w}(f)\geq 0
\]
for every finite, convex, piecewise-linear function $f$.
\end{defn}

\begin{lem}
\label{lem:weighted-barycenter-stability}
The following conditions are equivalent:
\begin{enumerate}
\item The weighted barycenter vanishes:
\[
\int_P y\,v(y)\,dy=0.
\]
\item $P$ is $(v,w)$-K-stable for $v,w$ in~\eqref{eq: vwDefintionforKstab}.
\end{enumerate}
\end{lem}

\begin{proof}
We first justify the integrals and integration by parts.
By the Minkowski--Weyl theorem, we may write $P=Q+C$,
where $Q$ is compact. The assumption on $\xi$ therefore
implies that there exist constants $a>0$ and $A$ such that
\[
\langle\xi,y\rangle\geq a|y|-A
\qquad\text{for all }y\in P.
\]
Consequently,
\[
v(y)\leq e^A e^{-a|y|}.
\]
Every finite convex piecewise-linear function has at most
linear growth and bounded gradient almost everywhere.
Thus all the integrals below converge absolutely.

Since $f$ is locally Lipschitz, the following identity holds
almost everywhere, and hence distributionally:
\[
\operatorname{div}(yfv)
=
v\langle y,Df\rangle
+
\bigl(n-\langle\xi,y\rangle\bigr)fv.
\]
On $F_j$, the outward unit normal is
$n_j=-\nu_j/|\nu_j|$, so
\[
\langle y,n_j\rangle\,dS=d\sigma.
\]
Apply the divergence theorem on $P\cap B_R$. The flux over
the artificial boundary $P\cap\partial B_R$ tends to zero
as $R\to\infty$, since its absolute value is bounded by a
polynomial in $R$ times $e^{-aR}$.
Passing to the limit gives
\[
\int_{\partial P}fv\,d\sigma
=
\int_P
\left[
\langle y,Df\rangle
+
\bigl(n-\langle\xi,y\rangle\bigr)f
\right]v\,dy.
\]
Therefore
\begin{equation}
\label{eq:weighted-futaki-divergence}
\mathcal{F}_{v,w}(f)
=
2\int_P \langle y,Df\rangle v\,dy.
\end{equation}

Assume first that the weighted barycenter vanishes.
For every affine function
$\ell(y)=\langle p,y\rangle+b$, we obtain
\[
\mathcal{F}_{v,w}(\ell)
=
2\left\langle p,\int_P yv\,dy\right\rangle
=0.
\]
Choose a supporting affine function $\ell$ of $f$ at $0$,
and put $g=f-\ell$. Then
\[
g\geq 0,\qquad g(0)=0,
\qquad
\mathcal{F}_{v,w}(g)=\mathcal{F}_{v,w}(f).
\]
At every point where $g$ is differentiable, convexity gives
\[
0=g(0)
\geq g(y)+\langle Dg(y),-y\rangle,
\]
and hence
\[
\langle y,Dg(y)\rangle\geq g(y).
\]
Using \eqref{eq:weighted-futaki-divergence}, we conclude that
\[
\mathcal{F}_{v,w}(f)
=
2\int_P \langle y,Dg\rangle v\,dy
\geq
2\int_P gv\,dy
\geq 0.
\]
If $\mathcal{F}_{v,w}(f)=0$, then $\int_P gv\,dy=0$.
Since $g$ is continuous and nonnegative and $v$ is strictly
positive, it follows that $g\equiv 0$ on $P$.
Thus $f=\ell$ is affine. Conversely, every affine function
has zero functional.

Finally, suppose condition~{\rm (2)} holds.
Applying nonnegativity to the affine functions $y_j$
and $-y_j$ and using
\eqref{eq:weighted-futaki-divergence}, we obtain
\[
\int_P y_jv\,dy=0
\qquad (j=1,\ldots,n).
\]
This proves the weighted barycenter condition.
\end{proof}

Let $(\pi:X\to Y,\xi)$ be a toric polarized Fano fibration,
and let $M_X$ and $N_X$ denote the character and cocharacter
lattices of the full toric torus of $X$. Let
\[
P_X
=
\left\{
y\in M_{X,\mathbb R}:
\langle \nu_j,y\rangle\geq -1,\quad j=1,\ldots,d
\right\}
\]
be its full-dimensional anticanonically normalized
polytope, where the $\nu_j\in N_X$ are primitive inward
facet normals. Write
\[
\mathcal R(P_X)
=
\left\{
\zeta\in N_{X,\mathbb R}:
\langle\zeta,c\rangle>0
\text{ for every }
c\in\operatorname{rec}(P_X)\setminus\{0\}
\right\},
\]
and assume $\xi\in\mathcal R(P_X)$.

\begin{lem}
\label{lem:K-semistability-barycenter}
If $(\pi:X\to Y,\xi)$ is K-semistable in the sense of
Sun--Zhang \cite{SZ}, then
\[
\int_{P_X}y\,e^{-\langle\xi,y\rangle}\,dy=0,
\]
where $dy$ is lattice-normalized Lebesgue measure.
In particular, the same conclusion holds if the
polarized Fano fibration is K-polystable.
\end{lem}

\begin{proof}
Fix $\eta\in N_X$, and let
\[
\lambda_\eta:\mathbb C^*\longrightarrow T_X^{\mathbb C}
\]
be the corresponding one-parameter subgroup. Since $\pi$
is toric, this action descends to $Y$, and it commutes
with $\xi$.

Consider the product fibration
\[
\Pi:X\times\mathbb A^1\longrightarrow Y\times\mathbb A^1,
\qquad
\Pi(x,\tau)=(\pi(x),\tau),
\]
equipped with the action
\[
t\cdot(x,\tau)
=
\bigl(\lambda_\eta(t)x,t\tau\bigr)
\]
and the induced action on $Y\times\mathbb A^1$.
This is a product special test configuration with central
fiber $(\pi:X\to Y,\xi)$.
Replacing $\eta$ by $-\eta$ gives another product special
test configuration; see
\cite[Definitions~5.2--5.3]{SZ}.

With the canonical lift of the torus action to $-K_X$,
the toric weighted-volume function is
\[
\mathbb W(\zeta)
=
\int_{P_X}e^{-\langle\zeta,y\rangle}\,dy,
\qquad
\zeta\in\mathcal R(P_X).
\]
The canonical anticanonical normalization of $P_X$ is
essential in this formula.

Since $\mathcal R(P_X)$ is open, $\xi+s\eta$ remains in
$\mathcal R(P_X)$ for all sufficiently small $|s|$.
The resulting uniform exponential decay justifies
differentiation under the integral sign.
By the definition of the Futaki invariant,
\cite[Definition~5.4 and Lemma~5.7]{SZ}, the product test
configuration above has invariant
\[
\begin{aligned}
\operatorname{Fut}_{\xi}(\eta)
&=
\left.\frac{d}{ds}\right|_{s=0}
\mathbb W(\xi+s\eta)\\
&=
-\int_{P_X}\langle\eta,y\rangle
e^{-\langle\xi,y\rangle}\,dy.
\end{aligned}
\]
In particular,
\[
\operatorname{Fut}_{\xi}(-\eta)
=
-\operatorname{Fut}_{\xi}(\eta).
\]

K-semistability applied to both product test configurations
therefore gives
\[
\operatorname{Fut}_{\xi}(\eta)\geq 0,
\qquad
-\operatorname{Fut}_{\xi}(\eta)\geq 0.
\]
Hence
\[
\left\langle
\eta,\int_{P_X}y\,e^{-\langle\xi,y\rangle}\,dy
\right\rangle
=0
\qquad\text{for every }\eta\in N_X.
\]
Because $N_X$ spans $N_{X,\mathbb R}$ and the pairing
between $N_{X,\mathbb R}$ and $M_{X,\mathbb R}$ is
nondegenerate, the vector integral must vanish.
\end{proof}


\begin{thebibliography}{99}
\bibitem{ADM} S. Alesker, S. Dar and V. Milman, {\em A remarkable measure preserving diffeomorphism between two convex bodies in ${\bf R}^n$}, Geom. Dedicata {\bf 74} (1999), no. 2, 201--212.

\bibitem{ApostolovCalderbank2004} V. Apostolov, D. M. J. Calderbank, P. Gauduchon and C. W. T{\o}nnesen-Friedman, {\em Hamiltonian 2-forms in K\"ahler geometry. II. Global classification}, J. Differential Geom. {\bf 68} (2004), no. 2, 277--345.

\bibitem{AAKM} S. Artstein-Avidan, B. Klartag and V. Milman, {\em The {S}antal\'o{} point of a function, and a functional form of the {S}antal\'o{} inequality}, Mathematika {\bf 51} (2004), no. 1--2, 33--48.

\bibitem{Bam20a} R. Bamler, {\em Entropy and heat kernel bounds on a Ricci flow background}, preprint, arXiv:2008.07093 (2020).

\bibitem{Bam20b} R. Bamler, {\em Compactness theory of the space of super Ricci flows}, preprint, arXiv:2008.09298 (2020).

\bibitem{Bam20c} R. Bamler, {\em Structure theory of non-collapsed limits of Ricci flows}, preprint, arXiv:2009.03243 (2020).

\bibitem{BCCD2024} R. H. Bamler, C. Cifarelli, R. J. Conlon and A. Deruelle, {\em A new complete two-dimensional shrinking gradient K\"ahler--Ricci soliton}, Geom. Funct. Anal. {\bf 34} (2024), no. 2, 377--392.

\bibitem{BermanBerndtsson13} R. J. Berman and B. Berndtsson, {\em Real Monge--Amp\`ere equations and K\"ahler--Ricci solitons on toric log Fano varieties}, Ann. Fac. Sci. Toulouse Math. {\bf 22} (2013), 649--711.

\bibitem{BermanWittNystrom2014} R. J. Berman and D. Witt Nystr\"om, {\em Complex optimal transport and the pluripotential theory of K\"ahler--Ricci solitons}, preprint, arXiv:1401.8264 (2014).



\bibitem{BJ} S. Boucksom and M. Jonsson, {\em On the Yau-Tian-Donaldson conjecture for weighted cscK metrics}, preprint, arXiv:2509.15016

\bibitem{Brenier1991} Y. Brenier, {\em Polar factorization and monotone rearrangement of vector-valued functions}, Comm. Pure Appl. Math. {\bf 44} (1991), no. 4, 375--417.

\bibitem{Caffarelli1990W2p} L. A. Caffarelli, {\em Interior $W^{2,p}$ estimates for solutions of the Monge--Amp\`ere equation}, Ann. of Math. (2) {\bf 131} (1990), no. 1, 135--150.

\bibitem{Caffarelli1991Regularity} L. A. Caffarelli, {\em Some regularity properties of solutions of Monge Amp\`ere equation}, Comm. Pure Appl. Math. {\bf 44} (1991), no. 8--9, 965--969.

\bibitem{Caffarelli1992Mappings} L. A. Caffarelli, {\em The regularity of mappings with a convex potential}, J. Amer. Math. Soc. {\bf 5} (1992), no. 1, 99--104.

\bibitem{Cao1996} H.-D. Cao, {\em Existence of gradient K\"ahler--Ricci solitons}, in {\em Elliptic and Parabolic Methods in Geometry}, A K Peters, 1996, pp. 1--16.

\bibitem{Cao1997} H.-D. Cao, {\em Limits of solutions to the {K}\"ahler-{R}icci flow}, J. Differential Geom. {\bf 45} (1997), no. 2, 257--272.

\bibitem{CaoZhou2010} H.-D. Cao and D. Zhou, {\em On complete gradient shrinking Ricci solitons}, J. Differential Geom. {\bf 85} (2010), no. 2, 175--186.

\bibitem{ChenDonaldsonSun2015I} X. Chen, S. Donaldson and S. Sun, {\em K\"ahler--Einstein metrics on Fano manifolds. I. Approximation of metrics with cone singularities}, J. Amer. Math. Soc. {\bf 28} (2015), no. 1, 183--197.

\bibitem{ChenDonaldsonSun2015II} X. Chen, S. Donaldson and S. Sun, {\em K\"ahler--Einstein metrics on Fano manifolds. II. Limits with cone angle less than $2\pi$}, J. Amer. Math. Soc. {\bf 28} (2015), no. 1, 199--234.

\bibitem{ChenDonaldsonSun2015III} X. Chen, S. Donaldson and S. Sun, {\em K\"ahler--Einstein metrics on Fano manifolds. III. Limits as cone angle approaches $2\pi$ and completion of the main proof}, J. Amer. Math. Soc. {\bf 28} (2015), no. 1, 235--278.



\bibitem{Cifarelli2024} C. Cifarelli, {\em Weighted K-stability for a class of non-compact toric fibrations}, J. Geom. Anal. {\bf 34} (2024), no. 5, Paper No. 120.

\bibitem{CCD2024Aubin} C. Cifarelli, R. J. Conlon and A. Deruelle, {\em An Aubin continuity path for shrinking gradient K\"ahler--Ricci solitons}, J. Reine Angew. Math. {\bf 815} (2024), 229--307.

\bibitem{CCDTypeI2024} C. Cifarelli, R. J. Conlon and A. Deruelle, {\em On finite time Type I singularities of the K\"ahler--Ricci flow on compact K\"ahler surfaces}, J. Eur. Math. Soc. {\bf 28} (2026), no. 2, 463--504.

\bibitem{CifarelliEsparza2025} C. Cifarelli and C. Esparza, {\em K-polystability of asymptotically conical K\"ahler--Ricci shrinkers}, preprint, arXiv:2512.03323 (2025).

\bibitem{CoSz} T. C. Collins and G. Sz\'ekelyhidi, {\em Sasaki-{E}instein metrics and {K}-stability}, Geom. Topol. {\bf 23} (2019), no. 3, 1339--1413.

\bibitem{CoSz1} T. C. Collins and G. Sz\'ekelyhidi,{\em K-semistability for irregular {S}asakian manifolds}, J. Differential Geom. {\bf 109} (2018), no. 1, 81--109.



\bibitem{CDS2024} R. J. Conlon, A. Deruelle and S. Sun, {\em Classification results for expanding and shrinking gradient K\"ahler--Ricci solitons}, Geom. Topol. {\bf 28} (2024), no. 1, 267--351.

\bibitem{CEF} D. Cordero-Erausquin and A. Figalli, {\em Regularity of monotone transport maps between unbounded domains}, Discrete Contin. Dyn. Syst. {\bf 39} (2019), no. 12, 7101--7112.

\bibitem{CEK} D. Cordero-Erausquin and B. Klartag, {\em Moment measures}, J. Funct. Anal. {\bf 268} (2015), no. 12, 3834--3866.

\bibitem{CLS} D. A. Cox, J. B. Little and H. K. Schenck, {\em Toric Varieties}, Graduate Studies in Mathematics, vol. 124, American Mathematical Society, Providence, RI, 2011.

\bibitem{DatarSzekelyhidi2016} V. Datar and G. Sz\'ekelyhidi, {\em K\"ahler--Einstein metrics along the smooth continuity method}, Geom. Funct. Anal. {\bf 26} (2016), 975--1010.

\bibitem{Delzant1988} T. Delzant, {\em Hamiltoniens p\'eriodiques et images convexes de l'application moment}, Bull. Soc. Math. France {\bf 116} (1988), no. 3, 315--339.


\bibitem{Donaldson2008Toric} S. K. Donaldson, {\em K\"ahler geometry on toric manifolds, and some other manifolds with large symmetry}, in {\em Handbook of Geometric Analysis}, No. 1, Adv. Lect. Math. (ALM), vol. 7, Int. Press, Somerville, MA, 2008, pp. 29--75.

\bibitem{DuistermaatPelayo2009} J. J. Duistermaat and A. Pelayo, {\em Reduced phase space and toric variety coordinatizations of Delzant spaces}, Math. Proc. Cambridge Philos. Soc. {\bf 146} (2009), no. 3, 695--718.

\bibitem{EMT2011} J. Enders, R. M\"uller and P. M. Topping, {\em On Type-I singularities in Ricci flow}, Comm. Anal. Geom. {\bf 19} (2011), no. 5, 905--922.

\bibitem{Esparza2025} C. Esparza, {\em Uniqueness of asymptotically conical shrinking gradient K\"ahler--Ricci solitons}, preprint, arXiv:2502.13521 (2025).

\bibitem{FIK2003} M. Feldman, T. Ilmanen and D. Knopf, {\em Rotationally symmetric shrinking and expanding gradient K\"ahler--Ricci solitons}, J. Differential Geom. {\bf 65} (2003), no. 2, 169--209.

\bibitem{Frad} M. Fradelizi, {\em Sections of convex bodies through their centroid}, Arch. Math. (Basel) {\bf 69} (1997), no. 6, 515--522.

\bibitem{Guillemin1994} V. Guillemin, {\em K\"ahler structures on toric varieties}, J. Differential Geom. {\bf 40} (1994), no. 2, 285--309.

\bibitem{HallgrenZhang2026} M. Hallgren and J. Zhang, {\em Singular K\"ahler--Ricci shrinkers and polarized Fano fibrations}, preprint, arXiv:2605.25213v2 (2026).

\bibitem{Hamilton1995} R. S. Hamilton, {\em The formation of singularities in the Ricci flow}, in {\em Surveys in Differential Geometry}, vol. II, International Press, 1995, pp. 7--136.



\bibitem{Huang2023} G. Huang, {\em The Guillemin boundary problem for Monge--Amp\`ere equation in the polygon}, Adv. Math. {\bf 415} (2023), Paper No. 108885, 29 pp.

\bibitem{HuangShen2026} G. Huang and W. Shen, {\em Monge--Amp\`ere equation with Guillemin boundary condition in high dimension}, Comm. Pure Appl. Math. {\bf 79} (2026), no. 7, 1495--1561.

\bibitem{Koiso1990} N. Koiso, {\em On rotationally symmetric Hamilton's equation for K\"ahler--Einstein metrics}, in {\em Recent Topics in Differential and Analytic Geometry}, Adv. Stud. Pure Math. 18-I, 1990, pp. 327--337.

\bibitem{Lahdili} A. Lahdili, {\em K\"ahler metrics with constant weighted scalar curvature and weighted {K}-stability}, Proc. Lond. Math. Soc. (3) {\bf 119} (2019), no. 4, 1065--1114.

\bibitem{Leg} E. Legendre, {\em Toric K\"ahler--Einstein metrics and convex compact polytopes}, J. Geom. Anal. {\bf 26} (2016), no. 1, 399--427.

\bibitem{LiWang2026} Y. Li and B. Wang, {\em On K\"ahler-Ricci shrinker surfaces}, Acta Math. {\bf 236} (2026), no. 1, 1--50.

\bibitem{McCann1995} R. J. McCann, {\em Existence and uniqueness of monotone measure-preserving maps}, Duke Math. J. {\bf 80} (1995), no. 2, 309--323.

\bibitem{Naber} A. Naber, {\em Noncompact shrinking four solitons with nonnegative curvature}, J. Reine Angew. Math. {\bf 645} (2010), 125--153.

\bibitem{Odaka} Y. Odaka, {\em On {S}un-{Z}hang's theory of {F}ano fibrations: weighted volumes, moduli and bubbling {F}ano fibrations}, Pure Appl. Math. Q. {\bf 22} (2026), no. 1, 295--337.

\bibitem{Perelman2002} G. Perelman, {\em The entropy formula for the Ricci flow and its geometric applications}, preprint, arXiv:0211159.

\bibitem{Rubin} D. Rubin, {\em The {M}onge-{A}mp\`ere equation with {G}uillemin boundary conditions}, Calc. Var. Partial Differential Equations {\bf 54} (2015), no. 1, 951--968.

\bibitem{SZ} S. Sun and J. Zhang, {\em K\"ahler--Ricci shrinkers and Fano fibrations}, preprint, arXiv:2410.09661 (2024).

\bibitem{Szekelyhidi} G. Sz\'ekelyhidi, {\em The partial {$C^0$}-estimate along the continuity method}, J. Amer. Math. Soc. {\bf 29} (2016), no. 2, 537--560.

\bibitem{TianZhu2000} G. Tian and X. Zhu, {\em Uniqueness of K\"ahler--Ricci solitons}, Acta Math. {\bf 184} (2000), no. 2, 271--305.

\bibitem{TianZhu02} G. Tian and X. Zhu, {\em A new holomorphic invariant and uniqueness of K\"ahler--Ricci solitons}, Comment. Math. Helv. {\bf 77} (2002), 297--325.

\bibitem{TosattiWangWeinkoveYang2015} V. Tosatti, Y. Wang, B. Weinkove propo and X. Yang, {\em $C^{2,\alpha}$ estimates for nonlinear elliptic equations in complex and almost complex geometry}, Calc. Var. Partial Differential Equations {\bf 54} (2015), no. 1, 431--453.

\bibitem{WangZhu2004} X.-J. Wang and X. Zhu, {\em K\"ahler--Ricci solitons on toric manifolds with positive first Chern class}, Adv. Math. {\bf 188} (2004), no. 1, 87--103.

\bibitem{Whitney} H. Whitney, {\em Differentiable even functions}, Duke Math. J. {\bf 10} (1943), 159--160.

\bibitem{Yau1978} S.-T. Yau, {\em On the Ricci curvature of a compact K\"ahler manifold and the complex Monge--Amp\`ere equation, I}, Comm. Pure Appl. Math. {\bf 31} (1978), no. 3, 339--411.



\end{thebibliography}
\end{document}